\documentclass[a4paper,11pt, twoside]{amsart}
\usepackage[utf8]{inputenc}
\usepackage{amssymb}
\usepackage{amsmath,amsthm}
\usepackage{etoolbox}
\usepackage{mathrsfs}
\usepackage[cal=boondox]{mathalfa}
\usepackage{amscd}
\usepackage{amsfonts}
\usepackage{tikz}
\usepackage{tikz-cd}
\usepackage[shortlabels]{enumitem}
\usepackage{rotating}
\usepackage{MnSymbol}
\usepackage{relsize}
\usepackage{bbm}
\usepackage{tcolorbox}
\usepackage{xcolor}
\usepackage{mathtools}
\usepackage[
backend=biber,
style=alphabetic,
sorting=anyt,
firstinits=true
]{biblatex}
\RequirePackage{doi}
\setlist{itemsep=1.0 pt}

\usepackage[a4paper,top=3cm,bottom=3cm,left=3cm,right=3cm]{geometry}
\usepackage{textgreek}

\theoremstyle{plain}
\newtheorem{teo}{Theorem}[section]
\newtheorem{prop}[teo]{Proposition}
\newtheorem{lemma}[teo]{Lemma}
\newtheorem{cor}[teo]{Corollary}

\theoremstyle{definition}

\newtheorem{defi}[teo]{Definition}
\newtheorem{nota}[teo]{Notation}
\newtheorem{ass}[teo]{Assumption}

\newtheorem{ex/}[teo]{Example}

\newtheorem{rmk/}[teo]{Remark}
\newenvironment{rmk}
  {%
   \pushQED{\qed}\begin{rmk/}}
  {\popQED\end{rmk/}}

\numberwithin{equation}{section}
\DeclareMathOperator{\NN}{\mathrm{NN-syn}}
\DeclareMathOperator{\NNc}{\mathrm{NN-syn,c}}
\DeclareMathOperator{\lrigs}{\mathrm{lrig-syn}}
\DeclareMathOperator{\rigs}{\mathrm{rig-syn}}
\DeclareMathOperator{\lrig}{\mathrm{lrig}}
\DeclareMathOperator{\rig}{\mathrm{rig}}
\DeclareMathOperator{\rigfp}{\mathrm{rig-Gfp}}

\newcommand{\Hom}{\mathrm{Hom}}
\newcommand{\Aut}{\mathrm{Aut}}
\newcommand{\Ker}{\mathrm{Ker}}

\newcommand{\Imm}{\mathrm{Im}}

\newcommand{\Ext}{\mathrm{Ext}}

\newcommand{\A}{\mathcal{A}}
\newcommand{\B}{\mathbb{B}}
\newcommand{\D}{\mathbb{D}}
\newcommand{\Q}{\mathbb{Q}}
\newcommand{\R}{\mathbb{R}}
\newcommand{\C}{\mathbb{C}}
\newcommand{\F}{\mathbb{F}}
\newcommand{\Z}{\mathbb{Z}}

\newcommand{\N}{\mathbb{N}}

\newcommand{\Hf}{{\boldsymbol{f}}}
\newcommand{\Hg}{{\boldsymbol{g}}}
\newcommand{\Hh}{{\boldsymbol{h}}}

\newcommand{\DD}{\mathcal{D}}

\newcommand{\FF}{\mathcal{F}}

\newcommand{\HH}{\mathcal{H}}

\newcommand{\OO}{\mathcal{O}}

\newcommand{\TT}{\mathcal{T}}

\newcommand{\W}{\mathcal{W}}
\newcommand{\XX}{\mathcal{X}}

\newcommand{\VV}{\mathcal{V}}
\newcommand{\WW}{\mathcal{W}}
\newcommand{\ZZ}{\mathcal{Z}}

\newcommand{\Spec}{\mathrm{Spec}}

\newcommand{\Spa}{\mathrm{Spa}}

\newcommand{\SL}{\mathrm{SL}}
\newcommand{\GL}{\mathrm{GL}}

\newcommand{\res}{\mathrm{Res}}

\newcommand{\ord}{\mathrm{ord}}

\newcommand{\Gal}{\mathrm{Gal}}

\newcommand{\Sum}{\mathlarger{\sum}}

\newcommand{\Tr}{\mathrm{Tr}}

\newcommand{\et}{\mathrm{\acute{e}t}}
\newcommand{\proet}{\mathrm{pro\acute{e}t}}
\newcommand{\etc}{\mathrm{\acute{e}t,c}}
\newcommand{\ket}{\mathrm{k\acute{e}t}}

\newcommand{\dR}{\mathrm{dR}}
\newcommand{\syn}{\mathrm{syn}}
\newcommand{\st}{\mathrm{st}}
\newcommand{\pst}{\mathrm{pst}}
\newcommand{\HK}{\mathrm{HK}}
\newcommand{\cris}{\mathrm{cris}}
\newcommand{\Fil}{\mathrm{Fil}}
\newcommand{\BK}{\mathrm{BK}}

\newcommand{\Ig}{\mathrm{Ig}}
\newcommand{\an}{\mathrm{an}}

\newcommand{\DET}{\mathtt{Det}}
\newcommand{\can}{\mathrm{can}}
\newcommand{\Tate}{\mathrm{Tate}}
\newcommand{\Sym}{\mathrm{Sym}}
\newcommand{\Hs}{\mathscr{H}}
\newcommand{\Ls}{\mathscr{L}}
\newcommand{\pr}{\mathtt{pr}}
\newcommand{\Es}{\mathscr{E}}

\begin{document}

\title{Explicit reciprocity laws for diagonal classes: higher level cases}
\author{Luca Marannino}
\date{last update: \today}

\begin{abstract}
We generalize the $p$-adic explicit reciprocity laws for balanced diagonal classes appearing in \cite{DR2017} and \cite{BSV2020a} to the case of geometric balanced triples $(f,g,h)$ of modular eigenforms where $f$ is a $p$-ordinary newform, while $g$ and $h$ are allowed to be (both) supercuspidal at $p$ or (both) ramified principal series at $p$.
\end{abstract}

\address{Università degli Studi di Milano, Dipartimento di Matematica ``Federigo Enriques", via Saldini 50, 20123, Milano, Italy}
\email{luca.marannino@unimi.it}

\maketitle
\tableofcontents

\section{Introduction and statement of the main results}
\subsection{Main results}
Let $p$ denote an odd rational prime. Fix a positive integer $M$ coprime to $p$ and a positive integer $t$ such that $Mp^t\geq 5$. Let $Y_t:=Y_1(Mp^t)_{\Q}$ denote the \emph{open} modular curve over $\Q$ of level $\Gamma_1(Mp^t)$ and $X_t:=X_1(Mp^t)_\Q$ denote the \emph{compactified} modular curve of that level. We consider a triple of cuspidal modular forms
\[
f\in S_k(Mp^t,\chi_f),\quad g\in S_l(Mp^t,\chi_g),\quad h\in S_m(Mp^t,\chi_h).
\]

\begin{ass}
We assume that the triple $(f,g,h)$ satisfies the following requirements:

\begin{itemize}
\label{introass}
\item [(i)] for $\xi\in\{f,g,h\}$, we assume that $\xi$ is a normalized eigenform (for the level $Mp^t$) and that $\xi$ is an eigenform for the $U_p$ operator;
\item [(ii)] The triple $(f,g,h)$ is \emph{self-dual}, i.e., $\chi_f\chi_g\chi_h$ is the trivial character modulo $Mp^t$ (in particular $k+l+m$ is an even integer);
\item [(iii)] The triple of weights $(k,l,m)$ is \textbf{balanced} and \emph{geometric}, i.e., $(k,l,m)$ are the sizes of the edges of a triangle and $\nu\geq 2$ for $\nu\in\{k,l,m\}$.
\end{itemize}
\end{ass}

We fix a finite and large enough extension $L$ of $\Q_p$ (with ring of integers $\OO_L$) containing the Fourier coefficients of $f,g,h$ (and a primitive $Mp^t$-th root of $1$).

We write $\mathbf{r}:=(k-2,l-2,m-2)\in(\Z_{\geq 0})^3$ and $r:=(k+l+m-6)/2$.

\medskip
In part 3 of \cite{BSVast1}, the authors associate to the triple $(f,g,h)$ a Galois cohomology class $\kappa(f,g,h)\in H^1(\Q,V(f,g,h))$. Here $V(f,g,h)$ is essentially a suitable twist of the tensor product of the (duals of the) Deligne representations attached to $f,g,h$ (more precisely, the direct sum of some copies of this tensor product). 

These cohomology classes are realized as the $(f,g,h)$-isotypical projection of the pushforward along the diagonal $d_t: Y_t\hookrightarrow Y_t^3$ ($p$-adic Abel--Jacobi map) of an invariant $\DET_\mathbf{r}^\et$ depending only on the triple of weights $(k,l,m)$ (or equivalently, as shown in \cite[Section 3.2]{BSVast1}, the pushforward of a so-called \emph{generalized Gross--Kudla--Schoen diagonal cycle} on the corresponding product of Kuga--Sato varieties, again only depending on $\mathbf{r}$). The construction of $\kappa(f,g,h)$ is recalled in more detail in Section \ref{etaleAbelJacobichapt}.

\medskip
We will always work under some $p$-ordinarity hypothesis on the form $f$.
\begin{itemize}
    \item  [($f$-ord)] There exist a positive integer $M_1\mid M$ and $s\in\Z_{\geq 1}$ such that $f\in S_k(M_1p^s,\chi_f)$ is a $p$-stabilized $p$-ordinary newform of level $M_1p^s\geq 5$ (i.e., $f$ is either $p$-ordinary and new of level $M_1p^s$ or the ordinary $p$-stabilization of a newform of level $M_1$)
\end{itemize}

The setting to which we dedicate more relevance in this work is characterised by the following assumption.

\begin{itemize}
    \item [(SC)] The forms $g$ and $h$ are supercuspidal at $p$ and lie in the kernel of $U_p$.
\end{itemize}

A first complication introduced by assumption (SC) is that one can only hope that the class $\kappa(f,g,h)$, viewed as a local class in $H^1(\Q_p, V(f,g,h))$, becomes crystalline over a non-trivial finite extension of $\Q_p$. Nevertheless, in Sections \ref{padicHTchapter} and \ref{syntomicabeljacobichapt} we explain (under some minimal assumptions) how to view the Bloch--Kato logarithm of $\kappa(f,g,h)$ as a linear functional
\[
\log_\BK^{fgh}(\kappa(f,g,h)):\Fil^0(V^*_\dR(f,g,h))\to L\,.
\]
Here $V^*(f,g,h)$ arises as the Kummer dual of $V(f,g,h)$ (and by the self-duality assumption on  $(f,g,h)$ it is actually isomorphic to $V(f,g,h)$ itself).

\medskip
One can find a distinguished element
\[
\eta^{\varphi=a_p}_f\otimes\omega_g\otimes\omega_h\otimes t_{r+2}\in\Fil^0(V^*_\dR(f,g,h))
\]
which is defined more precisely in Section \ref{studyphiNsection} (let us just note here that $\{t_n\}_{n\in\Z}$ denotes a compatible choice of generators for $\D_\dR(\Q_p(n))$ for varying $n\in\Z$, corresponding to a choice of a uniformizer $t\in\B_\dR^+$). The explicit reciprocity law alluded to in the title consists in giving an explicit and computable description of the value of $\log_\BK^{fgh}(\kappa(f,g,h))$ at $\eta^{\varphi=a_p}_f\otimes\omega_g\otimes\omega_h\otimes t_{r+2}$. 

\medskip
The following theorem is the main result of this work.

\begin{teo}[cf. Theorem \ref{explicirreclawconj}]
\label{erlintroteo}
Let $(f,g,h)$ be a triple of modular forms satisfying Assumption \ref{introass} and the conditions ($f$-ord) and (SC) as above. If the triple $(f,g,h)$ is moreover $(F,1-T)$-convenient in the sense of Definition \ref{expconvtriple} for some finite extension $F/\Q_p$, then
\[
\log_\BK^{fgh}(\kappa(f,g,h))(\eta^{\varphi=a_p}_f\otimes\omega_g\otimes\omega_h\otimes t_{r+2})
\]
is equal to
\[
(-1)^{k-2}(r-k+2)!\cdot a_1\big(e_{\breve{f}}(\Tr_{Mp^t/M_1p^t}(g\times d^{(k-l-m)/2}h))\big)\,.
\]
\end{teo}
The notation goes as follows. We let $a_1(\xi)$ denote the first Fourier coefficient of the $q$-expansion at $\infty$ of a modular form $\xi$. The modular form $\breve{f}$ is the normalized eigenform which is a scalar multiple of $w_{M_1}(f)$ (where $w_{M_1}$ is a suitably defined Atkin--Lehner operator) and $e_{\breve{f}}$ denotes $\breve{f}$-isotypical projection.

Finally, $d$ denotes Serre's derivative operator, which acts as $q\frac{d}{dq}$ on $q$-expansions. Note that for a negative integer $t$ (here $(k-l-m)/2<0$ since the triple of weights $(k,l,m)$ is balanced), we define $d^t$ as the $p$-adic limit of the operators $d^{t+(p-1)p^m}$ for $m\to +\infty$. Then $d^t$ is an operator sending $p$-adic modular forms of weight $\nu$ to $p$-adic modular forms of weight $\nu+2t$. In particular one can interpret $g\times d^{(k-l-m)/2}h$ as a $p$-adic modular form of weight $k$. Since the operator $e_{\breve{f}}$ includes an ordinary projection, Hida's classicality theorem shows that
\[
e_{\breve{f}}(\Tr_{Mp^t/M_1p^t}(g\times d^{(k-l-m)/2}h))
\]
is given by the $q$-expansion of a classical modular form of weight $k$ and level $M_1 p^t$.

\medskip
The condition that $(f,g,h)$ is $(F,1-T)$-convenient is not too restrictive. It means that the representation $V(f,g,h)$, seen as a representation of $\Gal(\bar{\Q}_p/\Q_p)$, becomes crystalline over $F$ and that the action of the $d$-th power of crystalline Frobenius on $\D_{\cris,F}(V(f,g,h))$ does not admit $1$ or $p^{-d}$ as eigenvalues, where $p^d$ is the cardinality of the residue field of $F$. In the setting of Theorem \ref{erlintroteo}, it is always satisfied if the weight $k$ of $f$ is at least $3$ (see Proposition \ref{efgequal}).

\subsection{Comparison with the existing literature and proof strategy}
\label{introcomparison}
\begin{rmk}
Note that the explicit reciprocity laws proven in \cite{DR2017} and \cite{BSV2020a} always require some finite slope (or ordinarity) assumption also on $g$ and $h$. To the author's knowledge, the case of $g$ and $h$ supercuspidal at $p$ has not been addressed in the literature so far.    
\end{rmk}

\begin{rmk}
In Theorem \ref{explicirreclawconj}, we actually just assume that the eigenforms $g$ and $h$ are $p$-depleted (always assuming that the triple $(f,g,h)$ is $(F,1-T)$-convenient in the sense described above). In particular, up to passing to $p$-depletions, we obtain an explicit reciprocity law which holds without any assumption on the local behaviour of the forms $g$ and $h$. This is again the first example appearing in the literature of such a result, as far as we are aware of. 
\end{rmk}

\medskip
The proof of Theorem \ref{erlintroteo}, developed in Section \ref{explicitreciprocitychapt}, follows closely the steps of the proof of Theorem A in \cite{BSV2020a} (namely a $p$-adic explicit reciprocity law for triples of eigenforms of level coprime to the fixed prime $p$). However, in our setting we have to face some further difficulties. In particular, since the forms $g$ and $h$ inevitably have non-trivial level at $p$, we are forced to work over modular curves (or products of such) that do not have good reduction at $p$ and only admit a semistable model over the ring of integers of a finite (typically ramified) extension of $\Q_p$. In this setting, cohomology theories such as Hyodo--Kato cohomology and syntomic cohomology - that usually allow this kind of computations - are more complicated to handle and we have to take these drawbacks into account.

\medskip
The work of Nekov\'{a}\v{r}--Nizio\l{} (\cite{NN2016})) produces a crystalline Hyodo--Kato cohomology with trivial coefficients for general varieties which compares with \'{e}tale cohomology. In our setting, this theory can be applied to the compactified Kuga--Sato varieties by cutting out a piece of the cohomology which corresponds to the motive of the modular forms of the weight of interest (or tensor product of them). On the \'{e}tale side, the existence of a theory with coefficients and of Leray spectral sequences allows to connect this pieces of the cohomology to the cohomology of the modular curve with values in \'{e}tale automorphic coefficients. For the proof of the explicit reciprocity laws, one typically needs to move to the rigid analytic (overconvergent) setting and then restrict the computation to suitable rigid (dagger) subspaces of the relevant varieties (namely ordinary loci). A rigid version of Hyodo--Kato and syntomic cohomology in the semistable case with coefficients has been developed in \cite{EY2021}, \cite{EY2024}, \cite{EY2025}, and \cite{Yam2025}. Also, a relationship with the theory of Nekov\'{a}\v{r}--Nizio\l{} has been established in the case of
trivial coefficients in \cite{EY2021}, as well as a relationship with the Berthelot's rigid cohomology, when smooth models exist, in \cite{EY2024}. 

\medskip
On the other hand, in the case of automorphic coefficients for modular curves (for which semistable models exist), it is not clear whether one can compare the theory of Ertl--Yamada with the theory of Nekov\'{a}\v{r}--Nizio\l{} (the latter being employed as explained above, i.e., cutting out pieces of the cohomology of Kuga--Sato varieties with trivial coefficients). This is the issue that hat had prevented us from obtaining an unconditional proof of Theorem \ref{erlintroteo} in the case of balanced triples of weights $(k,l,m)\neq (2,2,2)$ employing the formalism developed in the aforementioned works. We could only obtain a proof after postulating a reasonable behaviour of rigid syntomic cohomology with coefficients, cf. Assumption \ref{conditional}.

\medskip
As observed by one of the referees, a possibility to circumvent this problem would be to grant the existence of semistable integral models for the compactification of the Kuga--Sato varieties of arbitrary level restricting to smooth models on the ordinary loci. In this case, it would be possible to establish a relationship between the Hyodo--Kato cohomology of Nekov\'{a}\v{r}--Nizio\l{} and that of Ertl--Yamada and, moreover, to compare it with Berthelot's rigid cohomology. Furthermore, using suitable rigid analytic Leray spectral sequences, one could reduce the problem to a computation over ordinari loci of modular curves with rigid analytic coefficients. Unfortunately, obtaining these semistable models in the required generality seems out of reach at the moment.

\medskip
After that this article was completed, the preprint \cite{ABSV2025} appeared. It introduces a suitable formalism of crystalline syntomic cohomology with coefficients which compares with the \'{e}tale chomology and allows to bypass the problems mentioned above, since one only needs semistable models for compactified modular curves (rather than for compactified Kuga--Sato varieties), thus allowing us to upgrade Theorem \ref{erlintroteo} to cover general balanced triples of weights $(k,l,m)$. We refer to Section \ref{thegeneralcasesubsection} for more details. 

\medskip
For completeness, let us also mention that another candidate for crystalline Hyodo--Kato coefficients has been defined in \cite{DN2018} but, as explained by the authors in the introduction of loc. cit., a relationship between these coefficients and those of rigid analytic nature has not been proved.

\subsection{An overview on future applications}
We would like to close this introduction by underlining the close link between the generalized triple product $p$-adic $L$-function $\mathscr{L}^f_p(\Hf,\Hg,\Hh)$ studied in \cite{Mar2026} and the explicit reciprocity law of Theorem \ref{erlintroteo}. Indeed, when $\Hg$ and $\Hh$ are generalized $p$-adic families admitting classical specializations which are supercuspidal at $p$ (and $p$-depleted) - as it happens in the case of families of theta series of infinite $p$-slope described in \cite[Section 4]{Mar2026} - we have that, essentially by construction,
\[
\mathscr{L}^f_p(\Hf,\Hg,\Hh)(w)=\gamma_{\Hf_{\!x}}\cdot a_1\big(e_{\breve{\Hf\,}_{\!x}}(\Tr_{Mp^t/M_1p^t}(\Hg_y\times d^{(k-l-m)/2}\Hh_z))\big)\,.
\]
for every balanced triple of meaningful weights $w=(x,y,z)$, where $\gamma_{\Hf_{\!x}}$ denotes the specialization at $x$ of the congruence number $\gamma_{\Hf}$ of the Hida family $\Hf$ (see for instance the discussion in \cite[Section 3.3]{Hsi2021}). Assuming that $(\Hf_{\!x},\Hg_y,\Hh_z)$ satisfies the self-duality condition of Assumption \ref{introass} (ii) above, we deduce immediately the equality
\begin{equation}
\label{finalintroformula}
\mathscr{L}^f_p(\Hf,\Hg,\Hh)(w)=\frac{(-1)^{k-2}\cdot \gamma_{\Hf_{\!x}}}{(r-k+2)!}\cdot 
\log_\BK^{fgh}(\kappa(\Hf_{\!x},\Hg_y,\Hh_z))(\eta^{\varphi=a_p}_{\Hf_{\!x}}\otimes\omega_{\Hg_y}\otimes\omega_{\Hh_z}\otimes t_{r+2})
\end{equation}
Such a result is consistent with the fact that one expects to interpolate $p$-adically in a meaningful way the objects appearing in the RHS of formula \eqref{finalintroformula}, in particular the classes $\kappa(\Hf_{\!x},\Hg_y,\Hh_z)$. Following \cite{DR2017} and \cite{BSVast1}, there should exist a \emph{big} diagonal class $\kappa(\Hf,\Hg,\Hh)$ interpolating $p$-adically the diagonal classes $\kappa(\Hf_{\!x},\Hg_y,\Hh_z)$. 

\medskip
Motivated by results of \cite{BSVast2}, one expects to provide a relation between the specialization at $(2,1,1)$ of $\kappa(\Hf,\Hg,\Hh)$ (or a suitable improvement) to certain $p$-adic derivatives of $\mathscr{L}_p^f(\Hf,\Hg,\Hh)$ evaluated at $(2,1,1)$. In the arithmetic setting described in \cite[Section 6]{Mar2026}, this result would provide a link between Heegner points and ($p$-adic limits of) diagonal classes. 

\medskip
The study of the case in which $\Hg$ and $\Hh$ are $p$-adic families of theta series of infinite $p$-slope (coming from Hecke characters of a imaginary quadratic field where $p$ is inert) should also provide applications of Theorem \ref{erlintroteo} to the anticyclotomic Iwasawa theory of modular forms at inert primes. In this case less is known compared to the case where $p$ splits in relevant imaginary quadratic field (and the corresponding $p$-adic families of theta series are classical Hida families). In the preprint \cite{CD2023}, the authors use diagonal classes to study the anticyclotomic theory of modular forms in the $p$-split case. Theorem \ref{erlintroteo} should be an important step towards a generalization of some of the constructions and results of \cite{CD2023} to the $p$-inert setting.

\medskip
The author will investigate these sort of questions in future works.

\subsection*{Notation and conventions}
Throughout this work, we fix an algebraic closure $\bar{\Q}$ of $\Q$, an algebraic closure $\bar{\Q}_p$ of $\Q_p$ (for the fixed prime $p$) together with an embedding $\iota_p:\bar{\Q}\hookrightarrow\bar{\Q}_p$ extending the canonical inclusion $\Q\hookrightarrow\Q_p$. All algebraic extensions of $\Q$ (resp. $\Q_p$) are viewed inside the corresponding fixed algebraic closures. We extend the $p$-adic absolute value $|\cdot|_p$ on $\Q_p$ (normalized so that $|p|_p=1/p$) to $\bar{\Q}_p$ in the unique possible way. We denote by $\C_p$ the completion of $\bar{\Q}_p$ with respect to this absolute value. It is well-known that $\C_p$ is itself algebraically closed.
We also fix an embedding $\iota_{\infty}:\bar{\Q}\hookrightarrow\C$ extending the canonical inclusion $\Q\hookrightarrow\C$ and we often omit the embeddings $\iota_p$ and $\iota_\infty$ from the notation. 

\medskip
If $F$ is any field, we denote by $G_F$ the absolute Galois group of $F$ (defined after fixing a suitable separable closure) and we denote $F^{ab}$ the maximal abelian extension of $F$ (inside such a separable closure).

\medskip
Given a smooth function $f$ on the upper-half plane $\HH:=\{ \tau\in\C\mid \Imm(\tau)>0\}$ and $\omega=\begin{psmallmatrix}a & b\\ c & d\end{psmallmatrix}\in\GL_2(\R)^+$ (invertible $2\times 2$ matrices with positive determinant) and $k\in\Z$, we set
\[
f|_k\omega(\tau):=\det(\omega)^{k/2}\cdot (c\tau +d)^{-k}\cdot f\big(\tfrac{a\tau+b}{c\tau+d}\big)\qquad \tau\in\HH
\]

\medskip
If $\Gamma\subseteq\SL_2(\Z)$ is a congruence subgroup and $k\in\Z_{\geq 1}$, we let $M_k(\Gamma)$ (resp. $S_k(\Gamma)$) be the $\C$-vector space of (holomorphic) modular forms (resp. cusp forms) of weight $k$ and level $\Gamma$. For $\Gamma=\Gamma_1(N)$ for some $N\geq 1$ and $\chi$ a Dirichlet character modulo $N$, we let $M_k(N,\chi)$ (resp. $S_k(N,\chi)$) denote the spaces of modular forms (resp. cusp forms) of weight $k$, level $\Gamma_1(N)$ and nebentypus $\chi$. Unless otherwise specified, we refer to \cite{Miy1989} for all the basic facts concerning the analytic theory of modular forms which are mentioned freely without proof.

\medskip
We refer to the notes \cite{BC2009} for the basic facts concerning $p$-adic Hodge theory and for the definition of Fontaine's period rings $\B_\dR$, $\B_\cris$ and $\B_\st$.

\subsection*{Acknowledgements}
This work originates from the author's PhD thesis written under the supervision of Massimo Bertolini. We would like to thank him for suggesting this project and for his precious advice. We thank the anonymous referees for their corrections, comments and suggestions that led to a significant improvement of the paper. 
The author was funded by the DFG Graduiertenkolleg 2553 during his doctoral studies at Universität Duisburg-Essen (while writing the first version of this work) and by the Simons Collaboration on Perfection in Algebra, Geometry and Topology as postdoctoral researcher at CNRS (during most of the revision process). He is currently funded by the Progetto di Eccellenza U-GOV DECC23\_012\_AC as a postdoctoral researcher at Università degli Studi di Milano.


\section{Balanced diagonal classes and the étale Abel--Jacobi map}
\label{etaleAbelJacobichapt}
In this section we introduce the facts about étale and de Rham cohomology of modular curves with $p$-adic coefficients which are needed to define the balanced diagonal classes and the étale Abel--Jacobi map alluded to in the title above. Most of the material is covered in more detail in \cite[Section 1]{BDP2013} and/or in \cite[Sections 2 and 3]{BSVast1} and in the references given therein.

\subsection{De Rham cohomology of modular curves}
\label{derhametale}
For every integer $N\geq 5$, we let $Y_1(N)$ denote the \emph{open} modular curve of level $\Gamma_1(N)$, defined over $\Z[1/N]$, classifying isomorphism classes of pairs $(E,P)$ where $E$ is a family of elliptic curves over a $\Z[1/N]$-scheme $S$ and $P\in E(S)$ is a section of exact order $N$. The curve $Y_1(N)$ is affine and smooth over $\Z[1/N]$. The universal elliptic curve arising from this moduli problem will be denoted $u_1(N): \Es_1(N)\to Y_1(N)$.

\medskip
As explained in \cite[Chapter 8]{KM1985}, the normalization of the projective $j$-line in $Y_1(N)$ is a smooth projective curve over $\Z[1/N]$, usually denoted by $X_1(N)$. The curve $X_1(N)$ contains $Y_1(N)$ as an open subscheme and the complement $C_1(N)$ of $Y_1(N)$ in $X_1(N)$, with the reduced subscheme structure, is finite and étale over $\Z[1/N]$. It is usually called the subscheme of cusps. Indeed, over $\Z[1/N,\zeta_N]$ (where $\zeta_N$ is a fixed primitive $N$-th root of unity in $\bar{\Q}$), the scheme $C_1(N)$ is simply given by a disjoint union of points (usually one refers to those as the cusps).

\medskip
For any $\Z[1/N]$-algebra $R$, we let $Y_1(N)_R$ (resp. $X_1(N)_R)$, $C_1(N)_R$) denote the base change of $Y_1(N)$ (resp. $X_1(N)$, $C_1(N)$) to $R$.

\medskip
In what follows, $F$ will be a field of characteristic zero and we write $Y:=Y_1(N)_F$, $X:=X_1(N)_F$, $C:=C_1(N)_F$. We let, moreover, $u:\mathscr{E}\to Y$ denote the universal elliptic curve (over $F$).

\medskip
The notation $(\Tate(q),P_\Tate)_{/F(\!(q^{1/d})\!)}$ will denote the Tate elliptic curve $\mathbb{G}_m/q^\Z$, with a $\Gamma_1(N)$-level structure $P_\Tate$ defined over $F[\zeta_N](\!(q^{1/d})\!)$ for some $d\mid N$. We let $\omega_\can$ denote the canonical differential $\omega_\can:= d T/T$ over $F(\!(q)\!)$, where $T$ is the parameter on $\mathbb{G}_m$. The level structure $P_\Tate$ corresponds (over $F[\zeta_N](\!(q^{1/d})\!)$) to a point of order $N$ on $\Tate(q)$, i.e., a point of the form $\zeta_N^i q^{j/d}$ for some $i\in\{1,\dots,N-1\}$ and $j\in\{1,\dots,d-1\}$ with $(i,N)=1$ and $(j,d)=1$.

\medskip
We let $\omega:=u_*\Omega^1_{\Es/Y}$ be the line bundle of relative differentials on $\Es/Y$ and we let $\Hs=R^1 u_*(\Omega^\bullet_{\Es/Y})$ be the relative de Rham cohomology sheaf on $Y$. The latter is a rank $2$ vector bundle over $Y$, equipped with Hodge filtration
\begin{equation}
\label{Hodgefiltration}
0\to\omega\to\Hs\to\omega^{-1}\to 0\,.  
\end{equation}
The vector bundle $\Hs$ is also equipped with the so-called Gauss--Manin connection
\begin{equation}
\nabla: \Hs\to\Hs\otimes\Omega^1_Y\,.    
\end{equation}
The Kodaira--Spencer map is the composite
\begin{equation}
KS:\omega\to\Hs\xrightarrow{\nabla}\Hs\otimes\Omega^1_Y\to\omega^{-1}\otimes\Omega^1_Y\, 
\end{equation}
induced by the Gauss--Manin connection and the Hodge filtration. It turns out that $KS$ is an $\OO_Y$-linear isomorphism of sheaves, giving rise to an identification $\omega^2\cong\Omega^1_Y$ of line bundles on $Y$.

\medskip
One can extend $\omega$ and $\Hs$ to sheaves on the projective curve $X$, which will be again denoted (resp.) $\omega$ and $\Hs$. One way to proceed is to view $X$ as moduli space for a suitable moduli problem involving generalised elliptic curves. Since these constructions are well-known and appear often in the literature, we introduce only the facts needed in this work.

\medskip
The sheaf $\omega$ on $X$ is essentially characterised by the fact that for $F=\C$ there is an identification $H^0(X,\omega^k)=M_k(\Gamma_1(N))$ (where $M_k(\Gamma_1(N))$ denotes the $\C$-vector space of holomorphic modular forms of level $\Gamma_1(N)$). The local sections of $\omega$ in the (formal) neighbourhood over $\Spec(F[\zeta_N][\![q^{1/d}]\!])$ of the cusp attached to the pair $(\Tate(q),\zeta_N q^{1/d})$ are expressions of the form $h\cdot\omega_\can$ with $h\in F[\zeta_N][\![q^{1/d}]\!]$ and $\omega_\can$ the canonical differential on the Tate curve.

\medskip
The extension of $\Hs$ is then determined by the extension of $\omega$ and the Hodge filtration \eqref{Hodgefiltration}. The Gauss--Manin connection $\nabla$ extends to a connection with logarithmic poles at the cusps
\begin{equation}
\label{GaussManinX}
 \nabla:\Hs\to\Hs\otimes\Omega^1_X\langle C\rangle \,.  
\end{equation}

\medskip
The local sections of $\Hs$ in a neighbourhood of $(\Tate(q),\zeta_N q^{1/d})$ are $F[\zeta_N][\![q^{1/d}]\!]$-linear combinations of $\omega_\can$ and the local section $\eta_\can$, which is defined by the equation
\begin{equation}
\label{etacan}
\nabla \omega_\can = \eta_\can\otimes \frac{dq}{q}
\end{equation}
Over $\Spec(F[\zeta_N][\![q^{1/d}]\!])$, the Gauss--Manin connection is completely determined by the above equation and by $\nabla \eta_\can =0$.

\medskip
The Kodaira--Spencer map gives rise to an isomorphism $\sigma:\omega^2\xrightarrow{\cong}\Omega^1_X\langle C\rangle$, which over $\Spec(F[\zeta_N][\![q^{1/d}]\!])$ is simply described by $\sigma(\omega_\can^2)=\tfrac{dq}{q}$.

\medskip
A possible definition of modular forms of weight $k\geq 2$ and level $\Gamma_1(N)$ with $F$-coefficients is then given by
\[
M_k(\Gamma_1(N),F):=H^0(X,\omega^k)=H^0(X,\omega^{k-2}\otimes\Omega^1_X\langle C\rangle)
\]
with subspace of cusp forms defined as
\[
S_k(\Gamma_1(N),F):=H^0(X,\omega^{k-2}\otimes\Omega^1_X)\,.
\]

\medskip
We now let, for $r\geq 1$, $\Hs_r:=\Sym^r(\Hs)$ (which we view as a sheaf on $Y$ or $X$, depending on the context). The sheaf $\Hs_r$ is endowed with a Hodge filtration 
\[
\Hs_r\supset\Hs_{r-1}\otimes\omega\supset\cdots\supset\omega^r
\]
induced by the filtration \eqref{Hodgefiltration}. We will set $\Fil^i\Hs_r:=\Hs_{r-i}\otimes\omega^i$ for $i=0,\dots,r$.

\medskip
The Gauss--Manin connection \eqref{GaussManinX} extends to a connection $\nabla_r:\Hs_r\to\Hs_r\otimes\Omega^1_X\langle C\rangle$ which satisfies Griffiths transversality
\[
\nabla(\Fil^{i+1}\Hs_r)\subseteq\Fil^i\Hs_r\otimes\Omega^1_X\langle C\rangle\qquad (0\leq i\leq r-1)
\]
and induces isomorphisms on the graded pieces
\[
\mathrm{Gr}^{i+1}\Hs_r\cong\mathrm{Gr}^i\Hs_r\otimes\Omega^1_X\langle C\rangle\qquad (0\leq i\leq r-1)\,.
\]

\medskip
We let $\Ls:=\mathcal{H}om_{\OO_X}(\Hs,\OO_X)$ be the dual of $\Hs$ and for $r\geq 1$ we set $\Ls_r:=\mathrm{Tsym}^r(\Ls)$ (the notation $\mathrm{Tsym}$ refers to the submodule of symmetric tensors). We use the same notations for the restrictions of these sheaves to $Y$. Also $\Ls$ (and consequently $\Ls_r$) is equipped with an induced Hodge filtration and integrable connection. For $\mathscr{F}\in\{\Hs,\Ls\}$ we finally define the de Rham cohomology groups
\begin{equation}
H^i_\dR(X,\mathscr{F}_r):=\mathbb{H}^i(X,\mathscr{F}_r\xrightarrow{\nabla_r}\mathscr{F}_r\otimes\Omega^1_X\langle C\rangle)\,,
\end{equation}
where $\mathbb{H}$ denotes hypercohomology. These de Rham cohomology groups are naturally endowed with a two-step descending filtration. In particular, one has an identification
\[
\Fil^i H^1_\dR(X,\Hs_r)\otimes_F F[\zeta_N]= M_{r+2}(\Gamma_1(N),F)\otimes_F F[\zeta_N]\qquad i=1,\dots ,r+1\,.
\]
Since we are working with connections with logarithmic poles at the cusps, the de Rham cohomology groups defined above can actually be defined on the open modular curve $Y$, i.e., there is a natural isomorphism of filtered $F$-vector spaces
\begin{equation}
\label{XYdeRhamiso}
    H^i_\dR(X,\mathscr{F}_r)\cong H^i_\dR(Y,\mathscr{F}_r|_Y)\,.
\end{equation}

One can similarly define de Rham cohomology with compact support $H^i_{
\dR,c}(Y,\mathscr{F}_r)$.

\medskip
Over $Y$, there is a perfect pairing (essentially coming from Poincaré duality on the universal elliptic curve $\mathscr{E}$)
\[
(~,~):\Hs\otimes_{\OO_Y}\Hs\to\OO_Y[-1]\,.
\]
Here $\OO_Y[n]$ is the sheaf $\OO_Y$, with trivial connection and shifted filtration (i.e., $\Fil^j\OO_Y[n]=\OO_Y$ if $j\leq -n$ and $\Fil^j\OO_Y[n]=0$ if $j\geq 1-n$). Such a pairing induces a perfect duality
\begin{equation}
\label{deRhamdualsheaf}
    (~,~)_r:\Hs_r\otimes_{\OO_Y}\Hs_r\to \OO_Y[-r]\,,
\end{equation}
given on generic fibers by the rule
\[
(\alpha_1\,\cdots\,\alpha_r,\beta_1\,\cdots\,\beta_r)_r=\frac{1}{r!}\Sum_{\sigma\in\Sym(r)}(\alpha_1,\beta_{\sigma 1})\,\cdots \,(\alpha_r,\beta_{\sigma r}),
\]
where $\Sym(r)$ denotes the symmetric group on $r$ symbols.

At the level of de Rham cohomology this induces a perfect duality
\begin{equation}
\label{deRhamdualcoh}
    (~,~)_{\dR, Y, r}: H^1_\dR(Y,\Hs_r)\otimes_F H^1_{\dR,c}(Y,\Hs_r)\to H^2_{\dR,c}(Y,\OO_Y[-r])\cong F[-r-1]\,.
\end{equation}

\subsection{Étale cohomology of modular curves}
\label{etalemodularcurves}
The sheaves $\Hs_r$ (resp. $\Ls_r$) admit (Kummer) pro-étale versions. One can work in the more classical setting of \cite[\S 12]{FK1988} and obtain locally constant $p$-adic sheaves $\Hs_r$ (resp. $\Ls_r$) on $Y_1(N)$, so that it makes sense to study the cohomology groups $H^j_\et(Y_1(N)_R,\mathscr{\Hs}_r)$ (resp. $H^j_\et(Y_1(N)_R,\mathscr{\Ls}_r)$) for any $\Q$-algebra $R$ (cf. \cite[Section 2.3]{BSVast1} and the references therein). In this section, we prefer to focus on the rigid analytic (or better adic) setting and to adopt the more modern approach of \cite{Sch2013} (and its various generalizations, for instance \cite{DLLZ2023}).

\medskip
In this section we denote by $F$ a complete discretely valued field of mixed characteristic $(0,p)$ with perfect residue field and fixed algebraic closure $\bar{F}$.

\begin{nota}
For a rigid analytic variety $S$ over $F$, we write $\hat{\Z}_{p,S}$ (resp. $\hat{\Q}_{p,S}$) to denote the constant sheaf on $S_\proet$ associated to $\Z_p$ (resp. $\Q_p$) and for a $\Z_p$-local system $\mathbb{L}$ (resp. $\Q_p$-local
system) on $S_\et$, we let $\widehat{\mathbb{L}}$ to denote the lisse $\hat{\Z}_{p,S}$-module (resp. $\hat{\Q}_{p,S}$-module) on $S_\proet$ associated with $\mathbb{L}$ (cf. the discussion in \cite[Section 8.2]{Sch2013}). We let $H^j_\et(S,\mathbb{L})$ (resp. $H^j_\et(S_{\bar{F}},\mathbb{L})$) denote the $p$-adic étale cohomology groups.

For $A$ a complete subring of $\C_p$ and $\FF$ a $\hat{\Z}_{p,S}$-module, we also set $\FF_{\!A}:=\FF\otimes_{\hat{\Z}_{p,S}}\hat{A}_S$, where $\hat{A}_S$ is the constant pro-étale sheaf attached to $A$.
\end{nota}

\begin{rmk}
\label{etaleproetalesystems}
According to \cite[Proposition 8.2]{Sch2013}, the association $\mathbb{L}\to\widehat{\mathbb{L}}$ defines an equivalence between the category of $\Z_p$-local systems on $S_\et$ and the category of lisse $\hat{\Z}_{p,S}$-modules on $S_\proet$.    
\end{rmk}

In what follows we view the algebraic varieties $Y,X,\Es$ introduced in section \ref{derhametale} as rigid analytic varieties over $\Q_p$ (or better locally noetherian adic spaces over $\Spa(\Q_p,\Z_p)$), without changing the notation.

\begin{defi}
We let $\HH^1(\Es):= R^1 u_* \hat{\Z}_{p,\Es}$ to be the first relative pro-étale cohomology sheaf of the family $u:\Es\to Y$. We also let $\TT_p(\Es):= \HH om_{\hat{\Z}_{p,Y}}(\HH^1(\Es),\hat{\Z}_{p,Y})$ denote the relative $p$-adic Tate module of the family $u:\Es\to Y$.    
\end{defi}

The sheaves $\HH^1(\Es)$ and $\TT_p(\Es)$ are (pro-étale versions of) rank $2$ $\Z_p$-local systems. The perfect relative $p$-adic Weil pairing can then be seen as a perfect pairing
\begin{equation}
\label{relativeWeilproet}
    \TT_p(\Es)\otimes_{\hat{\Z}_{p,Y}} \TT_p(\Es)\to \hat{\Z}_{p,Y}(1)\,,
\end{equation}
under which one obtains an isomorphism $ \TT_p(\Es)\cong \HH^1(\Es)(1)$.

Here $\hat{\Z}_{p,Y}(1)$ is the Tate twist of $\hat{\Z}_{p,Y}$ (one can obtain it as usual as $\varprojlim_n \mu_{p^n,_Y}$ with Galois action given by the $p$-adic cyclotomic character). More generally for a lisse $\hat{\Z}_{p,Y}$-module and every $n\in\Z$, we let $\FF(n):=\FF\otimes_{\hat{\Z}_{p,Y}} \hat{\Z}_{p,Y}(n)$ (with the usual conventions on higher Tate twists).

\begin{defi}
For every integer $r\geq 0$, we define $\hat{\Z}_{p,Y}$-modules $\Hs_r:=\Sym^r(\HH^1(\Es))$ and $\Ls_r:=\mathrm{Tsym}^r( \TT_p(\Es))$ on $Y_\proet$.
\end{defi}

Clearly, the relative Weil pairing \eqref{relativeWeilproet} induces for all integers $r\geq 0$ an isomorphism
\begin{equation}
\label{isosr}
    s_r:\Hs_{r,\Q_p}\xrightarrow{\cong}\Ls_{r,\Q_p}(-r)
\end{equation}
(we have to invert $p$ to obtain an isomorphism when we pass to $\Sym^r$ and $\mathrm{Tsym}^r$).

The careful reader will immediately notice the abuse of notation for the symbols $\Hs_r$ and $\Ls_r$. This is justified in the following remark.

\begin{rmk}
Let $S$ be any smooth rigid analytic variety over $\Q_p$ (or any smooth locally noetherian adic space over $\Spa(\Q_p,\Z_p)$). As explained in \cite{Sch2013} (and extended in \cite{DLLZ2023} to the case of not necessarily proper varieties), when a $\Z_p$-local system $\mathbb{L}$ on $S_\et$ is \emph{de Rham} (cf. \cite[Definition 8.3]{Sch2013}, \cite[Theorem 3.9]{LZ2017}), one can attach to it a pair $(\mathcal{E},\nabla)$, where $\mathcal{E}$ is a filtered vector bundle on $S$ and $\nabla:\mathcal{E}\to\mathcal{E}\otimes\Omega^1_S$ is an integrable connection satisfying Griffiths transversality, in such a way that for all integers $j\geq 0$ there is a canonical isomorphism
\begin{equation}
\label{etaledeRhamcomp}
    \B_\dR\otimes_{\Z_p}H^j_\et(S_{\bar{\Q}_p},\mathbb{L})\cong \B_\dR\otimes_{\Q_p}H^j_\dR(S,\mathcal{E})\,,
\end{equation}
implying that
\[
\D_\dR(H^j_\et(S_{\bar{\Q}_p},\mathbb{L})):=H^0\big(\Q_p, \B_\dR\otimes_{\Z_p}H^j_\et(S_{\bar{\Q}_p},\mathbb{L})\big)\cong H^j_\dR(S,\mathcal{E})\,.
\]
Let us mention here two more features of this kind of results.
\begin{itemize}
    \item [(i)] One can extend such a result to a comparison between Kummer-étale cohomology and de Rham cohomology with poles along a divisor, considering the datum of $S$ as above together with a normal crossing divisor $D\subset S$. More precisely one can define a notion of \emph{de Rham} Kummer-étale $\Z_p$-local system on $(S,D)$ and, given such a local system $\mathbb{L}$, one can attach to it a pair $(\mathcal{E},\nabla)$, where $\mathcal{E}$ is a filtered vector bundle on $S$ and $\nabla$ is an integrable log connection (same as \ref{GaussManinX} above) satisfying Griffiths transversality (cf. \cite[Theorem 1.7]{DLLZ2023}). We will need this when $S=X_1(N)_{\Q_p}$ and $D=C_1(N)_{\Q_p}$ (the cusps).

    Moreover, the étale analogue of the isomorphism \ref{XYdeRhamiso} holds, i.e., there are  canonical identifications
    \begin{equation}
        H^i_\ket(X_1(N)_{\bar{\Q}_p}, \mathscr{F}_r(n))\cong H^i_\et(Y_1(N)_{\bar{\Q}_p},\mathscr{F}_r(n))
    \end{equation}
    for $\mathscr{F}\in\{\Hs,\Ls\}$, $r\in\Z_{\geq 0}$ and $n\in\Z$ (note that abusively we do not change the notation for the coefficients on the Kummer étale side).
    \item [(ii)] The whole picture extends also to the case of cohomology with compact support. We refer to \cite{LLZ2023} for this generalization.
\end{itemize}    
It is clear that, under the association $\mathbb{L}\mapsto (\mathcal{E},\nabla)$, the $\hat{\Z}_p$-local system $\HH^1(\Es)$ is sent to $(\Hs,\nabla)$, whence the choice of keeping the same notation for $\Hs_r$ and $\Ls_r$ for $r\in\Z_{\geq 0}$.
\end{rmk}

\medskip
In the sequel we will need the following facts concerning the Hecke action on cohomology groups.
\begin{itemize}
    \item [(i)] For $\mathscr{F}\in\{\Hs,\Ls\}$ one can define the action of Hecke operators $T_\ell$ for primes $\ell\nmid N$ and $U_\ell$ for $\ell\mid N$ and of dual Hecke operators $T'_\ell$ for $\ell\nmid N$ and $U'_\ell$ for $\ell\mid N$ on the cohomology groups $H^j_\et(Y_1(N)_R,\mathscr{F}_r)$ for any $\Q$-algebra $R$. We refer to \cite[Section 2.3]{BSVast1} (where $\Hs$ is denoted $\mathscr{S}$) for the precise definition of these operators (as usual they arise from the suitable Hecke correspondences).
    \item [(ii)]  One can also define, for every unit $d\in(\Z/N\Z)^\times$, a diamond operator $\langle d\rangle$ and a dual diamond operator $\langle d\rangle'$ on $H^j_\et(Y_1(N)_R,\mathscr{F}_r)$ for any $\Q$-algebra $R$. An Atkin--Lehner operator $w_N$ (and its dual $w'_N$) acts on $H^j_\et(Y_1(N)_R,\mathscr{F}_r)$ for any $\Q[\zeta_N]$-algebra $R$, where $\zeta_N$ is a fixed $N$-th root of unity in $\bar{\Q}$. We refer to \cite[Paragraph 2.3.1]{BSVast1} for more details.
    \item [(iii)] The action of Hecke and diamond operators can also be defined on compactly supported cohomology and on the corresponding de Rham cohomology groups.
    \item [(iv)] Assume that $N=N^\circ p^n$ with $p\nmid N^\circ$ and $n\in\Z_{\geq 1}$. We will also need Atkin--Lehner operators $w_{N^\circ}$ and $w_{p^n}$ acting on modular forms and more generally on the cohomology of $Y_1(N^\circ p^n)$ (where $p\nmid N^\circ$). 
    The action of $w_{N^\circ}$ on a cuspidal modular form $\xi\in S_k(\Gamma_1(N^\circ p^n))$ is given by
    \begin{equation}
    \label{wnoperator}
        w_{N^\circ}(\xi):=\langle 1 ; N^\circ\rangle (\xi|_k\,\omega_{N^\circ})\qquad \omega_{N^\circ}:=\omega_{N^{\circ}\!,p^n}:=\begin{pmatrix}N^{\circ} & -1 \\ N^{\circ}p^n c & N^{\circ}d\end{pmatrix}\,,
    \end{equation}
    where we require that $\det(\omega_{N^\circ})=N$ (for $c,d\in\Z$) and the diamond operator $\langle 1; N^\circ\rangle$ is the one corresponding to the unique element of $(\Z/N^\circ p^n)^\times$ which is congruent to $1$ modulo $N^\circ$ and to $N^\circ$ modulo $p^n$. Similarly, we define:
    \begin{equation}
    \label{wptoperator}
        w_{p^n}(\xi):=\langle p^n; 1\rangle(\xi|_k \omega_{p^n})\qquad \omega_{p^n}:=\omega_{p^n\!,N^\circ}:=\begin{pmatrix}p^n & -1 \\ N^\circ p^nc_n & p^nd_n\end{pmatrix}\,,
    \end{equation}
    where we require that $\det(\omega_{p^n})=p^n$ (for $c_n,d_n\in\Z$) and the diamond operator $\langle p^n; 1\rangle$ is the one corresponding to the unique element of $(\Z/N^\circ p^n)^\times$ which is congruent to $1$ modulo $p^n$ and to $p^n$ modulo $N^\circ$. 
    
    The operators $w_{N^\circ}$ and $w_{p^n}$ are the \emph{inverses} of the corresponding operators appearing in \cite{AL1978}. One can also define a geometric version of such operators (and of their duals $w'_{N^\circ}$ and $w'_{p^n}$) on $H^1_\et(Y_1(N)_R,\mathscr{F}_r)$ for $\mathscr{F}\in\{\Hs,\Ls\}$ and for any $\Q[\zeta_{N}]$-algebra $R$ (where $\zeta_{N}$ is a fixed primitive $N$-th root of unity in $\bar{\Q}$). We refer again \cite[Paragraph 2.3.1]{BSVast1} for more details.
    \item [(v)] The isomorphisms $s_r$ and Poincaré duality induce perfect pairings $\langle ~,~\rangle_r$:
    \begin{equation}
    \label{etalepairing}
    \begin{tikzcd}
    & H^1_\et(Y_1(N)_{\bar{\Q}},\Ls_r(1))_{\Q_p}\otimes_{\Q_p} H^1_{\et,c}(Y_1(N)_{\bar{\Q}},\Hs_r)_{\Q_p}\arrow[d, "\text{ $\langle ~,~\rangle_r$}"] \\
    & H^2_{\et,c}(Y_1(N)_{\bar{\Q}},\Z_p(1))_{\Q_p}\cong\Q_p\,.
    \end{tikzcd}    
    \end{equation}
    The operators $T_\ell$ and $T'_\ell$ (resp. $T'_\ell$ and $T_\ell$) are adjoint to each other under the pairing \eqref{etalepairing}. The same applies to the operators $U_\ell$ and $U'_\ell$ for $\ell\mid N$.
    \item [(vi)] After fixing an algebraic embedding $\Q_p\hookrightarrow\C$, the classical Eichler--Shimura isomorphism
    \begin{equation}
    \label{ESisoclass}
        H^1_\et(Y_1(N)_{\bar{\Q}},\Hs_r)\otimes_{\Z_p}\C\cong M_{r+2}(\Gamma_1(N))\oplus\overline{S_{r+2}(\Gamma_1(N))}
    \end{equation}
    commutes with the action of Hecke and diamond operators on both sides. If $N=N^\circ p^n$ with $p\nmid N^\circ$, the isomorphism \ref{ESisoclass} commutes also with the action of $w_{N^\circ}$ and $w_{p^n}$ on both sides.
\end{itemize}

\medskip
From now on in this section $L$ will denote a finite extension of $\Q_p$.

\begin{nota}
\label{Heckeisotypicalcomp}
If $V$ is a $\Q_p$-vector space, we write $V_L:=V\otimes_{\Q_p}L$ and if $M$ is a $\Z_p$-module, we write $M_L:=M\otimes_{\Z_p}L$.

For an $L$-vector space $H$ endowed with an action of the good Hecke operators of level $N$ and for $\xi\in S_\nu(N,\chi,L)$ an eigenform, we let $H[\xi]$ denote the $\xi$-Hecke isotypical component of $H$ (i.e., the maximal $L$-vector subspace where the Hecke operators act with the same eigenvalues as on $\xi$).  
\end{nota}

\begin{defi}
Given $\xi\in S_\nu(N,\chi,L)$ with $\nu\geq 2$ a normalized eigenform, one can attach to it two Galois representations.
\begin{itemize}
    \item [(i)] We let $V_{N}(\xi)$ be the maximal $L$-quotient of $H^1_{\et}(Y_1(N)_{\bar{\Q}},\Ls_{\nu-2}(1))_L$ where the dual good Hecke operators and dual diamond operators act with the same eigenvalues as those of $\xi$.
    \item [(ii)] We let $V^*_{N}(\xi):=H^1_\etc(Y_1(N)_{\bar{\Q}},\Hs_{\nu-2})_L[\xi]$ (cf. Notation \ref{Heckeisotypicalcomp}) be the maximal $L$-submodule of $H^1_\etc(Y_1(N)_{\bar{\Q}},\Hs_{\nu-2})_L$ where good Hecke and diamond operators act with the same eigenvalues as those of $\xi$.
    \end{itemize}
When the level $N$ is understood, we simply write $V(\xi)=V_N(\xi)$ and $V^*(\xi)=V^*_N(\xi)$.
\end{defi}

\begin{rmk}
If $\xi$ is new of level $N$, then $V^*_N(\xi)$ is identified with the $p$-adic Deligne representation associated to $\xi$ and $V_N(\xi)$ is identified with its dual. In general $V_N(\xi)$ (resp. $V_N^*(\xi)$ is non-canonically isomorphic to finitely many copies of $V_{N_\xi}(\xi^{\circ})$ (resp. $V^*_{N_\xi}(\xi^{\circ})$), where $\xi^{\circ}$ is the newform of level dividing $N$ associated to $\xi$ and $N_\xi\mid N$ is the corresponding level.
\end{rmk}

\begin{defi}
For an eigenform $\xi\in S_\nu(N,\chi,L)$ (with $\nu\geq 2$) and $?\in\{*,\emptyset\}$, we set
\[
V^?_{\dR,N}(\xi):=\D_\dR(V^?_{N}(\xi))=H^0(\Q_p,\B_\dR\otimes_L V^?_N(\xi))
\]
and we simply write $V^?_\dR(\xi)$ if the level $N$ is understood.   
\end{defi}

\medskip
The comparison isomorphism \eqref{etaledeRhamcomp} yields canonical isomorphisms
\begin{equation}
\label{fil0fil1}
    \Fil^0V_\dR(\xi)\cong S_\nu(\Gamma_1(N),L)[\xi^w]\qquad \Fil^1 V^*_\dR(\xi)\cong S_\nu(\Gamma_1(N),L)[\xi]
\end{equation}
(where $\xi^w:=w_{N}(\xi)$ and we follow Notation \ref{Heckeisotypicalcomp}) and a perfect duality
\begin{equation}
\label{dRduality}
\langle -,-\rangle_\xi: V_\dR(\xi)\otimes_L V^*_\dR(\xi)\to \D_\dR(L)=L\,,
\end{equation}
under which we get identifications
\begin{equation}
\label{vdrsufil1}
    V^*_\dR(\xi)/\Fil^1 V^*_\dR(\xi)\cong (S_\nu(\Gamma_1(N),L)[\xi^w])^\vee
\end{equation}
and
\begin{equation}
    V_\dR(\xi)/\Fil^0V_\dR(\xi)\cong (S_\nu(\Gamma_1(N),L)[\xi])^\vee\,,
\end{equation}
where $(-)^\vee$ denotes the $L$-dual of an $L$-vector space.

\subsection{The étale Abel--Jacobi map}
\label{abeljacobi}
In this section we fix a positive integer $M$ coprime to $p$ and a positive integer $t$ and we assume that $Mp^t\geq 5$. We consider a triple of cuspidal modular forms
\[
f=\sum_{n=1}^{+\infty}a_n(f)q^n,\qquad g=\sum_{n=1}^{+\infty}a_n(g)q^n,\qquad h=\sum_{n=1}^{+\infty}a_n(h)q^n
\]
with
\[
f\in S_k(Mp^t,\chi_f\omega^{2-k+k_0}\varepsilon_f),\qquad g\in S_l(Mp^t,\chi_g\omega^{2-l+l_0}\varepsilon_g),\qquad h\in S_m(Mp^t,\chi_h\omega^{2-m+m_0}\varepsilon_h).
\]
Here, $\omega$ is the Teichmüller character modulo $p$, $\chi_\xi$ is a character defined modulo $M$ and $\varepsilon_\xi$ is a character valued in $\mu_{p^\infty}$ for $\xi\in\{f,g,h\}$.
 
\begin{ass}
\label{balselfdual}
\begin{itemize}
\item [(i)] For $\xi\in\{f,g,h\}$, we assume that $\xi$ is a normalized eigenform, i.e., it holds $a_1(\xi)=1$ and $\xi$ is an eigenform for all the Hecke operators $T_\ell$ for all primes $\ell\nmid M$ (i.e., the good Hecke operators). We also assume that $\xi$ is an eigenform for the $U_p$ operator.
\item [(ii)] There exist a positive integer $M_1\mid M$ and a non-negative integer $s\leq t$ such that $f\in S_k(M_1p^s,\chi_f\omega^{2-k}\varepsilon_f,L)$ is a normalized $p$-ordinary newform of level $M_1p^s\geq 5$ or the ordinary $p$-stabilization of a $p$-ordinary newform of level $M_1\geq 5$.
\item [(iii)] The triple $(f,g,h)$ is \emph{tamely self-dual}, i.e., $\chi_f\chi_g\chi_h$ is the trivial character modulo $M$ and $k_0+m_0+l_0\equiv 0\mod (p-1)$.
\item [(iv)] The triple of weights $(k,l,m)$ is balanced and geometric, i.e., $(k,l,m)$ are the sizes of the edges of a triangle and $\nu\geq 2$ for $\nu\in\{k,l,m\}$.
\end{itemize}
\end{ass}
\medskip

\begin{nota}
\begin{itemize}
    \item [(i)]We fix a finite extension $L$ of $\Q_p$ containing the Fourier coefficients of $f,g,h$ (and a primitive $Mp^t$-th root of $1$) via the fixed embedding $\iota_p$.
    \item [(ii)] From now on in this section, we write $Y_t:=Y_1(Mp^t)_\Q$ (modular curve over $\Q$), with corresponding universal elliptic curve $u_t:\Es_t\to Y_t$.
\end{itemize}
\end{nota}

\begin{defi}
With the above notation, we set $r_1:=k-2$, $r_2:=l-2$, $r_3:=m-2$, $r:=(r_1+r_2+r_3)/2$, $\mathbf{r}:=(r_1,r_2,r_3)$. We also define the Dirichlet character of conductor a power of $p$ given by:
\[
\psi_{fgh}:=\omega^{(r_2+r_3-r_1-2k_0)/2}\cdot(\varepsilon_f^{-1}\varepsilon_g\varepsilon_h)^{-1/2}\,.
\]
We define $f':=f\otimes\omega^{k-2-k_0}\varepsilon_f^{-1}$ and $h':=h\otimes\psi_{fgh}$.
\end{defi}

\medskip
Applying \cite[Theorem 3.2]{AL1978}), we can also give the following definition.
\begin{defi}
\label{f'circdefi}
We let $f'^\circ$ denote the unique normalized newform of level dividing $M_1p^{2s}$ such that $f'=f\otimes\omega^{k-2-k_0}\varepsilon_f^{-1}=(f'^\circ)^{[p]}$ ($p$-depletion, i.e., $(f'^\circ)^{[p]}=f'^\circ- V_p\circ U_p (f'^\circ)$).    
\end{defi}

\begin{rmk}
\label{f'desc}
It holds that $f'^\circ\in S_k(M_1p^s,\chi_f\omega^{k-2-k_0}\varepsilon_f^{-1},L)$. More precisely:
\begin{itemize}
    \item [(i)] if $f$ is the ordinary $p$-stabilization of a newform $f^\circ$ of level $M_1$, then $f'^\circ=f^\circ$ (we have $k-k_0\equiv 2\mod p-1$ and $\varepsilon_f$ trivial in this case);
    \item [(ii)] if $f\in S_2(M_1p,\chi_f,L)$ is a newform, then $f'^\circ=f$;
    \item [(iii)] if $f$ is new of level $M_1p^s$ with $s\geq 1$ and it is $p$-primitive (i.e., the conductor of $\omega^{2-k+k_0}\varepsilon_f$ is exactly $p^s$), then $f'^\circ$ is the normalized eigenform given by a suitable multiple of $w_{p^s}(f)$ (this follows looking at the action of the good Hecke operators and knowing that $w_{p^s}(f)$ is new of level $M_1p^s$).
\end{itemize}
\end{rmk}

\begin{rmk}
\label{f'h'rmk}
Since we are not assuming that $g$ and $h$ are newforms, it is harmless to impose that $f$ is new of level $M_1p^s$ with $s\in\Z_{\geq 0}$ such that $s<t/2$. In this way, up to enlarging $t$, we can assume that $f'\in S_k(Mp^t,\chi_f\omega^{k-2-k_0}\varepsilon_f^{-1},L)$ and that $h'\in S_m(Mp^t,\chi_h\omega^{2-m+m_0}\varepsilon_h\psi_{fgh}^2)$ (i.e., that the $p$-part of the level is not increased by the twist). Note that if $\omega^{k-2-k_0}\varepsilon_f^{-1}$ (resp. $\psi_{fgh}$) is trivial, we want $f'=f^{[p]}$ (resp. $h'=h^{[p]}$).
\end{rmk}

\begin{defi}
We define
\begin{equation}
V(f,g,h):=V_{Mp^t}(f')\otimes_L V_{Mp^t}(g)\otimes_L V_{Mp^t}(h')(-1-r)      
\end{equation}
and
\begin{equation}
V^*(f,g,h):=V^*_{Mp^t}(f')\otimes_L V^*_{Mp^t}(g)\otimes_L V^*_{Mp^t}(h')(r+2)\,. 
\end{equation}
\end{defi}

\medskip
\begin{rmk}
\begin{itemize}
    \item [(i)] The representations $V(f,g,h)$ and $V^*(f,g,h)$ are Kummer self-dual by design and moreover they are canonically isomorphic to the Kummer dual of each other (essentially by the pairing \eqref{etalepairing} and by our choices of twists).
    \item [(ii)] Note that if the character of $f$ has trivial $p$-part (e.g. cases (i) and (ii) in Remark \ref{f'desc} above) and if we assume that the triple $(f,g,h)$ is self-dual (i.e., the product of the characters of the three forms is the trivial character), then $V^?_{Mp^t}(\xi')=V^?_{Mp^t}(\xi)$ for $\xi\in\{f,h\}$ and $?\in\{*,\emptyset\}$.
\end{itemize}
\end{rmk}

In \cite[Section 3]{BSVast1}, the authors associate to the triple $(f,g,h)$ a Galois cohomology class $\kappa(f,g,h)\in H^1(\Q,V(f,g,h))$. 

\medskip
Here is a diagram depicting the situation.

\begin{center}
    \begin{tikzcd}
   & \DET_\mathbf{r}^\et\in H^0_\et(Y_t,\Hs_\mathbf{r}(r))_L\arrow[dd, bend right = 90, mapsto]\arrow[r, "d_{t,*}"]\arrow[dr, "AJ_\et"]& H^4_\et(Y_t^3,\Hs_{[\mathbf{r}]}(r+2))_L
    \arrow[d, "\mathrm{HS}"] \\
    & H^1\big(\Q, H^3_\et(Y^3_{t,\bar{\Q}},\Ls_{[\mathbf{r}]}(2-r))_L\big)
    \arrow[d, twoheadrightarrow, "pr^{fgh}_\et"']& H^1\big(\Q, H^3_\et(Y^3_{t,\bar{\Q}},\Hs_{[\mathbf{r}]}(r+2))_L\big)\arrow[l, "s_\mathbf{r}", "\cong"'] \\
    & \kappa(f,g,h)\in H^1(\Q, V(f,g,h)) 
\end{tikzcd}
\end{center}

\begin{rmk}
\label{AJetrmk}
The following discussion explains the diagram above.
\begin{itemize}
    \item [(i)] $d_t: Y_t\hookrightarrow Y_t\times Y_t\times Y_t$ is the diagonal embedding and $d_{t,*}$ is the corresponding Gysin map.
    \item [(ii)] For $\mathscr{F}\in\{\Hs,\Ls\}$ and $\mathbf{r}=(r_1,r_2,r_3)$ as above, the $\Z_p$-local system $\mathscr{F}_\mathbf{r}$ on $Y_t$ is defined as
    \[    \mathscr{F}_\mathbf{r}:=\mathscr{F}_{r_1}\otimes_{\Z_p}\mathscr{F}_{r_2}\otimes_{\Z_p}\mathscr{F}_{r_3}
    \]
    and the $\Z_p$-local system $\mathscr{F}_{[\mathbf{r}]}$ on $Y_t^3$ is defined as
    \[        \mathscr{F}_{[\mathbf{r}]}:=p_1^*\mathscr{F}_{r_1}\otimes_{\Z_p}p_2^*\mathscr{F}_{r_2}\otimes_{\Z_p} p_3^*\mathscr{F}_{r_3}\,,
    \]
    where $p_j: Y_t\times Y_t\times Y_t\to Y_t$ is the natural projection on the $j$-th factor. In particular, it follows that $d_t^*\mathscr{F}_{[\mathbf{r}]}=\mathscr{F}_\mathbf{r}$.
    \item [(iii)] The map $\mathrm{HS}$ is a morphism coming from the Hochschild--Serre spectral sequence
    \begin{equation}
    \label{HSspectral}
        H^i\big(\Q,H^j_\et(Y_{t,\bar{\Q}}^3,\Hs_{[\mathbf{r}]}(r+2))\big)\Rightarrow H^{i+j}_\et(Y_t^3, \Hs_{[\mathbf{r}]}(r+2))\,.
    \end{equation}
    The existence of such a morphism is due to the fact that $H^4_\et(Y_{t,\bar{\Q}}^3,\Hs_{[\mathbf{r}]}(r+2))=0$ (by Artin vanishing theorem, as $Y_t$ is an affine curve).
    \item [(iv)] The étale Abel--Jacobi map in this setting is defined as $AJ_\et:=d_{t,*}\circ\mathrm{HS}$.
    \item [(v)] The isomorphism $s_\mathbf{r}$ is induced by the isomorphisms $s_{r_j}$ for $j\in\{1,2,3\}$ (cf. equation \eqref{isosr}).
    \item [(vi)] The projection $pr_\et^{fgh}$ is induced in Galois cohomology by viewing $V(f,g,h)$ as a quotient of $H^3_\et(Y^3_{t,\bar{\Q}},\Ls_{[\mathbf{r}]}(2-r))_L$ via Künneth formula and projection to the corresponding Hecke isotypical component.
\end{itemize}
\end{rmk}

\medskip
The class $\kappa(f,g,h)$ is the image under the composition $pr^{fgh}_\et\circ s_{\mathbf{r}}\circ AJ_\et$ of the element
    \begin{equation}
    \label{detrdef}
        \DET_\mathbf{r}^\et\in H^0_\et(Y_t,\Hs_\mathbf{r}(r))_L\,,
    \end{equation}
    which we now describe. Write $Y=Y_t$ till the end of this section. 
    
    We fix a geometric point $y:\Spec(\bar{\Q})\to Y$ and we consider the étale fundamental group $\pi_1^\et(Y,y)$. Passing to the stalk at $y$ induces an equivalence of categories
    \[
    \mathrm{Loc}_{Y}(\Z_p)\simeq \mathrm{Rep}_{\Z_p}^{cont}(\pi_1^\et(Y,y)),
    \]
    where: 
    \begin{itemize}
        \item [(i)] $\mathrm{Loc}_{Y}(\Z_p)$ is the category of étale $\Z_p$-local systems on $Y$;
        \item [(ii)] $\mathrm{Rep}_{\Z_p}^{cont}(\pi_1^\et(Y,y))$ is the category of continuous representations of $\pi_1^\et(Y,y)$ in finite free $\Z_p$-modules.
    \end{itemize}
    In particular, the stalk $ \TT_p(\Es)_y$ is a rank 2 free $\Z_p$-module. The $p$-adic Weil pairing $ \TT_p(\Es)_y\otimes_{\Z_p} \TT_p(\Es)_y\to \Z_p(1)$ is well-known to be perfect, $\Z_p$-bilinear, alternating and Galois-invariant (for the action of $G_{\Q}$). The construction recalled in Section \ref{appendixlinear} below applies to this setting (with $M= \TT_p(\Es)_y$), yielding an element
    \[
    \DET_\mathbf{r}\in H^0_{cont}(\Aut_{\Z_p}( \TT_p(\Es)_y), S_\mathbf{r}\otimes\Z_p[r])\,.
    \]
    The action of $\pi_1^\et(Y,y)$ on $ \TT_p(\Es)_y$ is encoded in a continuous morphism $\pi_1^\et(Y,y)\to \Aut_{\Z_p}( \TT_p(\Es)_y)$. We thus obtain a functor
    \[
(-)^\et:\mathrm{Rep}_{\Z_p}^{cont}(\Aut_{\Z_p}( \TT_p(\Es)_y))\to \mathrm{Rep}_{\Z_p}^{cont}(\pi_1^\et(Y,y))\simeq\mathrm{Loc}_{Y}(\Z_p)\,,\quad N\rightsquigarrow N^\et\,.
    \]
    It follows from the construction that, with the notation of the following Section \ref{appendixlinear}, $S_n^{\et}=\Hs_n$ for all integers $n\geq 0$ and $\Z_p[m]^\et=\Z_p(m)$ for all $m\in\Z$. As a consequence, we obtain an inclusion at the level of invariants
    \[
    H^0_{cont}(\Aut_{\Z_p}( \TT_p(\Es)_y), S_\mathbf{r}\otimes\Z_p[r])\hookrightarrow H^0_{cont}(\pi_1^\et(Y,y),\Hs_\mathbf{r}(r)_y)\cong H^0_\et(Y,\Hs_\mathbf{r}(r))\,.
    \]
    The element $\DET_\mathbf{r}^\et$ is defined as the image of the invariant $\DET_\mathbf{r}$ inside $H^0_\et(Y,\Hs_\mathbf{r}(r))_L$ via the above inclusion (followed by extension of scalars to $L$).
    
\subsection{Construction of the invariant \texorpdfstring{$\DET_\mathbf{r}$}{[DET]}}
\label{appendixlinear}
In this section, we let $R$ denote either a local ring or a principal ideal domain and we let $M$ denote a free $R$-module of rank $2$, with perfect alternating $R$-bilinear form
\[
\langle~,~\rangle_M: M\times M\to R\,.
\]
Set $G:=\Aut_R(M)$ and let $S:=\Hom_R(M,R)$ denote the $R$-dual of $M$, which we consider as a left-$R[G]$ module in the usual way, i.e.
\[
g*\lambda (v):=\lambda(g^{-1}(v))\qquad g\in G, \lambda\in S, v\in M\,.
\]
For $n\geq 0$ we let $S_n:=\Sym^n(S)$ denote the $n$-th symmetric power of $S$, defined as the maximal symmetric quotient of $S^{\otimes n}$. We view $S_n$ as a $R[G]$ module with the induced action.

For $m\in\Z$, we let $R[m]$ denote the ring $R$, with $G$-action via the $m$-th power of the determinant. We think of $R[m]$ as the free $R$-module of rank one, with fixed generator $a_m=1\in R$ such that $g*a_m=\det(g)^m\cdot a_m$ for all $g\in G$.

\medskip
In particular, the pairing $\langle~,~\rangle_M$ can be viewed as a perfect alternating $R$-bilinear form $\langle~,~\rangle_M:M\times M\to R[1]$, which is also $G$-equivariant, i.e. $\langle g*v_1,g*v_2\rangle_M=g*\langle v_1,v_2\rangle$.

\medskip
We fix a triple of non-negative integers $\mathbf{r}:=(r_1,r_2,r_3)$ such that:
\begin{itemize}
    \item [(i)] $r:=(r_1+r_2+r_3)/2\in\Z$;
    \item [(ii)] for every permutation of $\{i,j,k\}$ of $\{1,2,3\}$, it holds $r_i+r_j> r_k$. 
\end{itemize} 

We finally set $S_\mathbf{r}:= S_{r_1}\otimes_R S_{r_2}\otimes_R S_{r_3}$, which we endow with the structure of $R[G]$-module via the diagonal action on the three factors.

\medskip
The aim of this section is to produce a canonical invariant element
\[
\DET_\mathbf{r}\in H^0(G,S_\mathbf{r}\otimes R[r])\,.
\]

\medskip
In order to proceed more explicitly, we fix a symplectic $R$-basis $\{e_1,e_2\}$ of $M$ (i.e., $\langle e_1,e_2\rangle_M=1$, $\langle e_2,e_1\rangle_M=-1$). We let $\{\varepsilon_1,\varepsilon_2\}$ be the $R$-basis of $S$ which is dual to $\{e_1,e_2\}$.

In this way we identify $M=R\oplus R$ (column vectors), $S=R\oplus R$ (row vectors) and $G=\GL_2(R)$, with the natural action of $\GL_2(R)$ by matrix multiplication on column vectors on $M$ and the action of $\GL_2(R)$ on $S$ given by
\[
g*\lambda = \lambda\cdot g^{-1}\qquad g\in\GL_2(R),\lambda\in S\,.
\]
For every $n\geq 0$ we can then identify $S_n$ with the $R$-module of two-variable homogeneous polynomials of degree $n$ with $R$-coefficients, with $\GL_2(R)$-action given by
\[
g*P(x_1,x_2)= P\Big(\big(g^{-1}\cdot \begin{psmallmatrix}
     x_1 \\ x_2
\end{psmallmatrix}
\big)^t\Big)\,,
\]
where $(\cdot)^t$ denotes transposition.

\medskip
We view $S_\mathbf{r}$ as the $R$-module of polynomials in six variables $(x_1,x_2,y_1,y_2,z_1,z_2)$ with $R$-coefficients which are homogeneous of degree $r_1$ with respect to $(x_1,x_2)$, homogeneous of degree $r_2$ with respect to $(y_1,y_2)$ and homogeneous of degree $r_3$ with respect to $(z_1,z_2)$. The action of $\GL_2(R)$ on $S_\mathbf{r}$ can then be explicitly described in terms of the variables.

\medskip
We finally define an element $P_\mathbf{r}\in S_\mathbf{r}$ as follows
\[
P_\mathbf{r}(x_1,x_2,y_1,y_2,z_1,z_2):=\det\begin{pmatrix}
    x_1 & x_2 \\ y_1 & y_2
\end{pmatrix}^{r-r_3}\cdot \det\begin{pmatrix}
    x_1 & x_2 \\ z_1 & z_2
\end{pmatrix}^{r-r_2}\cdot
\det\begin{pmatrix}
    y_1 & y_2 \\ z_1 & z_2
\end{pmatrix}^{r-r_1}
\]
and one can easily check that for every $g\in\GL_2(R)$ it holds $g*P_\mathbf{r}=\det(g)^{-r}\cdot P_\mathbf{r}$, so that clearly $\DET_\mathbf{r}:=P_\mathbf{r}\otimes a_r\in S_\mathbf{r}\otimes R[r]$ is an invariant for the $G=\GL_2(R)$-action.

It is also clear that the element $\DET_\mathbf{r}$ depends only on $M$ and the pairing $\langle~,~\rangle_M$, but not on the choice of a symplectic basis.

\section{Some \texorpdfstring{$p$}{p}-adic Hodge theory}
\label{padicHTchapter}
In this section we introduce the necessary tools from $p$-adic Hodge theory and we apply them to the study of the Galois representation $V(f,g,h)$, seen as a representation of $G_{\Q_p}$. The main references are \cite[Section 1]{BLZ2016} and \cite{GM2009}.
\subsection{Filtered \texorpdfstring{$(\varphi,N)$}{(phi,N)}-modules and Galois representations}
\label{phiNsection}
For the generalities on $(\varphi,N)$-modules with coefficients we refer to \cite[Section 2]{GM2009}. As in the previous sections, we let $L$ be a finite and large enough extension of $\Q_p$ and we denote by $\Q_p^{nr}$ the maximal unramified extension of $\Q_p$ (inside the fixed algebraic closure $\bar{\Q}_p$) and we let $\sigma\in\Gal(\Q_p^{nr}/\Q_p)$ denote the Frobenius (i.e., the unique lift to $\Q_p^{nr}$ of the Frobenius automorphism $x\mapsto x^p$ of $\bar{\F}_p$).

\begin{defi}
\label{filteredphiNmodule}
A filtered $(\varphi,N,G_{\Q_p},L)$-module $D$ is a free $(\Q_p^{nr}\otimes_{\Q_p}L)$-module of finite rank endowed with:
\begin{itemize}
    \item [(i)] the Frobenius endomorphism: a $\sigma$-semilinear, $L$-linear, bijective map $\varphi:D\to D$;
    \item [(ii)] the monodromy operator: a $\Q_p^{nr}\otimes_{\Q_p}L$-linear, nilpotent endomorphism $N:D\to D$ such that $N\circ\varphi=p\cdot \varphi\circ N$;
    \item [(iii)] a $\sigma$-semilinear, $L$-linear action of $G_{\Q_p}$, commuting with $\varphi$ and $N$, and such that every element of $D$ has an open stabilizer for such action (i.e., it is a \emph{smooth} action);
    \item [(iv)] a decreasing, separated, exhaustive, $G_{\Q_p}$-stable filtration on $\bar{\Q}_p\otimes_{\Q_p^{nr}} D$ given by free $(\bar{\Q}_p\otimes_{\Q_p}L)$-submodules.
\end{itemize}
\end{defi}

There are similar definitions for filtered $(\varphi,N,G_F,L)$-modules, where $F$ is any finite extension of $\Q_p$ and one can view filtered $(\varphi,N,G_{\Q_p},L)$-modules as filtered $(\varphi,N,G_F,L)$-modules by restriction.

\begin{nota}
We fix a filtered $(\varphi,N,G_{\Q_p},L)$-module $D$ and we let $F$ be a finite extension of $\Q_p$. Denote by $F_0$ the maximal unramified subextension of $F$ and let $q=p^d$ be the cardinality of the residue field of $F$. We will always assume that $F\subseteq L$ in what follows. 

We set
\[
D_{\st,F_0}:=D^{G_F}\quad D_{\st,F}:=D^{G_F}\otimes_{F_0}F,\quad D_{\dR,F}:=(D\otimes_{\Q_p^{nr}}\bar{\Q}_p)^{G_F},\quad D_{\cris,F}=D_{\st,F}^{N=0}.
\]    
\end{nota}

\begin{ass}
Every $p$-adic Galois representation $V$ of $G_{\Q_p}$ appearing in the sequel will be a de Rham (equivalently, potentially semistable) representation.    
\end{ass}

\begin{defi}
\begin{itemize}
    \item [(i)] If $V$ is a $p$-adic Galois representation of $G_{\Q_p}$ with coefficients in $L$, $F$ is a finite extension of $\Q_p$ and $?\in\{\dR,\st,\cris\}$, we set
\[
\D_{?,F}(V):=(\B_?\otimes_{\Q_p}V)^{G_F}.
\] 
We say that $V$ is $F$-semistable (resp. $F$-crystalline) if 
$\D_{\st,F}(V)$ (resp. $\D_{\cris,F}(V)$) is a free $F_0\otimes_{\Q_p} L$-module of rank equal to $\dim_L(V)$.
    \item [(ii)] For $V$ as in (i), we also define the $(\Q_p^{nr}\otimes_{\Q_p}L)$-module attached to it as
    \begin{equation}
\label{Dpst}
    \D_\pst(V):=\varinjlim_{\substack{M/\Q_p \\ \text{finite}}} (\B_\st\otimes_{\Q_p}V)^{G_M}.
    \end{equation}
\end{itemize}
\end{defi}

\begin{rmk}
If $V$ is a $p$-adic Galois representation of $G_{\Q_p}$ with coefficients in $L$, the $(\Q_p^{nr}\otimes_{\Q_p}L)$-module $\D_\pst(V)$ inherits a structure of filtered $(\varphi,N,G_{\Q_p}, L)$-module. The functor $\D_\pst(-)$ provides an equivalence of categories between potentially semistable Galois representations with coefficients in $L$ and \emph{admissible} filtered $(\varphi,N,G_{\Q_p}, L)$-modules.
\end{rmk}

For later purposes, we define the so-called Bloch--Kato subspaces in first Galois cohomology group.
\begin{defi}
\label{BKsubspacesdef}
For $V$ a (de Rham) $p$-adic Galois representation of $G_{\Q_p}$ and $F$ a finite extension of $\Q_p$, one defines:
\begin{align*}
 H^1_g(F,V) & :=\Ker\big(H^1(F,V)\to H^1(F,\B_\dR\otimes_{\Q_p}V)\big)\,,\\
 H^1_f(F,V) & :=\Ker\big(H^1(F,V)\to H^1(F,\B_\cris\otimes_{\Q_p}V)\big)\,,\\
 H^1_e(F,V) & :=\Ker\big(H^1(F,V)\to H^1(F,\B_\cris^{\varphi=1}\otimes_{\Q_p}V)\big)\,.
\end{align*}    
\end{defi}

\medskip
One can study the cohomology of admissible filtered $(\varphi,N)$-modules and compare it with the Galois cohomology of the associated Galois representation.

\begin{defi}
Let $D$ be a filtered $(\varphi,N,G_{\Q_p},L)$-module and $F\subseteq L$ as above. The cohomology groups $H^i_\st(F,D)$ are given by the cohomology of the complex
\begin{equation}
\label{syntomiccomplex}
    \begin{tikzcd}[row sep =tiny, column sep = tiny]
    & C^\bullet_{\st,F}(D): &D_{\st,F_0}\oplus \Fil^0D_\dR\arrow[r] &D_{\st,F_0}\oplus D_{\st,F_0}\oplus D_{\dR,F}\arrow[r] & D_{\st,F_0} \\
    & & (u,v)\arrow[r, mapsto] &((1-\varphi)u, Nu, u-v) \\
    & & & (w,x,y)\arrow[r, mapsto] & Nw-(1-p\varphi)x
    \end{tikzcd} 
\end{equation}
concentrated in degrees $0,\,1$ and $2$.    
\end{defi}

\begin{teo}
\label{phiNGaloistheorem}
The complex $C^\bullet_{\st,F}(-)$ computes $\mathrm{Ext}$ groups $\Ext^i(\Q_p^{nr}\otimes_{\Q_p}L,-)$ for $i=0,1,2$ in the category of (\emph{admissible}) filtered $(\varphi, N, G_F,L)$-modules. If $V$ is a $p$-adic representation of $G_{\Q_p}$ with $L$-coefficients and $D=\D_\pst(V)$, the functor $\D_\pst$ induces functorial maps $H^i_\st(F,D)\to H^i(F,V)$ for every finite extension $F\subseteq L$ of $\Q_p$. These maps are isomorphisms for $i=0$ and injective for $i=1$. If $V$ is $F$-semistable, then for $i=1$ the image of the map $H^1_\st(F,D)\to H^i(F,V)$ coincides with the subspace $H^1_g(F,V)$ of Definition \ref{BKsubspacesdef}.
\begin{proof}
This is a well-known result. We refer to \cite[Remarks 2.6 and 2.7]{NN2016} and the references given therein for more details. Note that if $D=\D_\pst(V)$ with $V$ an $F$-semistable, then $H^1_\st(F,D)$ also classifies $L$-linear extensions of $L$ by $V$ which are potentially semistable. The statement concerning the image of the map $H^1_\st(F,D)\to H^i(F,V)$ follows easily recalling that $H^1_g(F,V)$ classifies $L$-linear extensions of $L$ by $V$ which are de Rham, equivalently potentially semistable.
\end{proof}
\end{teo}

\begin{defi}
\label{stBKexpdefi}
Let $V$ be an $F$-semistable representation of $G_{\Q_p}$ with $L$-coefficients. We define the semistable Bloch--Kato exponential map for $V$ as the isomorphism
\[
\exp_{\st,V}: H^1_\st(F,\D_\pst(V))\xrightarrow{\cong}H^1_g(F,V)
\]
afforded by Theorem \ref{phiNGaloistheorem}.
\end{defi}

\medskip
Now we describe slight generalizations/modifications of the complex $C^\bullet_{\st,F}(D)$ for $D$ a filtered $(\varphi,N, G_{\Q_p},L)$-module. We let $\Phi=\varphi^d$, so that $\Phi$ is $F_0$-linear on $D^{G_F}$ and extends to a linear endomorphism of $D_{\st,F}$.
For every polynomial $Q(T)\in 1+T\cdot L[T]$ one can define two variants of the above complex \eqref{syntomiccomplex} given by
\begin{equation}
\label{finitepolynomialphi}
\begin{tikzcd}[row sep =tiny, column sep = tiny]
    & C^\bullet_{\st,F,Q}(D): &D_{\st,F_0}\oplus \Fil^0D_\dR\arrow[r] &D_{\st,F_0}\oplus D_{\st,F_0}\oplus D_{\dR,F}\arrow[r] & D_{\st,F_0} \\
    & & (u,v)\arrow[r, mapsto] &(Q(\varphi)u, Nu, u-v) \\
    & & & (w,x,y)\arrow[r, mapsto] & Nw-Q(p\varphi)x
    \end{tikzcd}    
\end{equation}
and 
\begin{equation}
\label{finitepolynomialPhi}
\begin{tikzcd}[row sep =tiny, column sep = tiny]
    & \Tilde{C}^\bullet_{\st,F,Q}(D): &D_{\st,F}\oplus \Fil^0D_\dR\arrow[r] &D_{\st,F}\oplus D_{\st,F}\oplus D_{\dR,F}\arrow[r] & D_{\st,F} \\
    & & (u,v)\arrow[r, mapsto] &(Q(\Phi)u, Nu, u-v) \\
    & & & (w,x,y)\arrow[r, mapsto] & Nw-Q(q\Phi)x
    \end{tikzcd}    
\end{equation}
and define cohomology groups $H^i_{\st,Q}(F,D)$ and $\Tilde{H}^i_{\st,Q}(F,D)$ (for $i=0,1,2$) accordingly. In particular, note that $H^i_{\st,1-T}(F,D)=H^i_\st(F,D)$ in our notation.

\begin{rmk}
\label{inclH0}
Note that there is always a natural inclusion 
\[
H^0_{\st,1-T}(F,D)=D_{\st,F_0}^{\varphi=1}\cap\Fil^0 D_{\dR,F} \subseteq \Tilde{H}^0_{\st,1-T}(F,D)=D_{\st,F}^{\Phi=1}\cap\Fil^0 D_{\dR,F}.
\]
\end{rmk}

\begin{lemma}
\label{lemmast2}
Let $P_1,P_2\in 1+T\cdot L[T]$ be two polynomials. Then there is a natural morphism of complexes $C^\bullet_{\st,F,P_1}(D)\to C^\bullet_{\st,F,P_1P_2}(D)$ given by
\begin{center}
    \begin{tikzcd}[column sep = small]
    & C^\bullet_{\st,F,P_1}(D):\arrow[d, "c_{P_2}"] &D_{\st,F_0}\oplus \Fil^0D_{\dR,F}\arrow[r]\arrow[d, "id\oplus id"] &D_{\st,F_0}\oplus D_{\st,F_0}\oplus D_{\dR,F}\arrow[r]\arrow[d, "P_2(\varphi)\oplus id\oplus id"] & D_{\st,F_0}\arrow[d, "P_2(p\varphi)"] \\
    & C^\bullet_{\st,F,P_1P_2}(D): &D_{\st,F_0}\oplus \Fil^0D_{\dR,F}\arrow[r] &D_{\st,F_0}\oplus D_{\st,F_0}\oplus D_{\dR,F}\arrow[r] & D_{\st,F_0}
    \end{tikzcd}
\end{center}
which is a quasi-isomorphism if $P_2(\varphi)$ and $P_2(p\varphi)$ are bijective on $D_{\st,F_0}$.

Moreover the morphism $c_{P_2}$ always induces a short exact sequence of the form:
\begin{align*}
    0\to H^0_{\st,P_2}(F,D)\to D_{\st,F_0}^{P_2(\varphi)=0,\, N=0} &\to \Ker(H^1_{\st,P_1}(F,D)\to H^1_{\st,P_1P_2}(F,D))\to 0\\
    w &\mapsto [(w,0,0)]
\end{align*}
\begin{proof}
This is also an easy exercise. The only point where one has to be a slightly careful is to show that every class $[(w,x,y)]$ inside $\Ker(H^1_{\st,P_1}(F,D)\to H^1_{\st,P_1P_2}(F,D))$ can be represented as $[(w',0,0)]$ for a suitable $w'\in D_{\st,P_2}^{P_2(\varphi)=0,\, N=0}$.

But if  $(P_2(\varphi)w,x,y)=(P_1P_2(\varphi)u,Nu,u-v)$ for some $(u,v)\in D_{\st,F_0}\oplus\Fil^0D_{\dR,F}$, then $w'=w-P_1(\varphi)u$ does the job.
\end{proof}
\end{lemma}

\begin{defi}
\label{convenient}
An admissible $(\varphi,N,G_{\Q_p},L)$-module $D$ is $(F,Q)$-\textbf{convenient} if $D$ is $F$-crystalline (i.e., $D_{\dR,F}=D_{\st,F}=D_{\cris,F}$) and $Q(\Phi)$ and $Q(q\Phi)$ are bijective on $D_{\st,F}$.
\end{defi}
\begin{lemma}
\label{lemmast1}
In the above setting, if $D$ is $(F,Q)$-convenient, then the morphism of complexes $[\Fil^0D_{\dR,F}\to D_{\dR,F}\to 0]\to \Tilde{C}^\bullet_{\st,F,Q}(D)$ given by the obvious inclusions of $\Fil^0D_{\dR,F}$ inside $D_{\st,F}\oplus\Fil^0 D_{\dR,F}$ and of $D_{\dR,F}$ inside $D_{\st,F}\oplus D_{\st,F}\oplus D_{\dR,F}$ is actually a quasi-isomorphism. If, moreover, we consider $P(T)\in 1+T\cdot L[T]$ such that $P(T)\mid Q(T^d)$, then we can actually identify
\[
\frac{D_{\dR,F}}{\Fil^0D_{\dR,F}}\cong \Tilde{H}^1_{\st,Q}(F,D)\cong H^1_{\st,P}(F,D)\,.
\]
\begin{proof}
This is an easy exercise. One can check that the inverse to the isomorphism 
\[
\frac{D_{\dR,F}}{\Fil^0D_{\dR,F}}\xrightarrow{\cong} \Tilde{H}^1_{\st,Q}(F,D)
\]
is given by $[(w,x,y)]\mapsto y-Q(\Phi)^{-1}w\mod \Fil^0 D_{\dR,F}$. The identification $H^1_{\st,P}(F,D)\cong \Tilde{H}^1_{\st,Q}(F,D)$ follows immediately from Lemma \ref{lemmast2}, taking $P_1(T)=P(T)$ and $P_2(T)=Q(T^d)/P(T)$.
\end{proof}
\end{lemma}

\medskip
For $P,Q\in 1+T\cdot L[T]$, we let $P*Q\in 1+T\cdot L[T]$ be the polynomial whose roots are $\{\alpha_i\beta_j\}$ if $\{\alpha_i\}$ and $\{\beta_j\}$ are the roots of $P$ and $Q$ respectively. We then have the following proposition.

\begin{prop}
\label{cupprodst}
Let $D_1,D_2$ be two filtered $(\varphi,N,G_{\Q_p},L)$-modules. Then there are cup products
\[
C^\bullet_{\st,F,P}(D_1)\times C^\bullet_{\st,F,Q}(D_2)\to C^\bullet_{st,F,P*Q}(D_1\otimes_{\Q_p^{nr}\otimes_{\Q_p}L}D_2),
\]
associative and graded-commutative up to homotopy, hence inducing well-defined products on cohomology groups. The same applies to the $F$-lineared variants of such complexes, denoted $\Tilde{C}^\bullet_{\st,F,Q}(-)$ above.
\begin{proof}
See \cite[Proposition 1.3.2]{BLZ2016}.
\end{proof}
\end{prop}

\begin{rmk}
\label{tableremark}
Here we include the table (taken from \cite{BLZ2016}) giving the recipe to compute the pairing of the above Lemma \ref{cupprodst}. The cup product at the level of complexes is obtained summing up the various contributions.

\begin{equation}
\label{finitepolproduct}
\begin{tabular}{ |c||c|c|c|}
 \hline
  $\times$ & $(u',v')$ & $(w',x',y')$ & $z'$\\
  \hline\hline
 $(u,v)$& $(u\otimes u', v\otimes v')$& $\begin{psmallmatrix}
     b(\varphi_1,\varphi_2)(u\otimes w'),\\
     u\otimes x',\\
     (\lambda u +(1-\lambda)v)\otimes y'
 \end{psmallmatrix}$ & $b(\varphi_1,p\varphi_2)(u\otimes z')$ \\ 
 \hline
 $(w,x,y)$ & $\begin{psmallmatrix}
     a(\varphi_1,\varphi_2)(w\otimes u'),\\
     x\otimes u',\\
     y\otimes(\lambda v'+(1-\lambda)u')
 \end{psmallmatrix}$ & \begin{minipage}{4 cm}
   \centering{$-a(\varphi_1,p\varphi_2)(w\otimes x')$\\$+ b(p\varphi_1,\varphi_2)(x\otimes w')$}  
 \end{minipage} & 0\\
 \hline
$z$ & $a(p\varphi_1,\varphi_2)(z\otimes u')$ & $0$ & $0$\\
 \hline
\end{tabular}
\end{equation}

\medskip
In the table, we have fixed $a(X,Y), b(X,Y)\in L[X,Y]$ such that
\[
P*Q(XY)=a(X,Y)P(X)+b(X,Y)Q(Y)
\]
and $\lambda\in L$. One can check that changing the polynomials $a(X,Y)$ and $b(X,Y)$ or changing $\lambda$ will change the product by a chain homotopy (i.e., the induced pairings on cohomology groups are well-defined). A completely analogous table (replacing $p$ with $q=p^d$ and $\varphi_i$ with $\Phi_i$) describes the product in the $F$-linearized version of the semistable complexes.
\end{rmk}

\subsection{Study of the relevant filtered \texorpdfstring{$(\varphi,N)$}{[(phi,N)]}-modules}
\label{studyphiNsection}
Now we can go back to the setting of section \ref{abeljacobi} to study the filtered $(\varphi,N)$-modules attached to the modular forms $(f,g,h)$ more closely. Recall that 
\[
V^?_\dR(f,g,h):=\D_\dR(V^?(f,g,h))\,.
\]

We let $\breve{f}:=\lambda_{M_1}(f)^{-1}\cdot w_{M_1}(f)$, where $\lambda_{M_1}(f)$ is the pseudo-eigenvalue for the action of $w_{M_1}$ on $f$. It follows from \cite[theorem 1.1]{AL1978} that $\lambda_{M_1}(f)$ is an algebraic number of complex absolute value $1$. We denote by $e_{\breve{f}}$ the idempotent corresponding to $\breve{f}$ by the theory of $p$-stabilized ordinary newforms (cf. \cite[Chapter 4]{Hi1985a}).

\begin{defi}
\label{defidifferentials}
\begin{itemize}
    \item [(i)] For $\xi\in\{g,h'\}$ we let $\omega_\xi\in\Fil^1 V^*_{\dR, Mp^t}(\xi)$ be the element corresponding to $\xi$ under the isomorphism \eqref{fil0fil1}.
    \item [(ii)] Consider the $G_{\Q}$-representation $V_{Mp^t}^*(f')$. We let $\eta_{f'}\in V^*_{\dR,Mp^t}(f')/\Fil^1 V^*_{\dR,Mp^t}(f')$ be the element corresponding under the isomorphism \eqref{vdrsufil1} to the linear functional
\begin{center}
\begin{tikzcd}[row sep = tiny, column sep = small]
& S_k(\Gamma_1(Mp^t),L)[f'^w]\arrow[r] &S_k(M_1p^t,\chi_f^{-1}\omega^{2-k+k_0}\varepsilon_f,L)\arrow[r] & L\\
& \gamma\arrow[r, mapsto] & \mathrm{Tr}_{Mp^t/M_1p^t}(\gamma)\\
&&\delta\arrow[r, mapsto] & a_1(e_{\breve{f}}(\delta))
\end{tikzcd}
\end{center}
Note that this gives rise to a non-trivial functional since $\breve{f}\in S_k(\Gamma_1(Mp^t),L)[f'^w]$.
\end{itemize}
\end{defi}

\begin{rmk}
The fact the such a linear functional actually takes values in $L$ follows from the work of Hida (cf. \cite[Proposition 4.5]{Hi1985a}).
\end{rmk}

\begin{rmk}
\label{F1crys}
Since $f$ is $p$-ordinary, we know that $V_{Mp^t}^*(f)$ is $F_1$-semistable, where $F_1$ is a cyclotomic extension of $\Q_p$ generated by a $p^n$-th root of unity for some $n\geq 0$. More precisely (cf. Remark \ref{f'desc}):
\begin{itemize}
    \item [(i)] if $f$ is the ordinary $p$-stabilisation of a newform $f^\circ$ of level $M_1$, then $V_{Mp^t}^*(f)$ is already crystalline over $\Q_p$;
    \item [(ii)] if $f\in S_2(M_1p,\chi_f,L)$ is a newform, then $V_{Mp^t}^*(f)$ is semistable (but not crystalline) over $\Q_p$;
    \item [(iii)] if $f$ is new of level $M_1p^s$ and $p$-primitive, then $V_{Mp^t}^*(f)$ becomes crystalline over $\Q_p(\zeta_{p^s})$ (where $\zeta_{p^s}$ is a primitive $p^s$-th root of unity).
\end{itemize}
The same remarks apply to $V_{Mp^t}^*(f')$.   
\end{rmk}

\begin{lemma}
\label{decompeta}
The filtered $L$-vector space $V^*_{\dR, Mp^t}(f')$ decomposes, as $L$-vector space, as
\[
V^*_{\dR, Mp^t}(f')=\Fil^1 V^*_{\dR, Mp^t}(f')\oplus\Big(F_1\otimes_{\Q_p}\D_{\st,F_1}\big(V^*_{Mp^t}(f')
\big)^{\varphi=a_p(f)}\Big)^{\Gal(F_1/\Q_p)}
\]
\begin{proof}
In order to simplify the notation, we write $V=V^*_{Mp^t}(f')$, $V_\dR=V^*_{\dR, Mp^t}(f')$ and $V_{\st,F_1}=\D_{\st,F_1}\big(V^*_{Mp^t}(f')
\big)$ in this proof.

It is possible to describe explicitly the structure of filtered $(\varphi,N,\Gal(F_1/\Q_p),L)$-module of $V_{\st,F_1}$. We refer to \cite[Sections 3.1 and 3.2]{GM2009}, where the authors actually describe the duals of our modules.

Combining everything, we can give the following explicit description of $V_{\st,F_1}$. Such module is non-canonically isomorphic to a finite number of copies of a two-dimensional filtered $(\varphi,N,\Gal(F_1/\Q_p),L)$-module with Hodge-Tate weights $\{1-k,0\}$, which we denote by $D$. Clearly, it is enough to prove the corresponding statement of the lemma for $D$.

\medskip
If $V$ is $F_1$-crystalline, then $D$ has a basis $\{e_1,e_2\}$ as $L$-vector space such that
\[
\left\{
\begin{aligned}
    &\varphi(e_1) && = &&& \chi_f(p)p^{k-1}a_p(f)^{-1}\cdot e_1\\
    & \varphi(e_2) && = &&& a_p(f)\cdot e_2\\
    & \Fil^1 D && = &&& (F_1\otimes_{\Q_p} L)(x_f e_1+ y_f e_2)\\
    & N && = &&& 0\\
    & g(e_1) && = &&&  \big(\omega^{2-k}\varepsilon_f\big)(g)\cdot e_1\\
    & g(e_2) && = &&& e_2\\
    & g && \in &&& \Gal(F_1/\Q_p)
    \end{aligned}
    \right.
\]
where $x_f,y_f\in F_1\otimes_{\Q_p}L$ are elements such that either $y_f=0$ and $x_f\neq 0$ (the \emph{split} case) or they are both non-zero (the \emph{non-split} case). These elements are unique up to multiplication by elements of $L^\times$.

\medskip
If $V$ is not $F_1$-crystalline (i.e., $f\in S_2(M_1p,\chi_f)$ newform, under our assumptions), then we can choose $F_1=\Q_p$ and $D$ has a basis $\{e_1,e_2\}$ as $L$-vector space
\[
\left\{
\begin{aligned}
    &\varphi(e_1) && = &&&  p\cdot a_p(f)\cdot e_1\\
    & \varphi(e_2) && = &&& a_p(f)\cdot e_2\\
    & \Fil^1 D && = &&& (F_1\otimes_{\Q_p} L) (e_1 -\mathfrak{L} e_2)\\
    & N(e_1) && = &&& e_2\\
    & N(e_2) && = &&& 0
    \end{aligned}
    \right.
\]
where $\mathfrak{L}=\mathfrak{L}_p(f)$ is the $\mathfrak{L}$-invariant of $f$ (defined as in \cite{Mazur1994}).

\medskip
From the explicit description of the filtration on $D$, it follows easily in all cases that
\[
\Fil^1(D)\cap \big(F_1\otimes_{\Q_p} D^{\varphi=a_p(f)}\big)=\{0\}
\]
and that $D^{\varphi=a_p(f)}$ is one-dimensional (over $L$). We thus get a decomposition of $F_1\otimes_{\Q_p}D$ as $L$-vector space given by
\[
F_1\otimes_{\Q_p}D=\Fil^1 D\oplus \big(F_1\otimes_{\Q_p} D^{\varphi=a_p(f)}\big)
\]
and it follows that such a decomposition is stable for the action of $\Gal(F_1/\Q_p)$ (cf. the discussion in \cite[Section 3.2]{GM2009}) and gives an analogous decomposition for $F_1\otimes_{\Q_p}V_{\st,F_1}$. 

Taking $\Gal(F_1/\Q_p)$-invariants yields
\begin{align*}
V_\dR & =\big(F_1\otimes_{\Q_p}V_{\st,F_1}\big)^{\Gal(F_1/\Q_p)}=\\
 & =\big(\Fil^1 V_{\dR,F_1}\big)^{\Gal(F_1/\Q_p)}\oplus \big(F_1\otimes_{\Q_p} D^{\varphi=a_p(f)}\big)^{\Gal(F_1/\Q_p)} =\\
& = \Fil^1 V_\dR \oplus \big(F_1\otimes_{\Q_p} D^{\varphi=a_p(f)}\big)^{\Gal(F_1/\Q_p)}\, ,
\end{align*}
as we wished to prove.
\end{proof}
\end{lemma}

\medskip
\begin{rmk}
We observe that the decomposition proved in the above Lemma \ref{decompeta} does not depend on the choice of $F_1\subset \Q_p(\mu_{p^\infty})$ (Galois over $\Q_p$) such that $V^*_{Mp^t}(f')$ is $F_1$-semistable, i.e., if $F_2$ is another such extension we can identify
\[
\Big(F_1\otimes_{\Q_p} \D_{\st,F_1}\big(V^*_{Mp^t}(f')\big)^{\varphi=a_p(f)}\Big)^{\Gal(F_1/\Q_p)}=\Big(F_2\otimes_{\Q_p} \D_{\st,F_2}\big(V^*_{Mp^t}(f')\big)^{\varphi=a_p(f)}\Big)^{\Gal(F_2/\Q_p)}.
\]
This follows easily from the fact that the action of $\varphi$ and the Galois action commute and from Hilbert Theorem 90.
\end{rmk}

\medskip
\begin{defi}
\label{etafrobp}
With the notation introduced above, we define
\[
\eta_{f'}^{\varphi=a_p}\in\Big(F_1\otimes_{\Q_p} \D_{\st,F_1}\big(V^*_{Mp^t}(f')\big)^{\varphi=a_p(f)}\Big)^{\Gal(F_1/\Q_p)}\subset V^*_{\dR, Mp^t}(f')
\]
as the unique lift of $\eta_{f'}$ to the subspace $\Big(F_1\otimes_{\Q_p} \D_{\st,F_1}\big(V^*_{Mp^t}(f')\big)^{\varphi=a_p(f)}\Big)^{\Gal(F_1/\Q_p)}$.
\end{defi}

\begin{rmk}
Since the triple $(k,l,m)$ is balanced, we have that 
\begin{equation}
\label{testtriplecohom}
\eta_{f'}^{\varphi=a_p}\otimes\omega_g\otimes\omega_{h'}\otimes t_{r+2}\in\Fil^0(V^*_\dR(f,g,h)).  
\end{equation}
Here, for all $n\in\Z$, $t_n$ denotes a (canonical) generator of $\D_\dR(\Q_p(n))$ (on which the Frobenius $\varphi$ acts as multiplication by $p^{-n}$).    
\end{rmk}

\medskip
The idea is now to associate to $\kappa(f,g,h)$ an element in $\Fil^0(V_\dR(f,g,h))^\vee$, so that we can pair it with $\eta_{f'}^{\varphi=a_p}\otimes\omega_g\otimes\omega_{h'}\otimes t_{r+2}$. We are led to study the Bloch--Kato local conditions for the Galois representation $V(f,g,h)$. 

\begin{defi}
\label{expconvtriple}
A triple $(f,g,h)$ satisfying Assumption \ref{balselfdual} is called $F$-\textbf{exponential} if the equality
\[
H^1_e(F,V(f,g,h))=H^1_f(F,V(f,g,h))=H^1_g(F,V(f,g,h))
\]
holds (for an appropriate $F$ depending on $(f,g,h)$).

We say that a triple $(f,g,h)$ satisfying Assumption \ref{balselfdual} is $(F,Q)$-\textbf{convenient} for a finite extension $F/\Q_p$ and $Q\in 1+ T\cdot L[T]$ if the $(\varphi,N,G_{\Q_p},L)$-module $\D_\pst(V(f,g,h))$ is $(F,Q)$-convenient, in the sense of definition \ref{convenient}.
\end{defi}

\begin{rmk}
It is easy to check that if $(f,g,h)$ is $(F,1-T)$-convenient, then it is $F$-exponential.
\end{rmk}

As already explained in the introduction, in the sequel we will be mostly interested in the following setting.

\begin{ass}
\label{cuspasslocp}
The forms $g$ and $h$ are supercuspidal at $p$ and lie in the kernel of $U_p$. In particular (since $p$ is odd), the Galois representations $V_{Mp^t}(g)$ and $V_{Mp^t}(h)$, seen as local Galois representations of $G_{\Q_p}$, are isomorphic to (finitely many copies of) the induced representations of a character of a quadratic extension of $\Q_p$ (which is not the restriction of a character of $G_{\Q_p}$). 
\end{ass}

\begin{rmk}
Asking that $g$ and $h$ are supercuspidal at $p$ implies directly that $g$ and $h$ lie in the kernel of $U_p^m$ for some $m$ large enough, since we are not assuming that $g$ and $h$ are new of level $Mp^t$. The slightly stronger Assumption \ref{cuspasslocp} will be needed to simplify some arguments in what follows.
\end{rmk}

\begin{rmk}
\label{phigammagh}
As explained in \cite[Sections 3.3 and 3.4]{GM2009}, under our Assumption \ref{cuspasslocp} we can find a finite Galois extension of $\Q_p$, which we denote by $F$, such that:
\begin{itemize}
    \item [(i)] the maximal unramified subextension of $F$ is of the form $F_0:=\Q_{p^{2a}}$ for some $a\in\Z_{\geq 1}$;
    \item [(ii)] $F$ contains the cyclotomic extension $F_1$ of Remark \ref{F1crys};
    \item [(iii)] $F$ is contained in our field of coefficients $L$ (up to extending $L$ if necessary);
    \item [(iv)] $V_{Mp^t}(g)$, $V_{Mp^t}(h)$ and $V_{Mp^t}(h')$ are $F$-crystalline.
\end{itemize}

Moreover, one can describe the filtered $(\varphi,N,\Gal(F/\Q_p),L)$-modules associated to $V_{Mp^t}(\xi)$ for $\xi\in\{g,h,h'\}$ as follows (always relying on \cite[Sections 3.3 and 3.4]{GM2009}). Such modules are given by (finitely many copies of) a rank $2$ free $F_0\otimes_{\Q_p}L$-module $D_\xi$ with basis $\{v_\xi,w_\xi\}$ and such that
\[
\left\{
\begin{aligned}
    &\varphi(v_\xi) && = &&& \mu_\xi\cdot v_\xi \\
    & \varphi(w_\xi) && = &&& \mu_\xi\cdot w_\xi\\
    & N && = &&& 0\\
    & \mu_\xi && \in  &&&L^\times\\
    &\ord_p(\mu_\xi) && = &&&\tfrac{1-\nu}{2}
        \end{aligned}
    \right.
\]
where $\nu$ denotes the weight of the form $\xi$. Note that $\mu_h=\mu_{h'}$ since $h'$ is a twist of $h$ by a character of order a power of $p$.  

\medskip
One could also give a more explicit description of the Galois group $\Gal(F/\Q_p)$ and of its action on such modules, but we will not need it for the moment. The same remark applies to the filtration on $F\otimes_{F_0}D_\xi$.

\medskip
Recall that the Frobenius endomorphism $\varphi$ is $F_0$-semilinear, so that if we want to look at it as a linear operator on $D_\xi$, we have to view $D_\xi$ as an $L$-vector space of dimension $4a$.    
\end{rmk}

\begin{prop}
\label{efgequal}
Under Assumptions \ref{balselfdual} and \ref{cuspasslocp}, we have that:
\begin{itemize}
    \item [(a)] if the form $f$ has weight $k>2$, then the triple $(f,g,h)$ is $(F,1-T)$-convenient (for $F$ as in Remark \ref{phigammagh});
    \item [(b)] if the form $f$ is a newform in $S_2(M_1p,\chi_f,L)$, then the triple $(f,g,h)$ is $F$-exponential.
\end{itemize}
\begin{proof}
The inclusions
\[
H^1_e(F,V(f,g,h))\subseteq H^1_f(F,V(f,g,h))\subseteq H^1_g(F,V(f,g,h))
\]
are always true. Since $V(f,g,h)$ is Kummer self-dual (by design), in order to show that $(f,g,h)$ is $F$-exponential, it is enough to prove that $H^1_e(F,V(f,g,h))= H^1_f(F,V(f,g,h))$. As a consequence of \cite[corollary 3.8.4]{BK1990}, we are then reduced to show that
\[
\D_{\cris,F}(V(f,g,h))^{\varphi=1}=\{0\}\, .
\]
We will write
\[
D_{fgh}:=D_{f'}\otimes_{\tilde{L}}D_g\otimes_{\tilde{L}} D_{h'}\otimes_{\tilde{L}} (\tilde{L}\cdot t_{-1-r})\,,
\]
where $\tilde{L}:=F_0\otimes_{\Q_p}L$ and recall that $r=(k+l+m-6)/2$ and that $\varphi(t_{-1-r})=p^{r+1} t_{-1-r}$.

Here $D_{f'}$ is, up to extending scalars to $F_0$, the dual of the module $D$ appearing in the proof of Lemma \ref{decompeta}. We can fix an $\tilde{L}$-basis of $D_{f'}$ given by $\{v_f,w_f\}$ (essentially passing to the dual basis of the one described in the proof of Lemma \ref{decompeta}), such that
\[
\varphi(v_f)=\chi_f(p)^{-1}p^{1-k}a_p(f)\cdot v_f\, ,\qquad \varphi(w_f)=a_p(f)^{-1}\cdot w_f\,,\qquad N=0
\]
if $V(f)$ is $F_1$-crystalline and
\[
\varphi(v_f)=a_p(f)^{-1}\cdot v_f\, ,\quad\varphi(w_f)=(p a_p(f))^{-1}\cdot w_f\, ,\qquad N(v_f)=-w_f\,\quad N(w_f)=0
\]
if $f$ is a newform in $S_2(M_1p,\chi_f,L)$.

\medskip
Since $\D_{\st,F}(V(f,g,h))$ is isomorphic to a finite direct sum of copies of $D_{fgh}$, it is enough to prove that $D_{fgh}^{\varphi=1,N=0}=\{0\}$. We will look at $D_{fgh}$ as an $L$-vector space of dimension $16a$ where $\varphi$-acts $L$-linearly.

We fix $\alpha\in L$, a primitive $p^{2a}-1$-th root of $1$, so that $F_0=\Q_p(\alpha)$ and the arithmetic Frobenius $\sigma$ (i.e., the generator of $\Gal(F_0/\Q_p)$ inducing the arithmetic Frobenius modulo $p$) acts on $\alpha$ simply as $\sigma(\alpha)=\alpha^p$.

By our previous discussion we know that a basis $\mathscr{B}$ of $D_{fgh}$ as $L$-vector space can be described as follows:
\[
\mathscr{B}:=\big\{(\alpha^{p^j}\otimes 1)\cdot e_\tau \mid j=0,\dots,2a-1,\quad \tau\in\{v,w\}^{ \{1,2,3\}} \big\}
\]
where for $\tau:\{1,2,3\}\to\{v,w\}$ we let
\[
e_\tau=\tau(1)_f\otimes\tau(2)_g\otimes\tau(3)_{h'}\otimes t_{-1-r}\,.
\]

\medskip
Now we prove assertion (a) in the statement. Since $k>2$, $V_{Mp^t}(f)$ is $F$-crystalline (so that $V(f,g,h)$ is $F$-crystalline) and we have
\[
\varphi((\alpha^{p^j}\otimes 1)\cdot e_\tau):=\begin{cases}
\alpha^{(p-1)p^j}\cdot p^{r+2-k}\chi_f(p)^{-1}\cdot a_p(f)\cdot\mu_g\mu_{h'}\cdot((\alpha^{p^j}\otimes 1)\cdot e_\tau) &\text{ if }\tau(1)=v,\\
\alpha^{(p-1)p^j}\cdot p^{r+1}\cdot a_p(f)^{-1}\cdot \mu_g\mu_{h'}\cdot((\alpha^{p^j}\otimes 1)\cdot e_\tau)& \text{ if }\tau(1)=w.
\end{cases}
\]
If we let $\lambda_{j,\tau}$ to be the $\varphi$-eigenvalue relative to $(\alpha^{p^j}\otimes 1)\cdot e_\tau$ as described above, we can immediately compute that
\[
|\lambda_{j,\tau}|_p=\begin{cases}
    p^{k/2} &\text{ if }\tau(1)=v,\\
    p^{1-k/2} &\text{ if }\tau(1)=w.
\end{cases}
\]
so that if we assume $k>2$ it cannot happen that $\lambda_{j,\tau}=1$ or $\lambda_{j,\tau}=p^{-1}$. We have thus proven that $(f,g,h)$ is $(F,1-T)$-convenient (and in particular $F$-exponential).

\medskip
For assertion (b), note that $N((\alpha^{p^j}\otimes 1)\cdot e_\tau)=0$ if and only if $\tau(1)=w$. Assuming $\tau(1)=w$ we get
\[
\varphi((\alpha^{p^j}\otimes 1)\cdot e_\tau):=\alpha^{(p-1)p^j}\cdot p^r\cdot a_p(f)^{-1}\cdot \mu_g\mu_{h'}\cdot((\alpha^{p^j}\otimes 1)\cdot e_\tau)\,.
\]
and one checks that
\[
|\alpha^{(p-1)p^j}\cdot p^r\cdot a_p(f)^{-1}\cdot \mu_g\mu_{h'}|_p=p,
\]
so that also in this case $D_{fgh}^{\varphi=1,N=0}=\{0\}$ and the proof is complete.
\end{proof}
\end{prop}

\begin{rmk}
\label{princseriescase}
The triple $(f,g,h)$ is $(F,1-T)$-convenient also when the following condition is satisfied:
\begin{itemize}
\item [(PS)] the newforms of level dividing $Mp^t$ attached to $g$ and $h$ are $p$-primitive newforms (i.e., the power of $p$ dividing their level and the conductor of their character is the same).
\end{itemize} 
Equivalently, we are assuming that the newforms attached to $g$ and $h$ are either unramified principal series at $p$ (if their level is coprime to $p$) or they are ramified principal series at $p$ and have finite $p$-slope (otherwise).

In this case, the proof relies on the control of the complex absolute values of the eigenvalues of Frobenius (following from the Ramanujan--Petersson conjecture) and one can fix $F\subset\Q_p(\mu_{p^\infty})$ a finite, totally ramified, cyclotomic extension of $\Q_p$ generated by a $p^n$-th root of unity (for some $n\in\Z_{\geq 0}$ large enough) - so that in particular $F_0=\Q_p$. The proof is identical to that of \cite[Lemma 3.5]{BSVast1} and the proof of Lemma \ref{admissiblelemma} below contains (part of) the required computations, so we omit the details here.
\end{rmk}

When $(f,g,h)$ is $F$-exponential, \cite[Corollary 3.8.4]{BK1990} shows that the Bloch--Kato exponential map
\begin{equation}
\label{BKexpF}
\exp_{\BK,F}: \frac{\D_{\dR,F}(V(f,g,h))}{\Fil^0 \D_{\dR,F}(V(f,g,h))}\to H^1_e(F,V(f,g,h))=H^1_g(F,V(f,g,h))    
\end{equation}
is an isomorphism.

\medskip
The perfect duality \eqref{dRduality} induces a perfect duality
\begin{equation}
\label{dRdualityfgh}
\langle -,-\rangle_{fgh}: V_\dR(f,g,h)\otimes_L V^*_\dR(f,g,h)\to \D_\dR(L(1))=L,
\end{equation}
under which one has an identification
\[
V_\dR(f,g,h)/\Fil^0(V_\dR(f,g,h))\cong\Fil^0(V^*_\dR(f,g,h))^\vee.
\]

Moreover, note that we have identifications
\[
\D_{\dR,F}(V(f,g,h))=F\otimes_{\Q_p}V_\dR(f,g,h)\,,\qquad \Fil^0 \D_{\dR,F}(V(f,g,h))=F\otimes_{\Q_p}\Fil^0 V_\dR(f,g,h)
\]
yielding an isomorphism
\begin{equation}
\label{derhamid}
\frac{\D_{\dR,F}(V(f,g,h))}{\Fil^0 \D_{\dR,F}(V(f,g,h))}\cong F\otimes_{\Q_p}\frac{V_\dR(f,g,h)}{\Fil^0(V_\dR(f,g,h))}.    
\end{equation}

\medskip
\begin{defi}
Let $(f,g,h)$ be an $F$-exponential triple (for some finite Galois extension $F/\Q_p$). We define the Bloch--Kato logarithm
\[
\log_\BK^{fgh}: H^1_g(\Q_p,V(f,g,h))\to \Fil^0(V^*_\dR(f,g,h))^\vee
\]
as the following composition (we write $V=V(f,g,h)$ for simplicity):
\begin{center}
\begin{tikzcd}[column sep = huge]
H^1_g(\Q_p,V)\arrow[r, "\res_{F/\Q_p}", "\cong"'] & H^1_g(F,V)^{\Gal(F/\Q_p)}\arrow[r, "\tfrac{1}{[F:\Q_p]}\cdot \exp_{\BK,F}^{-1}"]&\frac{V_\dR}{\Fil^0(V_\dR)}\cong\Fil^0(V^*_\dR)^\vee\,. 
\end{tikzcd}
\end{center}
\end{defi}

\begin{rmk}
\label{BKremark}
The first isomorphism $\res_{F/\Q_p}$ above follows from the Hochschild--Serre spectral sequence in continuous group cohomology and the fact that the representations involved are vector spaces over characteristic zero fields ($\B_\dR$ is a field). Note that the map $\log_{BK}^{fgh}$ is actually independent of $F$ (i.e., if $F'$ is another finite Galois extension of $\Q_p$ such that $V(f,g,h)$ is $F'$-exponential, we obtain the same map via our construction). Note, moreover, that we are crucially using the isomorphism \eqref{derhamid} and the fact that the exponential map \eqref{BKexpF} is equivariant for the action of $\Gal(F/\Q_p)$ (so that it induces and isomorphism between the Galois invariants on both sides).

\medskip
In the next section, we will show that $\kappa(f,g,h)\in H_g(\Q_p,V(f,g,h))$, so that one can indeed view $\log_\BK^{fgh}(\kappa(f,g,h))$ as a linear functional on $\Fil^0(V^*_\dR(f,g,h))$ and study its value at 
\[
\eta_{f'}^{\varphi=a_p}\otimes\omega_g\otimes\omega_{h'}\otimes t_{r+2}\in\Fil^0(V^*_\dR(f,g,h))\,.
\]
\end{rmk}


\section{The syntomic Abel--Jacobi map}
\label{syntomicabeljacobichapt}
This section develops the formalism of syntomic and finite polynomial cohomology, with the aim of giving a syntomic description of the Abel--Jacobi map in our setting. We also recall the needed facts about the Coleman's study of the geometry of modular curves as rigid analytic spaces. 

\subsection{Nekov\'{a}\v{r}--Nizio\l{} syntomic cohomology}
\label{NNsyntomicsection}
In this section we recollect some facts concerning syntomic (and finite-polynomial) cohomology for varieties over $p$-adic fields. 

We will have to consider recent extensions of the original theory for varieties over $p$-adic fields admitting a smooth integral model (due, among others, to Besser, cf. \cite{Bes2000}), since the modular curves $X_1(Mp^t)$ (for $t\geq 1$) only admit a semistable model over the ring of integers of a (large enough) finite extension of $\Q_p$.

\medskip
As in the previous sections, let $F$ denote a fixed finite extension of $\Q_p$, with ring of integers $\OO_F$ and residue field $k_F$ of cardinality $q=p^d$. We fix a choice of a uniformizer $\varpi\in\OO_F$. Let $F_0=W(k_F)[1/p]$ denote the maximal unramified subextension of $\Q_p$ inside $F$.

\medskip
Let $X$ be any $F$-variety (i.e., a reduced separated $F$-scheme of finite type). For every $n\in\Z$, Nekov\'{a}\v{r}--Nizio\l{} (\cite{NN2016}) define complexes $R\Gamma_{\NN}(X,n)$ and $R\Gamma_{\NN,c}(X,n)$ which compute syntomic cohomology, generalizing Besser's work. The main properties of this cohomology theory are collected in \cite[Theorem A]{NN2016}.  

\medskip
In \cite{BLZ2016}, the authors extend the constructions of Nekov\'{a}\v{r}--Nizio\l{} and define cohomology theories $R\Gamma_{\NN}(X,n,P)$ and $R\Gamma_{\NN,c}(X,n,P)$ for any polynomial $P\in 1+T\cdot \Q_p[T]$, which reduce to the theory of \cite{NN2016} when $P(T)=1-T$.

\begin{rmk}
\label{NNsyntomicfeaturesrmk}
In the following list we recall some features of Nekov\'{a}\v{r}--Nizio\l{} syntomic $P$-cohomology that we will need in the sequel.
\begin{itemize}
    \item [(i)] For $P\in 1+T\cdot\Q_p[T]$, $n\in\Z$ and $?\in\{c,\emptyset\}$, there is a spectral sequence
    \begin{equation}
    \label{etaletosyntomicspectralsequencetrivial}  E_2^{i,j}:=H^i_{\st,P}\Big(F,\D_\pst\big(H^j_{\et,?}(X_{\bar{\Q}_p},\Q_p(n))\big)\Big)\Rightarrow H^{i+j}_{\NN,?}(X,n,P).
    \end{equation}
    \item [(ii)] Given $P,Q\in 1+T\cdot\Q_p[T]$, one has for every $n,m\in\Z$ a natural map
    \begin{equation}
    \label{PQsyntomicpairingtrivial}
        R\Gamma_{\NN}(X,n,P)\otimes_{K_0}R\Gamma_{\NNc}(X,m,Q)\to R\Gamma_{\NNc}(X,n+m,P*Q)\,,
    \end{equation}
    inducing cup products
    \begin{equation}
    \label{syntomiccupprodtrivial}
    H^i_{\NN}(X,n,P)\times H^j_{\NNc}(X,m,Q)\xrightarrow{\cup}H^{i+j}_{\NNc}(X,n+m,P*Q)\,.
    \end{equation}
    The spectral sequence \eqref{etaletosyntomicspectralsequencetrivial} is compatible with this cup product and the cup product of Proposition \ref{cupprodst}. This compatibility is made completely explicit in \cite[Section 2.4]{BLZ2016} and includes sign changes in the formulas of Table \eqref{tableremark} according to the degrees of the cochains involved.
    \item [(iii)] If $P\in 1+T\cdot\Q_p[T]$ is such that $P(\zeta)\neq 0\neq P(\zeta/p)$ for all $\zeta\in\mu_d$ and $X$ has dimension $d_X$, then there is a trace map
    \begin{equation}
    \label{syntomictracemaptrivial}
        \Tr_{X,\NN, P}: H^{2d_X+1}_{\NNc}(X,d_X+1,P)\twoheadrightarrow H^1_{\st,P}(F,\Q_p^{nr}(1))\cong F
    \end{equation}
    whose kernel is the second step $\Fil^2$ of the $3$-step filtration induced on the cohomology group $H^{2d_X+1}_{\NNc}(X,d_X+1,P)$ by the spectral sequence \eqref{etaletosyntomicspectralsequencetrivial}.

    The isomorphism $H^1_{\st,P}(F,\Q_p^{nr}(1))\cong F$ is explicitly given by sending a class $[(w,z,y)]\in H^1_{\st,P}(F,\Q_p^{nr}(1))$ to $y-P(\sigma/p)^{-1}w$.
    \item [(iv)] Assume that $\iota:Z\to X$ is a smooth proper morphism of smooth $F$-varieties of relative dimension $j$. Then, for every $i\in\Z_{\geq 0}$, $P,Q\in 1+T\cdot \Q_p[T]$, $n\in\Z$, there are pushforward maps
    \begin{equation}
    \label{syntomicGysin}
        \iota_*: H^i_{\NN}(Z,n,P)\to H^{i+2j}_{\NN}(X,n+j,P)\,.
    \end{equation}
    satisfying the following projection formula:
    \begin{equation}
    \label{syntomicprojformulatrivial}
        \iota_*\alpha\cup\beta = \iota_*(\alpha\cup\iota^*\beta)
    \end{equation}
    for every $\alpha\in H^i_{\NN}(Z,n,P)$ and every $\beta\in H^j_{\NNc}(X,m,Q)$, where the cup product is the syntomic cup product defined in \eqref{syntomiccupprodtrivial} (on $Z$ and $X$ respectively). Moreover, the pushforward maps are compatible with the pushforward in \'{e}tale cohomology (cf. the discussion at the beginning of \cite[Section 6.4]{LZ2024} for this last fact, which is not directly stated in \cite{NN2016}).
\end{itemize}
\end{rmk}

\begin{rmk}
One can also define an $F$-linearized version of Nekov\'{a}\v{r}--Nizio\l{} syntomic and finite polynomial cohomology, replacing the semilinear Frobenius $\varphi$ with its $d$-th power (recall $d=[F_0:\Q_p]$) and extending scalars to $F$ everywhere in the construction. This is what actually happens in \cite{BLZ2016}, so we refer to this work for more details. We will denote this linearized theory with $\widetilde{R\Gamma}_{\NN}(X,n,P)$ (resp. $\widetilde{R\Gamma}_{\NNc}(X,n,P)$ for the compactly supported version). Now one can consider polynomials $P\in 1+T\cdot F[T]$. This version of syntomic and finite polynomial cohomology is less refined that the semilinear one, but more treatable when it comes to computations. Note that the two versions agree when $F=\Q_p$.

\medskip
The discussion of Remark \ref{NNsyntomicfeaturesrmk} has an analog for the $F$-linearized syntomic/finite polynomial cohomology, which is essentially obtained replacing $R\Gamma_{\NN,?}(-)$ with $\widetilde{R\Gamma}_{\NN,?}(-)$ and $H^i_{st,P}$ with $\Tilde{H}^i_{\st,P}$. Let us just remark that in this case the trace map of (iii) above is defined for polynomials such that $P(1)\neq 0\neq P(1/q)$ (for $q=p^d$) and that the isomorphism $\Tilde{H}^1_{\st,P}(F,\Q_p^{nr}(1))\cong F$ is now explicitly given by sending a class $[(w,z,y)]\in \Tilde{H}^1_{\st,P}(F,\Q_p^{nr}(1))$ to $y-P(1/q)^{-1}w$.
\end{rmk}

\subsection{Liebermann's trick for Nekov\'{a}\v{r}--Nizio\l{} syntomic cohomology}
\label{NNsyntomicwithcoeff}
Thanks to the so-called \emph{Liebermann's trick} (cf. for instance \cite[Paragraph 4.1.3]{BSV2020a} for the case of modular curves with smooth $\Z_p$-model and \cite[Sections 5.3 and 6.5]{LZ2024} for the general case of PEL type Shimura varieties in which no smooth $\Z_p$-model is required), we will always be able to assume that the relevant cohomology groups for (products of) modular curves appearing in the sequel are direct summands of cohomology groups of (products of) suitable Kuga-Sato varieties with trivial coefficients. Hence it is possible to define Nekov\'{a}\v{r}--Nizio\l{} syntomic $P$-cohomology over modular curves (resp. triple product of modular curves) with coefficients in an algebraic representation $V$ of $\GL_2$ (resp. $\GL_2\times\GL_2\times\GL_2$). 

\medskip
We denote these cohomology groups by $H^i_{\NN}(W,\mathcal{V},n,P)$ (resp. $H^i_{\NNc}(W,\VV,n,P)$ for the compactly supported version), where $W$ is a modular curve or a triple product of modular curves defined over $\Q_p$ (or a finite extension $F$ of $\Q_p$). 

\begin{rmk}
\label{NNsynwithcoeffrmk}
Let $Y:=Y_1(Mp^t)_F$ and denote $d:Y\hookrightarrow Y^3$ the diagonal embedding. 
As observed in \cite[Section 6.5]{LZ2024}, the properties (i)-(iv) described in Remark \ref{NNsyntomicfeaturesrmk} have a counterpart in this setting, given as follows. For (i)' and (ii)' we write $S=Y$ or $S=Y^3$ and $V$ is an algebraic representation of $\GL_2$ or $\GL_2\times\GL_2\times\GL_2$ accordingly.
\begin{itemize}
    \item [(i)'] For $P\in 1+T\cdot\Q_p[T]$, $n\in\Z$ and $?\in\{c,\emptyset\}$, there is a spectral sequence:
    \begin{equation}
    \label{etaletosyntomicspectralsequencecoeff}  E_2^{i,j}:=H^i_{\st,P}\Big(F,\D_\pst\big(H^j_{\et,?}(S_{\bar{\Q}_p},\VV(n))\big)\Big)\Rightarrow H^{i+j}_{\NN,?}(S,\VV, n,P).
    \end{equation}
    \item [(ii)'] Given $P,Q\in 1+T\cdot\Q_p[T]$, one has for every $n,m\in\Z$
    \begin{equation}
    \label{syntomiccupprodcoeff}
    H^i_{\NN}(S,\VV,n,P)\times H^j_{\NNc}(S,\WW,m,Q)\xrightarrow{\cup}H^{i+j}_{\NNc}(S,\VV\otimes\WW, n+m,P*Q)\,.
    \end{equation}
    The spectral sequences \eqref{etaletosyntomicspectralsequencecoeff} are compatible with this cup product and the cup product of Proposition \ref{cupprodst} (with the same \emph{caveat} concerning sign changes as in \eqref{syntomiccupprodtrivial}).
    \item [(iv)'] Fix an algebraic representation $V$ of $\GL_2\times\GL_2\times\GL_2$ and let $W=V|_{\GL_2}$. Then for every $P\in 1+ T\cdot \Q_p[T]$, $n\in\Z$, there are pushforward maps (again compatible with pushforward in \'{e}tale cohomology)
    \begin{equation}
    \label{syntomicGysincoeff}
        d_*=(d^{\WW})_*: H^i_{\NN}(Y,\WW,n,P)\to H^{i+4}_{\NN}(Y^3,\VV,n+2,P)\,.
    \end{equation}
    and pullback maps
    \begin{equation}
    \label{syntomicpullbackcoeff}
        d^*=(d^{\WW})^*: H^i_{\NNc}(Y^3,\VV,n,P)\to H^i_{\NNc}(Y,\WW,n,P)\,.
    \end{equation}
    
    satisfying the following projection formula:
    \begin{equation}
    \label{syntomicprojformulacoeff}
        d_*\alpha\cup\beta = d_*(\alpha\cup d^*\beta)
    \end{equation}
    for every $\alpha\in H^i_{\NN}(Y,\WW,n,P)$ and every $\beta\in H^j_{\NNc}(Y^3,\VV,m,Q)$, where the cup product is the syntomic cup product defined in \eqref{syntomiccupprodcoeff} (on $Y$ and $Y^3$ respectively).
\end{itemize}
\end{rmk}

\begin{rmk}
The discussion of this section applies \emph{mutatis mutandis} to the $F$-linearized version of Nekov\'{a}\v{r}--Nizio\l{} syntomic/finite polynomial cohomology.
\end{rmk}

\subsection{The syntomic Abel--Jacobi map}
\label{abeljacobisyn}
The following diagram summarizes the étale and the syntomic versions of the Abel--Jacobi maps over $F$. Note again that we use the same symbol for the \emph{syntomic coefficients} and their étale counterparts. Here $(f,g,h)$ is again a triple of modular forms satisfying assumption \ref{balselfdual}.

\begin{center}
    \begin{tikzcd}
    & H^0_{\NN}(Y_{t,F},\Hs_\mathbf{r},r)\cong H^0_\et(Y_{t,F},\Hs_\mathbf{r}(r))\arrow[d, "AJ_{\syn,F}" ']\arrow[dr, bend left= 10, "AJ_{\et,F}"]\\
    & H^1_\st\big(F,\D_\pst(H^3_\et(Y^3_{t,\bar{\Q}_p},\Hs_{[\mathbf{r}]}(r+2))_L)\big)\arrow[d,"s_\mathbf{r}"', "\cong"]\arrow[r, "\exp_\st", "\cong"'] & H^1_g\big(F, H^3_\et(Y^3_{t,\bar{\Q}_p},\Hs_{[\mathbf{r}]}(r+2))_L\big)\arrow[d, "s_\mathbf{r}", "\cong"'] \\
    & H^1_\st\big(F,\D_\pst(H^3_\et(Y^3_{t,\bar{\Q}_p},\Ls_{[\mathbf{r}]}(2-r))_L)\big)\arrow[d, twoheadrightarrow, "pr^{fgh}_{\st}"']\arrow[r, "\exp_\st", "\cong"'] & H^1_g\big(F, H^3_\et(Y^3_{t,\bar{\Q}_p},\Ls_{[\mathbf{r}]}(2-r))_L\big)\arrow[d, twoheadrightarrow, "pr^{fgh}_{\et}"] \\
    & H^1_\st(F,\D_\pst(V(f,g,h)))\arrow[r, "\exp_\st", "\cong"'] & H^1_g(F, V(f,g,h))
    \end{tikzcd}
\end{center}

\begin{rmk}
\label{AJsyntomicrmk}
Some remarks are in order.
\begin{itemize}
    \item [(i)] The isomorphism $H^0_{\NN}(Y_{t,F},\Hs_\mathbf{r},r)\cong H^0_\et(Y_{t,F},\Hs_\mathbf{r}(r))$ can be realized as the composition of the following canonical isomorphisms:
    \begin{align*}
        H^0_{\NN}(Y_{t,F},\Hs_\mathbf{r},r) & \cong H^0_\st\big(F,\D_\pst(H^0_\et(Y_{t,\bar{\Q}_p},\Hs_\mathbf{r}(r)))\big)\\
        &\cong  H^0\big(F,H^0_\et(Y_{t,\bar{\Q}_p}, \Hs_\mathbf{r}(r))\big)\\
        &\cong H^0_\et(Y_{t,F}, \Hs_\mathbf{r}(r))\,.
    \end{align*}
    The first isomorphism follows from the spectral sequence \eqref{etaletosyntomicspectralsequencecoeff}, the second isomorphism follows from Theorem \ref{phiNGaloistheorem} and the third isomorphism is afforded by the Hochschild--Serre spectral sequence.
    \item [(ii)] The construction of the syntomic Abel--Jacobi map $AJ_{\syn,F}$ follows the same ideas as for the étale counterpart (cf. the diagram above Remark \ref{AJetrmk} and the remark itself). As described in Remark \ref{NNsynwithcoeffrmk}, the diagonal embedding $d_t: Y_{t,F}\hookrightarrow Y_{t,F}^3$ gives rise to a pushforward map
    \[
    d_{t,*}: H^0_{\NN}(Y_{t,F},\Hs_\mathbf{r},r)\to H^4_{\NN}(Y_{t,F}^3,\Hs_{[\mathbf{r}]},r+2)\,.
    \]
    Set $D^j:=\D_\pst\big(H^j_\et(Y_{t,\bar{\Q}_p}^3,\Hs_{[\mathbf{r}]}(r+2)_L)\big)$. Then the spectral sequence
    \begin{equation}
    \label{spectralWt}
        E_2^{i,j}:=H^i_\st(F,D^j)\Rightarrow H^{i+j}_{\NN}(Y_{t,F}^3,\Hs_{[\mathbf{r}]},r+2)_L
    \end{equation}
    induces a surjective morphism $H^4_{\NN}(Y_{t,F}^3,\Hs_{[\mathbf{r}]},r+2)_L\twoheadrightarrow H^1_\st(F,D^3)$, whose kernel is the second step filtration $\Fil^2\big(H^4_{\NN}(Y_{t,F}^3,\Hs_{[\mathbf{r}]},r+2)_L\big)$. The syntomic Abel--Jacobi maps is the composition
    \[
    AJ_{\syn,F}: H^0_{\NN}(Y_{t,F},\Hs_\mathbf{r},r)\xrightarrow{d_{t,*}} H^4_{\NN}(Y_{t,F}^3,\Hs_{[\mathbf{r}]},r+2)\twoheadrightarrow H^1_\st(F,D^3)\,.
    \]
    \item [(iii)] The commutativity $AJ_{\syn,F}=\exp_{\st}\circ AJ_{\et,F}$ uses crucially the compatibility between pushforwards in syntomic and \'{e}tale cohomology (cf. Remark \ref{NNsyntomicfeaturesrmk}, (iv)).
\end{itemize}
\end{rmk}

\begin{cor}
The class $\kappa(f,g,h)\in H^1(\Q_p,V(f,g,h))$ defined in Section \ref{abeljacobi} always belongs to the subspace $H^1_g(\Q_p,V(f,g,h))$.
\begin{proof}
The above discussion and Theorem \ref{phiNGaloistheorem} show that the image of $\kappa(f,g,h)$ inside $H^1(F,V(f,g,h))$ belongs to the subspace $H^1_g(F,V(f,g,h))$ and it is $\Gal(F/\Q_p)$-invariant. We conclude recalling that restriction induces an isomorphism (cf. Remark \ref{BKremark})
\[
H^1_g(\Q_p,V(f,g,h))\xrightarrow[\cong]{\res_{F/\Q_p}}H^1_g(F,V(f,g,h))^{\Gal(F/\Q_p)}\,.
\]
\end{proof}
\end{cor}

\begin{rmk}
\label{usetildermk}
If we assume that the triple $(f,g,h)$ is $(F,1-T)$-convenient (cf. Definition \ref{expconvtriple}), then as observed in Lemma \ref{lemmast1} we have an identification $\Tilde{H}^1_{\st,1-T}(F,V(f,g,h))\cong H^1_{\st,1-T}(F,V(f,g,h))$. Moreover, it follows from Remark \ref{inclH0} that we have an inclusion $H^0_{\NN}(Y_{t,F},\Hs_{\mathbf{r}},r)\subseteq \Tilde{H}^0_{\NN}(Y_{t,F},\Hs_{\mathbf{r}},r)$. Hence, if we assume that $(f,g,h)$ is $(F,1-T)$-convenient, we can proceed with the $F$-linearized version of syntomic/finite polynomial cohomology (and a corresponding syntomic Abel--Jacobi map) for what concerns the proof of the explicit reciprocity law that we will give in Section \ref{explicitreciprocitychapt}. 
\end{rmk}

\subsection{Log-rigid syntomic and finite-polynomial cohomology}
\label{materialonsynfp}

We will refer to the work of Ertl and Yamada (cf. \cite{EY2021} and the preprints \cite{EY2024}, \cite{EY2025}) for a version of Hyodo--Kato cohomology for strictly semistable log-schemes over $k_F^0$ (where $k_F^0$ denotes the log-scheme $\Spec(k_F)$ with the log structure associated with the monoid homomorphism $\N\to k_F$, $1\mapsto 0$).
Building up on the aforementioned works, \cite{Yam2025} and \cite{HW2022}) develop versions of syntomic (and finite polynomial) cohomology with coefficients for strictly semistable log schemes over $\OO_F^\varpi$ (where $\OO_F^{\varpi}$ denotes the scheme $\Spec(\OO_F)$ with canonical log structure associated with the monoid homomorphismm $\N\to\OO_F$, $1\mapsto\varpi$). 

\medskip
We fix once and for all the choice of the uniformizer $\varpi\in\OO_F$ and we follow the discussion in \cite[Section 7]{LZ2024} to summarize the main properties of such cohomology theories.

\begin{defi}
A strictly semistable $\OO_F$-scheme with boundary is a pair $(X,D)$ where is an $\OO_F$-scheme flat of finite type and $D$ is a closed subscheme (also flat over $\OO_F$) such that:
\begin{enumerate}[(i)]
    \item the union of $D$ and of the special fiber $X_0$ of $X$ is a strict normal crossing divisor (see \cite[Tag 0CBN]{stacks-project} for the definition);
    \item each point of $X$ has a Zariski-open neighbourhood which is smooth over
    \[
    \Spec\big(\OO_F[t_1,\dots,t_m,s_1,\dots,s_n]/(t_1\cdots t_m-\varpi)\big)
    \]
    for some $m,n\in\Z_{\geq 0}$ and with $D$ corresponding to the locus $s_1\cdots s_n=0$.
\end{enumerate}
\end{defi}

We will write $U:X-D$ in what follows.

\medskip
By endowing a strictly semistable $\OO_F$-scheme with boundary $(X,D)$ with the log structure associated with the divisor $D\cup X_0$, one obtains a strictly semistable log-scheme with boundary over $\OO_F^\varpi$ (cf. \cite[Definition 4.4]{EY2021}).

In \cite[Section 2]{EY2021} the authors construct a Hyodo--Kato map in the derived category of $F_0$-vector spaces (depending on the choice of the uniformizer $\varpi\in\OO_F$ and of a branch of the $p$-adic logarithm $\log:F^\times\to F$)
\[
\Psi_\varpi:=\Psi_{\varpi,\log,F_0}: R\Gamma^{\HK}_{\rig}(X_0\langle D_0\rangle)\to R\Gamma_{\lrig}(X_0\langle D_0\rangle/\OO_F^\varpi)\,.
\]
The notations is as follows:
\begin{enumerate}[(i)]
    \item $R\Gamma^{\HK}_{\rig}(X_0\langle D_0\rangle)$ is a complex of $F_0$-vector spaces with $F_0$-semilinear Frobenius $\varphi$ and $F_0$-linear monodromy $N$ such that $N\varphi=p\varphi N$, called rigid Hyodo--Kato cohomology;
    \item $R\Gamma_{\lrig}(X_0\langle D_0\rangle/\OO_F^\varpi)$ is a complex of $F$-vector spaces, quasi-isomorphic to the de Rham cohomology of the dagger space $\XX=]X_0[_{X_F^{an}}^\dag$ (the tube of $X_0$ inside the analytification $X_F^{an}$ of $X_F$) with log-poles along $D_F$, called log-rigid cohomology.
\end{enumerate}

\begin{rmk}
\label{HKisoafterbasechange}
The Hyodo--Kato map $\Psi_\varpi$ becomes a quasi-isomorphism after extending scalars to $F$ on the source (cf. the discussion in \cite[p. 278]{EY2021}).
\end{rmk}

\begin{rmk}
\label{logrigpro}
Note that, if $X_F$ is proper over $F$, then $R\Gamma_{\lrig}(X_0\langle D_0\rangle/\OO_F^\varpi)$ is quasi-isomorphic to $R\Gamma_{\dR}(U_F)=R\Gamma(X_F,\Omega^\bullet_{X_F}\langle D_F\rangle)$, i.e., the algebraic de Rham cohomology of $X_F$ with log-poles along $D_F$.
\end{rmk}

Analogously, \cite{EY2024} in the compactly supported case constructs a map
\[
\Psi_\varpi:=\Psi_{\varpi,\log,F_0,c}: R\Gamma^{\HK}_{\rig}(X_0\langle -D_0\rangle)\to R\Gamma_{\lrig}(X_0\langle -D_0\rangle/\OO_F^\varpi)\,.
\]
such that $\Psi_{\varpi,\log,F_0,c}\otimes_{F_0}F$ is a quasi-isomorphism.

\medskip
In \cite[Section 4.2]{EY2021}, the authors attach to a pair $(X,D)$ and to an integer $r\geq 0$ a complex $R\Gamma_{\lrigs}(X\langle D\rangle,r)$ computing log-rigid syntomic cohomology, defined as the homotopy limit over a suitable diagram. If one assumes that $X$ is proper over $\OO_F$, then such complex is quasi isomorphic to the homotopy limit (cf. \cite[Section 1.2]{EY2021} for the conventions on the notation):
\begin{equation}
\label{rigidsyntomicdefinition}
\begin{aligned}
    R\Gamma_{\lrigs}&(X\langle D\rangle,r)\\
    & :=\left[ \hspace{-1.8 cm}
\begin{tikzcd}[column sep =  huge]
    & R\Gamma_{\rig}^{\HK}(X_0\langle D_0\rangle)\arrow[r, "\text{$(1-p^{-r}\varphi,\Psi_\varpi)$}"]\arrow[d, "N"] & R\Gamma^{\HK}_{\rig}(X_0\langle D_0\rangle)\oplus R\Gamma_{\dR}(U_F)/\Fil^r\arrow[d, "\text{$(N,0)$}"]\\
    & R\Gamma^{\HK}_{\rig}(X_0\langle D_0\rangle)\arrow[r, "1-p^{1-r}\varphi"] & R\Gamma^{\HK}_{\rig}(X_0\langle D_0\rangle)
\end{tikzcd}\right]\,.
\end{aligned}
\end{equation}

Here we are still writing $\Psi_\varpi$ to denote the composition of the Hyodo--Kato map $\Psi_\varpi$ with the quasi-isomorphism between log-rigid cohomology and de Rham cohomology afforded by Remark \ref{logrigpro} when $X$ is proper.

\begin{teo}
\label{logrigidvscristalline}
Let $(X,D)$ be a strictly semistable $\OO_F$-scheme with boundary such that $X$ is proper and let $U=X-D$. Then for all $r\in\Z_{\geq 0}$ there are canonical quasi-isomorphisms
\[
R\Gamma_{\lrigs}(X\langle D\rangle,r)\cong R\Gamma_{\NN}(U_F,r).
\]
\begin{proof}
This is \cite[Theorem 5.1]{EY2021}. 
\end{proof}
\end{teo}

Assume from now on in this section that $X$ is proper over $\OO_F$. 

\medskip
More generally, one can define log-rigid finite polynomial cohomology, replacing $1-p^{-r}\varphi$ with $P(p^{-r}\varphi)$ for $P\in 1+T\cdot\Q_p(T)$ in the diagram defining log-rigid syntomic cohomology and obtain isomorphisms
\begin{equation}
\label{rigidtocrissyn}
R\Gamma_{\lrigs}(X\langle D\rangle,r,P)\cong R\Gamma_{\NN}(U_F,r,P).   
\end{equation}

\begin{rmk}
\label{logrigidcristallinecompatibility}
Following \cite{EY2024} (as summarized in \cite[Section 7.3]{LZ2024}), one can also define a compactly supported version of log-rigid syntomic (and finite polynomial) cohomology and prove the existence of canonical quasi-isomorphisms for $r\in\Z_{\geq 0}$
\begin{equation}
\label{rigidtocrissync}
R\Gamma_{\lrigs}(X\langle -D\rangle,r,P)\cong R\Gamma_{\NNc}(U_F,r,P).    
\end{equation}
One can then obtain cup products
\[
R\Gamma_{\lrigs}(X\langle D\rangle,r,P)\times R\Gamma_{\lrigs}(X\langle -D\rangle,s,Q)\to R\Gamma_{\lrigs}(X\langle -D\rangle,r+s,P*Q)
\]
for $r,s\in\Z_{\geq 0}$ which are compatible with the cup product \eqref{syntomiccupprodtrivial} under the isomorphisms \eqref{rigidtocrissyn} and \eqref{rigidtocrissync}.
\end{rmk}

\begin{rmk}
\label{logrigidtilded}
Even for log-rigid syntomic (or finite polynomial) cohomology, one can obtain an $F$-linearized theory. When $X$ is proper over $\OO_F$, one just has to replace $\varphi$ with $\Phi=\varphi^d$, $p$ with $q=p^d$ and $R\Gamma_{\rig}^{\HK}(X_0\langle D_0\rangle)$ with $R\Gamma_{\rig}^{\HK}(X_0\langle D_0\rangle)\otimes_{F_0}F$ in the above diagram \eqref{rigidsyntomicdefinition}. We will denote it by $\widetilde{R\Gamma}_{\lrigs}(X\langle \pm D\rangle,r,P)$. Since Theorem \ref{logrigidvscristalline} (and more generally the quasi-isomorphisms \eqref{rigidtocrissyn} and \eqref{rigidtocrissync}) are a formal consequence of quasi-isomorphisms between log-rigid Hyodo--Kato cohomology and log-crystalline Hyodo--Kato cohomology, analog quasi-isomorphisms are valid for the $F$-linearized variants.
\end{rmk}

\subsection{Syntomic coefficients}
\label{syntomiccoefficientssection}
Log-rigid Hyodo--Kato and syntomic (or finite polynomial) cohomologies for strictly semistable schemes seem to lack a satisfying theory of coefficients, as far as we can deduce from the existing literature. In this section we try to underline some of the issues related with this problem and we outline some results that a reasonable theory of coefficients should afford.

\medskip
In \cite[Section 9]{Yam2025}, the author defines a category of syntomic coefficients as triples $(\mathscr{E},\Phi,\Fil)$, where $(\mathscr{E},\Phi)$ is a log overconvergent $F$-isocrystal (cf. \cite[Section 2.4]{Yam2025}) and $\Fil$ is a filtration on the de Rham realization $\mathscr{E}_{\dR}$ that satisfies Griffiths' transversality.

\medskip
Inside such category of coefficients, Yamada singles out a subcategory given by the triples $(\mathscr{E},\Phi,\Fil)$ in which $\mathscr{E}$ is \emph{unipotent} (i.e., an iterated extension of the trivial isocrystal). This condition is necessary to obtain a proof of the fact that the version of Hyodo--Kato map which allows non-trivial coefficients $(\mathscr{E},\Phi)$ is still a quasi-isomorphism (cf. \cite[Proposition 8.8]{Yam2025} and \cite[Corollary 3.8]{EY2024}).

\medskip
In \cite{HW2022}, the authors develop the formalism of finite polynomial cohomology with coefficients, employing the same category of coefficients as Yamada's (i.e., restricting to unipotent ones). On the other hand (cf. \cite[Section 6]{HW2022}), if one was able to prove that the Hyodo--Kato map is an isomorphism also for an isocrystal $(\mathscr{E},\Phi,\Fil)$ which is not necessarily unipotent, then the definition of syntomic (and finite polynomial) cohomology via a suitable mapping fiber and the subsequent formalism would work in the same way as in the unipotent case.

\begin{ass}
\label{conditional}
Till the end of Section \ref{endofproof}, we will \textbf{assume} that the results recalled in Theorem \ref{logrigidvscristalline}, Remark \ref{logrigidcristallinecompatibility} (above) and Proposition \ref{extrestrprop} (below) are still valid when $X$ is a suitable strictly semistable model of a modular curve and when the syntomic coefficients come from automorphic coefficients $\Hs_r$ for $r\geq 0$.
\end{ass}

As already mentioned in the introduction and explained in more detail in Section \ref{thegeneralcasesubsection} below, the recent preprint \cite{ABSV2025} develops a formalism that allows to obtain an unconditional proof of Theorem \ref{erlintroteo} (cf. also the more general statement of Theorem \ref{explicirreclawconj} below) for all balanced triples of weights $(k,l,m)$. We believe that it is still interesting to ask whether (and how) the results postulated in Assumption \ref{conditional} can be proven, possibly without having to recur to a Liebermann trick for Ertl--Yamada cohomology, which would require the existence of semistable integral models for the relevant compactified Kuga--Sato varieties (cf. the discussion in Section \ref{introcomparison} above).

\subsection{Rigid syntomic and finite polynomial cohomology for smooth schemes}
\label{rigidsyntomicsmoothsection}
For smooth $\OO_F$-schemes, one can rely on Besser's simpler version of rigid syntomic and finite polynomial cohomology (\cite{Bes2000}), which also admits a nice theory of syntomic coefficients (overconvergent filtered $F$-isocrystals, cf. \cite[Definition 8.1.1]{LZ2024}). In the rest of the section, we omit the coefficients to make the notation lighter.

\medskip
In this section we let $Z$ denote a smooth $\OO_F$-scheme, whose generic fiber $Z_F$ is proper over $F$. As above, we denote by $\ZZ$ the dagger space tube of $Z_0$ in the rigid analytification $Z_F^{an}$ of $Z_F$. We also let $D$ be a divisor on $Z$ intersecting transversely with the special fibre, so that $(Z,D)$ fits again in the definition of strictly semistable $\OO_F$-scheme with boundary.

Following \cite[Definition 8.1.4]{LZ2024} (resp. \cite[Definition 8.1.9]{LZ2024}), one can define for $r\in\Z$ and $P\in 1+T\cdot\Q_p(T)$ the rigid finite polynomial cohomology complex $R\Gamma_{\rigs}(Z\langle D\rangle,r,P)$ (resp. its compactly supported version $R\Gamma_{\rigs,c}(Z\langle-D\rangle,r,P)$).

Rigid finite polynomial cohomology admits a cup product (cf. \cite[p. 37]{Bes2012})
\[
R\Gamma_{\rigs}(Z\langle D\rangle,r,P)\times R\Gamma_{\rigs,c}(Z\langle -D\rangle,s,Q)\to R\Gamma_{\rigs,c}(Z\langle -D\rangle,r+s,P*Q).
\]

In what follows, we will make use of the following crucial result.

\begin{prop}
\label{extrestrprop}
Let $(X,D)$ be a strictly semistable $\OO_F$-scheme with boundary with $X$ proper and let $Z\subset X$ be a smooth open subscheme such that $Z_F$ is proper over $F$. Then for $r,s\in\Z_{\geq 0}$ and $P,Q\in 1+T\cdot\Q_p(T)$ there exist:
\begin{enumerate}[(i)]
    \item a restriction map $\res_Z\colon R\Gamma_{\lrigs}(X\langle D\rangle,r,P)\to R\Gamma_{\rigs}(Z\langle D\cap Z\rangle,r,P);$
    \item an extension-by-0 map $\Ext_0\colon R\Gamma_{\rigs}(Z\langle-D\cap Z\rangle,s,Q)\to R\Gamma_{\lrigs}(X\langle-D\rangle,s,Q)$.
\end{enumerate}
which are the transposes of each other under the pairings in log-rigid finite polynomial cohomology for $X$ and rigid finite polynomial cohomology for $Z$.
\begin{proof}
This follows combining Lemma 8.1.7, Proposition 8.1.11 and Corollary 8.1.15 of \cite{LZ2024}.
\end{proof}
\end{prop}

\begin{rmk}
\label{descriptionFlinearrigid}
Once more, one can define an $F$-linearized (i.e., tilded, according to our notation) version of the theory and of Proposition \ref{extrestrprop}. The definition simplifies to the mapping fibre:
\begin{equation}
\widetilde{R\Gamma}_{\rigs}(Z\langle D\rangle,r,P)=\big[\Fil^r R\Gamma_{\dR}(Z_F\langle D_F\rangle)\xrightarrow{P(q^{-r}\Phi)\circ\mathrm{sp}} R\Gamma_{\rig}(Z_0\langle D_0\rangle)_F\big].    
\end{equation}
Here $R\Gamma_{\rig}(Z_0\langle D_0\rangle)$ denotes the rigid cohomology of the special fibre, which can be obtained as
\[
R\Gamma_{\rig}(Z_0\langle D_0\rangle):=R\Gamma_{\dR}(\ZZ\langle\DD\rangle)
\]
and 
\[
\mathrm{sp}\colon R\Gamma_{\dR}(Z_F\langle D_F\rangle)\to R\Gamma_{\rig}(Z_0\langle D_0\rangle)
\]
is the so-called \emph{specialization map} (cf. \cite[p. 34]{Bes2012}).

One has short exact sequences
\begin{equation}
\label{sesrigidsyntomic}
\begin{tikzcd}[column sep=small, row sep=small]
&0\arrow[r]&\frac{H^{i-1}_{\rig}(Z_0\langle D_0\rangle)}{P(q^{-r}\Phi)\circ\mathrm{sp}(\Fil^r H^{i-1}_{\dR}(Z_F\langle D_F\rangle))}\arrow[r, "\mathbf{i}_P"] &\Tilde{H}^i_{\rigs}(Z,r,P)\arrow[dl, controls={+(-1,-1) and +(1,1)}]\\
&&\Fil^r H^{i}_{\dR}(Z_F\langle D_F\rangle)^{P(q^{-r}\Phi)\circ\mathrm{sp}=0}\arrow[r] &0.   
\end{tikzcd}
\end{equation}
For the compactly supported variant, $\widetilde{R\Gamma}_{\rigs,c}(Z\langle -D\rangle,r,P)$ one can again give a definition as mapping fibre which gives rise to analogous short exact sequences (cf. \cite[pp. 50-51]{Bes2012}). In particular, one has a surjection 
\begin{equation}
\label{ppsurjection}
\mathbf{p}_P\colon \Tilde{H}^i_{\rigs,c}(Z\langle -D\rangle,r,P)\twoheadrightarrow \Fil^r H^{i}_{\dR,c}(Z_F\langle-D_F\rangle)\cap H^{i}_{\rig,c}(Z_0\langle-D_0\rangle)^{P(\Phi)=0}.
\end{equation}
To be more precise, the above intersection is between $\Fil^r H^{i}_{\dR,c}(Z_F\langle-D_F\rangle)$ and the image of $H^{i}_{\rig,c}(Z_0\langle-D_0\rangle)^{P(\Phi)=0}$ under the so-called \emph{cospecialization map}.

\medskip
In the sequel we will use the following projection formula (cf. \cite[Proposition 4.6]{Bes2012}). For $x\in H^{i-1}_{\rig}(Z_0\langle D_0\rangle)$ and $y\in \Tilde{H}^i_{\rigs,c}(Z\langle -D\rangle,r,P)$ it holds
\begin{equation}
\label{crucialprojectionformula}
\mathbf{i}_Q(x)\cup_{\rigs} y=\mathbf{i}_Q(x\cup_{\rig}(\mathbf{p}_P(y)))    
\end{equation}
where the notation $\cup_{\rigs}$ (resp. $\cup_{\rig}$) refers to the cup product in rigid finite polynomial (resp. rigid) cohomology.
\end{rmk}

\subsection{De Rham cohomology and overconvergent modular forms}
\label{overconvergentsection}
In this section we recall some facts concerning the geometry of $X_1(Mp^t)$ (for $t\geq 1$ and $p\nmid M$) as a rigid analytic variety over $\Q_p$ (or a suitable extension of $\Q_p$). One can find a more detailed treatment in \cite{Col1997}, \cite[Section 4.4]{BE2010} and \cite[Section 4.1]{DR2017}. Recall that $L$ denotes a large enough finite extension of $\Q_p$ as usual.

\medskip
Set $K_t:=\Q_p(\zeta_{p^t})$ (where $\zeta_{p^t}$ denotes a primitive $p^t$-th root of unity). Thanks to \cite{KM1985}, we know that (if $M>4$) $X_1(Mp^t)_{K_t}$ admits a proper flat regular model $X_t$ over $\OO_{K_t}=\Z_p[\zeta_{p^t}]$, which represents the moduli problem $([\Gamma_1(M)],[\mathrm{bal}.\Gamma_1(p^t)^{\mathrm{can}}])$, whose special fiber is described in \cite[Theorem 13.11.4]{KM1985}. 
It is the union of $t+1$ irreducible components (which are smooth proper curves over $\F_p$), intersecting at supersingular points. In particular, the aforementioned theorem shows that there are exactly two irreducible components of $X_{t,0}$ (the special fiber of $X_t$) isomorphic to the Igusa curve usually denoted $\Ig(Mp^t)$ as curves over $\F_p$ (this Igusa curve represents over $\F_p$ the moduli problem denoted $([\Gamma_1(M)],[\Ig(p^r)])$ in \cite[Section 12.3]{KM1985}). Usually, these two irreducible components are denoted $\Ig_\infty(p^t)$ and $\Ig_0(p^t)$, to stress the fact that $\Ig_\infty(p^t)$ contains the cusp $\infty$.

\medskip
Up to extending scalars to a finite extension $F\subseteq L$ of $\Q_p(\zeta_{p^t})$, one can obtain a strictly semistable model $\Tilde{X}_t$ of $X_{t,F}$ over $\Spec(\OO_F)$, together with a birational map $\Tilde{X}_t\to X_t$. Since such a semistable model can be obtained by successive blowups at supersingular points in the special fiber, it follows that this map identifies two irreducible components of $\Tilde{X}_{t,0}$ (the special fiber of $\Tilde{X}_t$) with $\Ig_\infty(p^t)_{k_F}$ and $\Ig_0(p^t)_{k_F}$ (where $k_F$ is the residue field of $F$). Since the cusps are not involved in the successive blowups, we still have a cuspidal divisor $\Tilde{D}\subset \Tilde{X}_t$, so that $(\Tilde{X}_t,\Tilde{D})$ is a strictly semistable $\OO_F$-scheme with boundary.

\medskip
We can look at $X_{t,F}$ as a rigid analytic space $X_{t,F}^{\an}$ over $F$. We define $\W_\infty(p^t)$ and (resp.) $\W_0(p^t)$ of $X_{t,F}^{\an}$ as the preimages (i.e., the tubes) of $\Ig_\infty(p^t)_{k_F}$ and (resp.) $\Ig_0(p^t)_{k_F}$ under the reduction map. The spaces $\W_\infty(p^t)$ and $\W_0(p^t)$ are examples of what Coleman calls \emph{wide open subspaces}, i.e., they are isomorphic to the complement of a finite disjoint union of closed disks in a smooth curve. The preimage of the smooth locus of $\Ig_\infty(p^t)$ (resp. $\Ig_0(p^t)$) is an affinoid subset of $X_{t,F}^{\an}$ which will be denoted by $\A_\infty(p^t)$ (resp. $\A_0(p^t)$).

\medskip
Since $X_{t,F}$ is proper over $F$, rigid de Rham cohomology agrees with algebraic de Rham cohomology. In particular we can consider for every $\nu\geq 2$ and for $?\in\{\infty,0\}$ restriction maps
\[
\res_?: H^1_\dR(X_{t,F},\Hs_{\nu-2})_L\to H^1_\dR(\W_?(p^t),\Hs_{\nu-2})_L\, .
\]
According to \cite[Lemme 4.4.1]{BE2010}, one can equip the groups $H^1_\dR(\W_?(p^t),\Hs_{\nu-2})$ with an action of Hecke and diamond operators away from $Mp$ in such a way that the restriction maps are Hecke equivariant.

\medskip
Following the discussion in Section \ref{materialonsynfp}, we can consider (up to fixing once and for all a uniformizer $\varpi\in\OO_F$ and a branch of the $p$-adic logarithm) the rigid Hyodo--Kato cohomology $H^{1,\HK}_{\rig}(\Tilde{X}_{t,0}\langle \pm \Tilde{D}_0\rangle)$. 

\medskip
Recall that the $F_0$-vector space $H^{1,\HK}_{\rig}(\Tilde{X}_{t,0}\langle \pm \Tilde{D}_0\rangle)$ is endowed with the structure of $(\varphi,N)$-module. We let $\Phi:=\varphi^d$ (where $d=[F_0:\Q_p]$) so that, extending scalars from $F_0$ to $F$, we can endow (via the isomorphism induced by $\Psi_{\varpi}\otimes_{F_0}F$, cf. Remark \ref{HKisoafterbasechange}) $H^1_\dR(X_{t,F})$ with an $F$-linear Frobenius $\Phi$ and monodromy operator $N$ such that $N\circ\Phi=p^d\cdot \Phi\circ N$.

Using the crystalline version of Hyodo--Kato cohomology for semistable schemes over $\OO_F$ (namely log-crystalline cohomology), one can obtain a structure of $(\Phi,N)$-module on $H^1_\dR(X_{t,F},\Hs_{\nu-2})$ also when $\nu>2$. Thanks to the comparison isomorphisms between log-crystalline and rigid Hyodo--Kato cohomology (see \cite[Corollary 5.4 and Proposition 5.5]{EY2021}), these structures coincide when $\nu=2$ (and they should coincide more generally, cf. Assumption \ref{conditional}).

\medskip
One can endow the $F$-vector spaces $H^1_\dR(\W_?(p^t),\Hs_{\nu-2})$ with an $F$-linear Frobenius $\Phi_?$ such that the restriction map $\res_?$ is Frobenius equivariant for $?\in\{\infty,0\}$, as explained in \cite[pp. 647-648]{DR2017}. Since the operator $w_{p^t}$ interchanges $\W_\infty(p^t)$ and $\W_0(p^t)$ (cf. \cite[Lemma 4.4.3]{BE2010}), one has that
$\Phi_0=w_{p^t}^{-1}\circ\Phi_\infty\circ w_{p^t}$. 

\medskip
One of the main results of \cite{Col1997} (namely Theorem 2.1) can be stated as follows. The restriction maps induce an isomorphism
\begin{equation}
\label{decompinfzero}
    H^1_\dR(X_{t,F}, \Hs_{\nu-2})_L^{\mathrm{prim}}\xrightarrow{\cong}H^1_\dR(\W_\infty(p^t),\Hs_{\nu-2})_L^*\oplus H^1_\dR(\W_0(p^t),\Hs_{\nu-2})_L^*
\end{equation}
Here $H^1_\dR(X_{t,F}, \Hs_{\nu-2})_L^{\mathrm{prim}}$ denotes the $p$-primitive subspace, i.e., the subspace spanned by the common eigenspaces for the action of Hecke and diamond operators relative to systems of eigenvalues coming from $p$-primitive modular eigenforms of exact level dividing $Mp^t$ (recall that a modular newform is $p$-primitive if the $p$-part of the conductor of its nebentypus equals the $p$-part of its level).

The notation $H^1_\dR(\W_?(p^t),\Hs_{\nu-2})_L^*$ denotes the so-called \emph{pure classes}, i.e., those whose restriction to supersingular annuli is trivial (such annuli correspond to the complement $\W_?(p^t)-\A_?(p^t)$).

\medskip
According to \cite[p. 36]{BE2010}, the isomorphism \eqref{decompinfzero} is equivariant for the action of diamond and Hecke operators (away from $Mp$) and for the action of $\Phi$ on $H^1_\dR(X_{t,F}, \Hs_{\nu-2})_L^{\mathrm{prim}}$ and the action of $(\Phi_\infty,\Phi_0)$ on $H^1_\dR(\W_\infty(p^t),\Hs_{\nu-2})_L^*\oplus H^1_\dR(\W_0(p^t),\Hs_{\nu-2})_L^*$.

\medskip
In the same way as in \cite[Lemme 4.3.2]{BE2010}, one can show that the action of $\Phi_\infty$ on overconvergent modular forms of weight $\nu$, seen as sections of $\omega^{\nu}$ in a wide open neighbourhood of $\A_\infty(p^t)$ inside $\W_\infty(p^t)$, can be described as the $d$-th power of an operator $\varphi_\infty$ which on $q$-expansions act as $p^{(\nu-1)}\langle p\,;1\rangle V$. Here (cf. the end of Section \ref{derhametale}) $\langle a\,; b\rangle$ denotes the diamond operator corresponding to the unique class in $(\Z/Mp^t\Z)^\times$ which is congruent to $a$ modulo $M$ and to $b$ modulo $p^t$, while $V$ acts on $q$-expansions sending $q\mapsto q^p$ as usual.

\begin{rmk}
\label{phiinftyrmk}
According to \cite[p. 398]{Col1997} (see also \cite[Lemme 4.2.1]{BE2010}), for every quasi-Stein open  $W\subseteq W_\infty(p^t)$, the inclusion
\[
H^0(W,\omega^{\nu})\cong H^0(W,\omega^{\nu-2}\otimes\Omega^1\langle D_F\rangle)\hookrightarrow H^0(W,\Hs^{\nu-2}\otimes\Omega^1\rangle D_F\rangle)
\]
induces a Hecke-equivariant isomorphism
\begin{equation}
\label{Colemaniso}
\frac{H^0(W,\omega^{\nu})}{d^{k-1}\big(H^0(W,\omega^{2-k})\big)}\cong\frac{H^0(W,\Hs_{\nu-2}\otimes\Omega^1\langle D_F\rangle)}{\nabla\big(H^0(W,\Hs_{\nu-2})\big)}=H^1_{\dR}(W,\Hs_{\nu-2}),    
\end{equation}
where $d=q\cdot d/dq$ is Serre's derivative operator. The sections in $H^0(W,\omega^\nu)$ with $p$-depleted $q$-expansion admit a preimage under $d^{k-1}$ and, if $W$ contains the closure of $\A_\infty(p^t)$ (i.e., it is a strict neighbourhood of $\A_\infty(p^t)$), the restriction
\[
H^1_{\dR}(\W_\infty(p^t),\Hs_{\nu-2})\to H^1_{\dR}(W,\Hs_{\nu-2})
\]
is an isomorphism (cf. \cite[Remarque 4.4.5]{BE2010}). It follows that in this case the operator $U_p$ is an isomorphism on $H^1_\dR(\W_\infty(p^t),\Hs_{\nu-2})_L$  and it holds $U_p\circ \varphi_\infty=\varphi_\infty\circ U_p=\langle p\;;1\rangle p^{\nu-1}$.
\end{rmk}


\section{The explicit reciprocity law}
\label{explicitreciprocitychapt}
In this section we state and prove the main result of this work.
\subsection{Statement and first consequences}
We consider as before a triple $(f,g,h)$ of modular forms satisfying assumption \ref{balselfdual} and which is moreover $F$-exponential (cf. definition \ref{expconvtriple}). In what follows we keep the notation introduced in the previous sections.

\begin{defi}
\label{p-adicperiod}
We define the $p$-adic period
\[
\mathscr{I}_p(f,g,h):=\log_\BK^{fgh}(\kappa(f,g,h))(\eta_{f'}^{\varphi=a_p}\otimes\omega_g\otimes\omega_{h'}\otimes t_{r+2})\in L
\]
attached to the $F$-exponential triple $(f,g,h)$.
\end{defi}

Note again that the definition of $\mathscr{I}_p(f,g,h)$ does not depend on $F$, but only on the triple $(f,g,h)$ (and possibly on the common level $Mp^t$ chosen).

\begin{teo}
\label{explicirreclawconj}
Assume that the triple $(f,g,h)$ is $(F,1-T)$-convenient and that the forms $g$ and $h$ are $p$-depleted eigenforms. 
Then
\[
\mathscr{I}_p(f,g,h)=(-1)^{k-2}(r-k+2)!\cdot a_1(e_{\breve{f}}(\Tr_{Mp^t/M_1p^t}(d^{(k-l-m)/2}g\times h')))
\]
where $d=q\tfrac{d}{dq}$ denotes Serre's derivative operator and $a_1(-)$ denotes the first Fourier coefficient of the $q$-expansion at $\infty$.
\end{teo}

\begin{rmk}
Note that Theorem \ref{explicirreclawconj} applies to the situation described by Assumption \ref{cuspasslocp}, as long as the triple $(f,g,h)$ is $(F,1-T)$-convenient.
\end{rmk}

\medskip
The rest of the paper is devoted to the proof of Theorem \ref{explicirreclawconj}. We will initially obtain a proof that, in the case of balanced triples of weights $(k,l,m)\neq(2,2,2)$ is conditional on Assumption \ref{conditional}. In the final Section \ref{thegeneralcasesubsection}, we indicate how the machinery introduced in the recent preprint \cite{ABSV2025} can be applied to obtain an unconditional proof of Theorem \ref{explicirreclawconj} for general balanced triples of weights $(k,l,m)$. 

\subsection{Restriction to the ordinary locus}
For $\xi\in\{g,h'\}$ we have that
\[
\label{omegaxi1}
\omega_\xi\otimes t_{\nu-1}\in\Fil^{\nu-1} H^1_{\dR,c}(Y_{t,\Q_p},\Hs_{\nu-2})_L,
\]
where (as before) $t_n$ is a generator of $\D_{\dR}(\Q_p(n))$ (depending on a choice of a compatible system of $p$-th power roots of unity in $\C_p$).

Then there exists a polynomial $P_\xi(T)\in 1+T\cdot L[T]$ such that $P_\xi(\Phi)(\omega_\xi\otimes t_{\nu-1})=0$. Here, we are letting the Frobenius $\Phi$ act on de Rham cohomology via the comparison isomorphism (recall that $Y_t$ acquires semistable reduction over $F$)
\begin{align*}
F\otimes_{F_0}\D_{\st,F}\big(H^1_{\et,c}(Y_{t,\bar{\Q}_p},\Hs_{\nu-2}(\nu-1))_L\big)&=\D_{\dR,F}\big(H^1_{\et,c}(Y_{t,\bar{\Q}_p},\Hs_{\nu-2}(\nu-1))_L\big)\\
& \cong H^1_{\dR,c}(Y_{t,F},\Hs_{\nu-2}[\nu-1])_L.
\end{align*}

In the situation of Remark \ref{phigammagh}, one can choose $P_\xi(T):=1-\mu_\xi^d p^{d(\nu-1)}\cdot T\in 1+T\cdot L[T]$. 

\medskip
With the notation introduced in Section \ref{phiNsection}, we have that
\[
\omega_\xi\otimes t_{\nu-1}\in \Tilde{H}^0_{\st,P_\xi}\big(F,\D_{\pst}(H^1_{\et,c}(Y_{t,\bar{\Q}_p},\Hs_{\nu-2}(\nu-1))_L)\big).
\]

\begin{lemma}
\label{degenklarge}
For every $P(T)\in 1+T\cdot L[T]$, every $\nu\in\Z_{\geq 2}$ and every $n\in\Z$, there are natural isomorphisms
\begin{equation}
\label{degeniso}
H^1_{\NNc}(Y_{t,F},\Hs_{\nu-2},n,P)_L\cong H^0_{\st,P}(F,D^1(n)).
\end{equation}
If, moreover, $\nu>2$, there are natural isomorphisms
\begin{equation}
\label{degeniso2}
H^2_{\NNc}(Y_{t,F},\Hs_{\nu-2},n)_L\cong H^1_{\st,P}(F,D^1(n)),
\end{equation}
where
\[
D^j(n):=\D_\pst\big(H^j_{\et,c}(Y_{t,\bar{\Q}_p},\Hs_{\nu-2}(n))_L\big).
\]
Analogous isomorphisms hold for the $F$-linearized (or tilded) variants.
\begin{proof}
It holds $H^0_{\et,c}(Y_{t,\bar{\Q}_p},\Hs_{\nu-2})=0$ for every $\nu\geq 2$ and, if $\nu>2$, it also holds that $H^0_\et(Y_{t,\bar{\Q}_p},\Hs_{\nu-2})=0$ (cf. \cite[Lemma 2.1]{BDP2013}) and (by duality) $H^2_{\et,c}(Y_{t,\bar{\Q}_p},\Hs_{\nu-2})=0$.
The claimed isomorphisms follow from the study of the degeneration of the spectral sequence (cf. equation \eqref{etaletosyntomicspectralsequencecoeff})
\begin{equation}
\label{syntomictostspectral}
    E_2^{i,j}:=H^i_{\st,P}(F,D^j(n))\Rightarrow H^{i+j}_{\NNc}(Y_{t,F},\Hs_{\nu-2},n,P)_L.
\end{equation}
The same proof applies to the $F$-linearized variants.
\end{proof}
\end{lemma}

The next lemma shows that, under suitable hypothesis, one can recover the isomorphism \eqref{degeniso2} also for weight $\nu=2$.
\begin{lemma}
Let $P(T)\in 1+T\cdot L[T]$ and $n\in\Z$ such that $P(\zeta p^{n-1})\neq 0\neq P(\zeta p^n)$ for all $\zeta\in\mu_d$. Then there is a natural isomorphism
\begin{equation}
H^2_{\NNc}(Y_{t,F},n,P)\cong H^1_{\st,P}\big(F,\D_\pst\big(H^1_{\et,c}(Y_{t,\bar{\Q}_p},\Q_p(n))_L\big)\big)  
\end{equation}
The analogous isomorphism holds for the $F$-linearized variants if the chosen polynomial satisfies $P(q^{n-1})\neq 0\neq P(q^n)$ (where recall $q=p^d$).
\begin{proof}
Note that in this case $\D_\pst\big(H^2_{\et,c}(Y_{t,\bar{\Q}_p},\Q_p(n))_L\big)\cong\Q_p^{nr}(1-n)$. Under the given hypothesis, Lemma \ref{lemmast1} shows that $H^i_{\st,P}(F,\Q_p^{nr}(1-n))=0$ for $i\neq 1$. This implies that, as in the previous Lemma \ref{degenklarge}, we can exploit the degeneration of the spectral sequence \eqref{syntomictostspectral} to deduce the result. The same proof applies to the $F$-linearized variants.
\end{proof}
\end{lemma}

\medskip
\begin{cor}
\label{syntomicliftsgh'}
For $\xi\in\{g,h'\}$, the class $\omega_\xi\otimes t_{\nu-1}$ can be uniquely lifted to a class
\[
\omega^{\NN}_{\xi,\nu-1}\in \Tilde{H}^1_{\NNc}(Y_{t,F},\Hs_{\nu-2},\nu-1,P_\xi)_L.
\]
\end{cor}

\begin{rmk}
\label{etalift}
Similarly, one can lift $\eta_{f'}^{\varphi=a_p}\in V^*_{\dR,Mp^t}(f')\subseteq H^1_{\dR,c}(Y_{t,\Q_p},\Hs_{k-2})_L$ to a syntomic class. We have
\[
\eta_{f'}^{\varphi=a_p}\in \Fil^0 H^1_{\dR,c}(Y_{t,F},\Hs_{k-2})_L
\]
and, setting, $P_f(T):=1-a_p(f)^{-d}\cdot T\in 1+T\cdot L[T]$ we have (by construction, cf. Definition \ref{etafrobp}) that
\[
P_f(\Phi)(\eta_{f'}^{\varphi=a_p})=0.
\]
One can proceed as above to get a unique lift
\begin{equation}
\label{synteta}
    \eta^{\NN}_{f'}\in \Tilde{H}^1_{\NNc}(Y_{t,F},\Hs_{k-2},0,P_f)_L
\end{equation} 
of $\eta_{f'}^{\varphi=a_p}$.   
\end{rmk}

\begin{prop}
\label{syntomictrace}
Assume that $(f,g,h)$ is $(F,1-T)$-convenient. Then the $p$-adic period introduced in Definition \ref{p-adicperiod} satisfies (with the notation introduced above):
\begin{equation}
\label{syntomictraceformula}
    \mathscr{I}_p(f,g,h) = \Tr_{Y_{t,F},\NN,P_{fgh}}\Big(\eta^{\NN}_{f'}\bigcup \Upsilon(\DET_\mathbf{r}^{\NN}\cup\omega^{\NN}_{g,l-1}\cup\omega^{\NN}_{h',m-1})\Big)\,.
\end{equation}
The notation is as follows.
\begin{itemize}
    \item [(i)] $\DET_\mathbf{r}^{\NN}\in H^0_{\NN}(Y_{t,F},\Hs_\mathbf{r},k-2,1-p^{k-2-r}T)$ is the syntomic incarnation of $\DET_\mathbf{r}^\et\otimes t_{k-2-r}$ (via the isomorphism discussed in Remark \ref{AJsyntomicrmk} (i)), which we view as element of $\Tilde{H}^0_{\NN}(Y_{t,F},\Hs_\mathbf{r},k-2,1-q^{k-2-r}T)$ with $q=p^d$, cf. Remark \ref{usetildermk}.
    \item [(ii)] The cup products are taken in ($F$-linearized) finite polynomial cohomology and the big cup product $\bigcup$ arises from the pairing
    \[
    \Hs_{k-2}\otimes_{\OO_{Y_{t,F}}}\Hs_{k-2}[k]\to \OO_{Y_{t,F}}[2]\,,
    \]
    where recall that for all $i\geq 0$ one has pairings $\Hs_i\otimes_{\OO_{Y_{t,F}}}\Hs_i\to \OO_{Y_{t,F}}[-i]$ (cf. equation \ref{deRhamdualsheaf}), which we also use to produce the map
    \[
    \Upsilon:\Hs_\mathbf{r}[k-2]\otimes_{\OO_{Y_{t,F}}}\Hs_{l-2}[l-1]\otimes_{\OO_{Y_{t,F}}}\Hs_{m-2}[m-1]\to\Hs_{k-2}[k]\,.
    \]
    \item [(iii)] The syntomic trace map
    \begin{equation}
    \label{tracefgh}
    \Tr_{Y_{t,F},\syn, P_{fgh},L}: \Tilde{H}^3_{\NNc}(Y_{t,F},2,P_{fgh})_L\twoheadrightarrow F\otimes_{\Q_p}L\xrightarrow{[F:\Q_p]^{-1}\Tr_{F/\Q_p}\otimes 1}L    
    \end{equation}
    is defined since $P_{fgh}:=P_f*(1-p^{d(k-2-r)}T)*P_g*P_{h'}$ satisfies the conditions $P_{fgh}(1)\neq 0\neq P_{fgh}(1/q)$ (by the assumption that $(f,g,h)$ is $(F,1-T)$-convenient). 
     \end{itemize}
\begin{proof}
Since we are assuming that $\D_\pst(V(f,g,h))$ is an $(F,1-T)$-convenient quotient of $\D_\pst(H^3_\et(Y^3_{t,\bar{\Q}_p},\Hs_{[\mathbf{r}]}(r+2))_L)$, 
the compatibility between cup products in the spectral sequences of type \eqref{etaletosyntomicspectralsequencecoeff} (cf. Remark \ref{NNsynwithcoeffrmk}, (ii)') (or better their $F$-linearized variants) implies that
\begin{equation}
\label{syntomicZt}
    \mathscr{I}_p(f,g,h)=\Tr_{Y_{t,F}^3,\NN,P_{fgh}}\big(\delta_{t,*}(\DET_\mathbf{r}^{\NN})\cup_{[\mathbf{r}]}\omega^{\NN}_{fgh}\big)\,.
\end{equation}

The notation is as follows:
\begin{itemize}
    \item [(i)] The trace $\Tr_{Y^3_{t,F},\syn,P_{fgh},L}:\Tilde{H}^7_{\NNc}(Y_{t,F}^3,4,P_{fgh})_L\twoheadrightarrow L$ is again defined as in \eqref{tracefgh} above.
    \item [(ii)] The class $\omega^{\NN}_{fgh}$ is the image of the element
    \[
    (\eta_{f'}^{\varphi=a_p})\otimes(\omega_g\otimes t_{l-1})\otimes (\omega_{h'}\otimes t_{m-1})\in \Tilde{H}^0_{\st,P_{fgh}}\big(F,\D_\pst(H^3_{\et,c}(Y^3_{t,\bar{\Q}_p},\Hs_{[\mathbf{r}]}(2r+4-k))_L)\big)
    \]
    under the isomorphism
    \[
    \Tilde{H}^3_{\NNc}(Y_{t,F}^3,\Hs_{[\mathbf{r}]},2r+4-k,P_{fgh})_L\cong \Tilde{H}^0_{\st, P_{fgh}}\big(F,\D_\pst(H^3_{\et,c}(Y^3_{t,\bar{\Q}_p},\Hs_{[\mathbf{r}]}(2r+4-k))_L)\big)
    \]
    (note that we have implicitly used Künneth decomposition). This isomorphism is once more afforded by the spectral sequence \eqref{etaletosyntomicspectralsequencecoeff}, observing that $H^i_{\et,c}(Y^3_{t,\bar{\Q}_p},-)=0$ for $i\notin\{3,4,5,6\}$, since $Y^3_{t,F}$ is an affine threefold.
    \item [(iii)] The cup product $\cup_{[\mathbf{r}]}$ is induced by the pairing 
    \[
    \Hs_{[\mathbf{r}]}[k]\otimes_{\OO_{Y^3_{t,F}}}\Hs_{[\mathbf{r}]}[2r+4-k]\to\OO_{Y^3_{t,F}}[4]\,.
    \]
\end{itemize}

The careful reader will note that we have performed a suitable twist on both sides of the pairing (which does not affect the equality \eqref{syntomicZt}).

It is easy to see that
\[
\omega^{\NN}_{fgh}=p_1^*(\eta^{\NN}_{f'})\cup p_2^*(\omega^{\NN}_{g,l-1})\cup p_3^*(\omega^{\NN}_{h',m-1}),
\]
where $p_i:Y^3\to Y$ denotes the $i$-th projection for $i\in\{1,2,3\}$.

Applying the projection formula \eqref{syntomicprojformulacoeff}, it follows that
\begin{equation}
\label{syntomicXt}
    \mathscr{I}_p(f,g,h) = \Tr_{Y_{t,F},\NN,P_{fgh}}\Big(\DET_\mathbf{r}^{\NN}\cup_\mathbf{r}\big(\eta^{\NN}_{f'}\cup\omega^{\NN}_{g,l-1}\cup\omega^{\NN}_{h',m-1}\big)\Big)\,,
\end{equation}
where the cup product $\cup_\mathbf{r}$ arises from the pairing $\Hs_\mathbf{r}[k-2]\otimes_{\OO_{Y_{t,F}}}\Hs_\mathbf{r}[2r+4-k]\to \OO_{X_{t,F}}[2]$, obtained combining the pairings $\Hs_i\otimes_{\OO_{Y_{t,F}}}\Hs_i\to\OO_{X_{t,F}}[-i]$ for $i\in\{r_1,r_2,r_3\}$.

Formula \eqref{syntomictraceformula} in the statement of the proposition is essentially \eqref{syntomicXt} after suitably rearranging the pairings.
\end{proof}
\end{prop}

\medskip
Once obtained the formula in Proposition \ref{syntomictrace}, we move the computation to the rigid analytic settings, where the objects involved can be made more explicit.

\begin{nota}
We introduce the notation $\langle -,-\rangle_{Y_{t,F},\NN,P_{fgh}}$ to denote the pairing
\[
\Tilde{H}^1_{\NNc}(Y_{t,F},\Hs_{k-2},0,P_f)_L\otimes_{F\otimes_{\Q_p}L} \Tilde{H}^2_{\NN}(Y_{t,F},\Hs_{k-2},k,P_{gh'})_L\to F\otimes_{\Q_p}L
\]
(where $P_{gh'}:=(1-q^{k-2-r}T)*P_g*P_{h'}$) appearing in the statement of Proposition \ref{syntomictrace}. The corresponding pairing for log-rigid finite polynomial cohomology (cf. Remark \ref{logrigidcristallinecompatibility} and Assumption \ref{conditional}) will be denoted $\langle-,-\rangle_{\Tilde{X}_t,\lrigs,P_{fgh}}$.
\end{nota}

\begin{cor}
\label{rigidsyntomiccorollary}
Assume that $(f,g,h)$ is $(F,1-T)$ convenient and that $(k,l,m)=(2,2,2)$. Then we have:
\[
\mathscr{I}_p(f,g,h)=\Big\langle\eta^{\lrig}_{f'}\;,\;\omega^{\lrig}_{g,1}\cup\omega^{\lrig}_{h',1}\Big\rangle_{\Tilde{X}_t,\lrigs,P_{fgh}}.
\]
Under Assumption \ref{conditional} (in the case $(k,l,m)\neq (2,2,2)$), we have that
\[
\mathscr{I}_p(f,g,h)=\Big\langle\eta^{\lrig}_{f'}\;,\;\Upsilon(\DET^{\lrig}_{\mathbf{r}}\cup \omega^{\lrig}_{g,l-1}\cup\omega^{\lrig}_{h',m-1})\Big\rangle_{\Tilde{X}_t,\lrigs,P_{fgh}}.
\]
Here the superscript $(-)^{\lrig}$ denotes the image of $(-)$ under the (conjectural in higher weight) comparison isomorphism between Nekov\'{a}\v{r}--Nizio\l{} and log-rigid finite polynomial cohomology.
\begin{proof}
This is a direct consequence of the discussion in Remark \ref{logrigidcristallinecompatibility} and Remark \ref{logrigidtilded}(for the case $(k,l,m)=(2,2,2)$). See again Assumption \ref{conditional} for the discussion concerning the higher weight case.
\end{proof}
\end{cor}

\begin{ass}
\label{F1-Tconvenient}
From now on, we assume that the triple $(f,g,h)$ satisfying assumptions \ref{balselfdual} is always $(F,1-T)$-convenient and that the eigenforms $g$ and $h$ are $p$-depleted.
\end{ass}

\begin{rmk}
\label{dualUponeta}
From the explicit description of the class $\eta_{f'}^{\varphi=a_p}$ (cf. Definitions \ref{defidifferentials} and \ref{etafrobp}) and the fact that the operators $U_p$ and $U_p'$ are adjoint to each other under the Poincaré pairing in de Rham cohomology, one can easily check that 
\[
U_p'(\eta_{f'}^{\varphi=a_p})=a_p(f)\chi_f(p)^{-1}\cdot \eta_{f'}^{\varphi=a_p}\,.
\]
In particular, $\eta_{f'}^{\varphi=a_p}$ is fixed by the anti-ordinary projector $e'_\ord=\lim_{n\to +\infty} (U'_p)^{n!}$.
\end{rmk}

\begin{lemma}
\label{restrtoinfty}
Let $V_0$ denote the union of all the irreducible components of the special fiber of $\Tilde{X}_t$ except the one denoted $\Ig_\infty(p^t)$, with open complement $Z_0$. Let $\VV$ and $\ZZ$ denote the respective tubes inside $X_{t,F}^{an}$. Then the image of $\eta_{f'}^{\varphi=a_p}$ under the restriction
\[
H^1_{\dR}(X_{t,F}\langle-D_F\rangle,\Hs_{k-2})_L\cong H^1_{\dR}(\XX_{t,F}\langle-\DD\rangle,\Hs_{k-2})_L\to H^1_{\dR}(\VV\langle -\DD\rangle,\Hs_{k-2})_L
\]
is trivial. In particular, looking at the long exact sequence (cf. \cite[Proposition 8.2.18]{LS2007})
\[
\to H^i_{\dR,c}(\ZZ\langle -\DD\rangle,\Hs_{k-2})_L\to H^i_{\dR}(\XX_{t,F}\langle -D_F\rangle,\Hs_{k-2})_L\to H^i_{\dR}(\VV\langle -\DD\rangle,\Hs_{k-2})_L\xrightarrow{+1},
\]
we deduce that there exists a class $\eta_{f',\infty}^{\varphi=a_p}\in H^1_{\dR,c}(\ZZ\langle -\DD\rangle,\Hs_{k-2})_L$ whose image inside $H^1_{\dR,c}(\XX_{t,F}\langle -\DD\rangle,\Hs_{k-2})_L$ is precisely $\eta_{f'}^{\varphi=a_p}$.
\begin{proof}
Since the form $f'^{\circ}$ is $p$-primitive (cf. Remark \ref{f'desc} and note our assumptions rule out case (ii) there), the isomorphism \eqref{decompinfzero} (which, recall, is induced by restriction to $\W_\infty(p^t)$ and $\W_0(p^t)$), gives an isomorphism
\[
H^1_\dR(X_{t,F},\Hs_{k-2})_L[f'^\circ]\cong H^1_\dR(\W_\infty(p^t),\Hs_{k-2})_L^*[f'^\circ]\oplus H^1_\dR(\W_0(p^t),\Hs_{k-2})_L^*[f'^\circ]\,.
\]
In particular, since clearly $\eta_{f'}^{\varphi=a_p}\in H^1_\dR(X_{t,F},\Hs_{k-2})_L[f'^\circ]$, we are left to prove that $\eta_{f'}^{\varphi=a_p}|_{\W_0(p^t)}=0$.

Since $\langle p\,;1\rangle\circ U'_p=w_{p^t}\circ U_p\circ w_{p^t}^{-1}$ and $w_{p^t}^2=\langle p^t; -1\rangle$, one sees immediately that Remark \ref{dualUponeta} implies that:
\[
U_p\circ w_{p^t}(\eta^{\varphi=a_p}_{f'})=a_p(f)\cdot w_{p^t}(\eta^{\varphi=a_p}_{f'})\,.
\]
Restricting our attention to $\W_0(p^t)$ and keeping track of the various definitions (cf. Remark \ref{phiinftyrmk}), one obtains:
\begin{align*}
\Phi_0(\eta^{\varphi=a_p}_{f}|_{\W_0(p^t)_F}) & =w_{p^t}^{-1}\circ\Phi_\infty\circ w_{p^t}(\eta^{\varphi=a_p}_{f'}|_{\W_0(p^t)_F})\\
& =\big(\frac{p^{k-1}\cdot\chi_f(p)}{a_p(f)}\big)^d\cdot\eta^{\varphi=a_p}_{f'}|_{\W_0(p^t)_F}\\
& = \beta_p(f)^d\cdot \eta^{\varphi=a_p}_{f'}|_{\W_0(p^t)_F}\,.
\end{align*}

\medskip
Since by construction $\Phi(\eta^{\varphi=a_p}_{f'})=(a_p(f))^d\cdot \eta^{\varphi=a_p}_{f'}$ (and $|a_p(f)|_p\neq|\beta_p(f)|_p$ since $f$ is $p$-ordinary and $p$-primitive of weight $k\geq 2$), we deduce that $\eta^{\varphi=a_p}_{f'}|_{\W_0(p^t)_F}=0$.
\end{proof}
\end{lemma}

Observe that $Z_0$ is nothing but the smooth locus of $\Ig_\infty(p^t)$ inside $\Tilde{X}_t$ and $\ZZ$ is precisely what we previously denoted $\A_\infty(p^t)$ (or rather the associated dagger space).

\medskip
Let $Z:=X_{t,F}\cup Z_0\subset\Tilde{X}_t$. We can apply Proposition \ref{extrestrprop} to our computation and obtain the following crucial equality:
\begin{equation}
\mathscr{I}_p(f,g,h)=\langle\eta^{\rig}_{f',\infty,k-2}\,, \res_Z(\Upsilon(\DET_\mathbf{r}^{\lrig}\cup\omega^{\lrig}_{g,l-1}\cup\omega^{\lrig}_{h',m-1})\rangle_{Z,\rigs,P_{fgh}}   
\end{equation}
Here $\eta^{\rig}_{f',\infty}$ is any lift of $\eta_{f',\infty}^{\varphi=a_p}$ under the surjection denoted $\mathbf{p}_{P_f}$ (cf. equation \eqref{ppsurjection}), so that its image under the extension-by-0 map (cf. Proposition \ref{extrestrprop}) is the class denoted $\eta^{\lrig}_{f'}$ above. Clearly $\langle - ,-\rangle_{Z,\rigs,P_{fgh}}$ denotes the corresponding pairing in rigid finite polynomial cohomology.

\medskip
Finally we can apply the projection formula \eqref{crucialprojectionformula} to reduce to a pairing in the rigid cohomology of $Z_0$:
\begin{equation}
\label{rigidpairing}
\mathscr{I}_p(f,g,h)=\frac{1}{P_{fgh}(p^{-d})}\cdot \langle\eta^{\varphi=a_p}_{f',\infty}\,, \Xi_{g,h'}\rangle_{Z_0,\rig} 
\end{equation}
where $\Xi_{g,h'}\in H^1_{\rig}(Z_0\langle D_0\rangle,\Hs_{k-2})_L$ is defined as follows. We have a short exact sequence (cf. \eqref{sesrigidsyntomic})
\begin{equation}
\label{sesH2rigsint}
\begin{tikzcd}[column sep=small, row sep=small]
&0\arrow[r]&\frac{H^{1}_{\rig}(Z_0\langle D_0\rangle,\Hs_{k-2})_L}{P(q^{-r}\Phi)\circ\mathrm{sp}(\Fil^k H^{1}_{\dR}(Z_F\langle D_F\rangle,\Hs_{k-2})_L)}\arrow[r, "\mathbf{i}_{P_{gh'}}"] &\Tilde{H}^2_{\rigs}(Z\langle D\rangle,\Hs_{k-2},k,P_{gh'})_L\arrow[dl, controls={+(-1,-1) and +(1,1)}]\\
&&\Fil^k H^{2}_{\dR}(Z_F\langle D_F\rangle,\Hs_{k-2})_L^{P_{gh'}(q^{-r}\Phi)\circ\mathrm{sp}=0}\arrow[r] &0.   
\end{tikzcd}
\end{equation}
Since $H^2_{\dR}(Z_F\langle D_F\rangle,\Hs_{k-2})\cong H^2_{\dR}(Y_{t,F},\Hs_{k-2})=0$ (as $Y_{t,F}$ is an affine curve), we can lift the class $\res_Z(\Upsilon(\DET_\mathbf{r}^{\lrig}\cup\omega^{\lrig}_{g,l-1}\cup\omega^{\lrig}_{h',m-1}))$ to an element $\Xi_{g,h'}\in H^1_{\rig}(Z_0\langle D_0\rangle,\Hs_{k-2})_L$. Such lift is unique, since $\Fil^k H^{1}_{\dR}(Z_F\langle D_F\rangle,\Hs_{k-2})_L=0$.

\medskip
The denominator appearing in the RHS of equation \ref{rigidpairing} comes from the definition of the trace in finite polynomial cohomology (which has to be compatible with respect to change of polynomial).

\subsection{End of the proof}
\label{endofproof}
This section concludes the proof of Theorem \ref{explicirreclawconj}. We are left to give a description of the overconvergent modular form which gives rise to the class $\Xi_{g,h'}$.

\medskip
In order to better justify our computations, we need to introduce a modified version of rigid finite polynomial cohomology, namely Gros fp-cohomology for the $\OO_F$-scheme $Z$. We refer to \cite[Section 8.2]{LZ2024} for some more details about this construction.

Here we just recall that one can define Gros fp-cohomology as mapping fibre
\[
\widetilde{R\Gamma}_{\rigfp}(Z\langle D\rangle,\Hs_s,r,P):=\big[\widetilde{\Fil}^r R\Gamma_{\dR}(\ZZ\langle\DD\rangle,\Hs_s)_F\xrightarrow{P(q^{-r}\Phi)}R\Gamma_{\rig}(Z_0\langle D_0\rangle,\Hs_s)_F\big]
\]
Here $\widetilde{\Fil}^r R\Gamma_{\dR}(\ZZ\langle\DD\rangle,\Hs_s)_F$ denotes the so-called truncated rigid cohomology, defined as the cohomology of the subcomplex $\Fil^{r-\bullet}\Hs_r\otimes\Omega^\bullet_{\ZZ/F}\langle\DD\rangle$ (note that the cohomology groups of such complex are not necessarily finite dimensional).

\medskip
A completely analogous description affords the compactly supported version of Gros fp-cohomology and one has cup products (we omit the coefficients for reasons of space)
\[
\widetilde{R\Gamma}_{\rigfp}(Z\langle D\rangle,r_1,P)\times \widetilde{R\Gamma}_{\rigfp,c}(Z\langle -D\rangle,r_2,Q)\to\widetilde{R\Gamma}_{\rigfp,c}(Z\langle -D\rangle,r_1+r_2,P*Q).
\]
defined using the formalism of cup product in finite polynomial cohomology (cf. \cite[p. 37]{Bes2012}).

\medskip
As observed in \cite[Remark 8.2.6]{LZ2024}, the specialization map gives a map
\[
\gamma^*: \widetilde{R\Gamma}_{\rigs}(Z\langle D\rangle,r,P)\to \widetilde{R\Gamma}_{\rigfp}(Z\langle D\rangle,\Hs_s,r,P)
\]
and the cospecialization map gives a map
\[
\gamma_*: \widetilde{R\Gamma}_{\rigfp,c}(Z\langle -D\rangle,\Hs_s,r,P)\to\widetilde{R\Gamma}_{\rigs,c}(Z\langle -D\rangle,r,P).
\]
The cup product satisfies the adjunction formula
\[
\gamma_*(\gamma^*(x)\cup y)=x\cup\gamma_*(y).
\]

\begin{rmk}
We observe that we can safely move our computation to a pairing in Gros-fp cohomology due to the following facts.
\begin{enumerate}[(i)]
    \item From Lemma \ref{restrtoinfty} and the description of Gros-fp cohomology, we see that $\eta^{\varphi=a_p}_{f',\infty}$ can be lifted to a class $\eta^{\rigfp}_{f',\infty}\in\tilde{H}^1_{\rigfp,c}(Z\langle -D\rangle,\Hs_{k-2},0,P_f)$ such that $\gamma_*(\eta^{\rigfp}_{f',\infty})=\eta^{\rig}_{f',\infty}$.
    \item Using the analogue of the short exact sequence \eqref{sesH2rigsint} for Gros-fp cohomology, we obtain identifications
    \[
    H^1_{\rig}(Z_0\langle D_0\rangle,\Hs_{k-2})_L\cong \Tilde{H}^2_{\rigs}(Z\langle D\rangle,\Hs_{k-2},k,P_{gh'})_L\cong \Tilde{H}^2_{\rigfp}(Z\langle D\rangle,\Hs_{k-2},k,P_{gh'})_L
    \]
    so that the class that we denoted $\Xi_{g,h'}$ can be computed via Gros-fp cohomology, without changing the outcome.
\end{enumerate}
\end{rmk}

\begin{rmk}
Recall that the class $\eta^{\varphi=a_p}_{f',\infty}$ is fixed by the antiordinary projector $e'_{\ord}$. Hence we obtain that only the ordinary projection of $\Xi_{g,h'}$ contributes to the computation, i.e., we have
\begin{equation}
\mathscr{I}_p(f,g,h)=\frac{1}{P_{fgh}(p^{-d})}\cdot \langle\eta^{\varphi=a_p}_{f',\infty}\,, e_{\ord}(\Xi_{g,h'})\rangle_{Z_0,\rig}
\end{equation}
\end{rmk}

\medskip
The advantage of working with Gros fp-cohomology is that classes in
\[
\Tilde{H}^i_{\rigfp}(Z\langle D\rangle,\Hs_{s},r,P)
\]
can be explicitly described as pairs $(\omega,F)$ with $\omega\in\Fil^{r-i}\Hs_s\otimes\Omega^i\langle\DD\rangle(\ZZ)$ closed and $F\in\Hs_s\otimes\Omega^{i-1}(\ZZ)$ such that $\nabla F=P(\Phi)\omega$ (up to suitable identifications).

\begin{rmk}
\label{Ukillspdepleted}
Recall that $U_p\circ\varphi_\infty=\varphi_\infty\circ U_p=\langle p\,;1\rangle p^{r+1}$ on $H^1_{\rig}(Z_0\langle D_0\rangle,\Hs_r)_L$ for every $r\in\Z_{\geq 0}$ and that the operator $U_p$ kills sections of the form $\xi(q)\cdot \omega_\can^{r-j}\otimes\eta^j\otimes \frac{dq}{q}$, where $\xi(q)$ is a $p$-depleted $q$-expansion (i.e., a $q$-expansion where the terms corresponding to integers divible by $p$ vanish). It follows that such sections are exact (hence trivial in $H^1_{\rig}(Z_0\langle D_0\rangle,\Hs_r)_L$). We are assuming that the $q$-expansions of $g$ and $h'$ are $p$-depleted, so that the classes $\omega_g$ and $\omega_{h'}$ become trivial when restricted to $H^1_{\rig}(Z_0\langle D_0\rangle,\Hs_r)_L$.
\end{rmk}

Overconvergent primitives of the restriction of $\omega_{\xi}\otimes t_{\nu-1}$ to $\ZZ$ for $\xi\in\{g,h'\}$ are given by sections $F^\flat_\xi$ of $\Hs_{\nu-2}$ (where $\nu=l$ if $\xi=g$ and $\nu=m$ if $\xi=h'$) defined on a strict neighbourhood of $\A_\infty(p^t)$ and can be described explicity around the cusp $\infty$ (cf. \cite[Proposition 3.24]{BDP2013}) as:
\[
F^\flat_{\xi,\infty}=\sum_{j=0}^{\nu-2} (-1)^j j!\binom{\nu-2}{j}\cdot d^{-1-j}\xi\cdot\omega_\can^{\nu-2-j}\otimes \eta^j_\can\otimes t_{\nu-1}.
\]
For $\xi\in\{g,h'\}$, the class of $\gamma^*\res_Z(\omega^{\lrig}_{\xi,\nu-1})$ can be then described by the pair $(\omega_{\xi,\nu-1},P_\xi(\Phi_\infty)F^{\flat}_{\xi})$, where we view $\omega_{\xi,\nu-1}$ as a section of $\Fil^{\nu-1}\Hs_{\nu-2}\otimes\Omega^1\langle\DD\rangle$.

Note that here we are implicitly using the fact that the Frobenius is horizontal with respect to the Gauss--Manin connection (as prescribed by definition of overconvergent $F$-isocrystal, cf. for instance \cite[Definition 3.13]{BDP2013}) to exhibit $P_\xi(\Phi_\infty)F_{\xi}^{\flat}$ as a (local) primitive of $P_\xi(\Phi_\infty)(\omega_{\xi,\nu-1})$.

\medskip
The formalism of cup-product in finite polynomial cohomology (cf. \cite[p. 37]{Bes2012}) shows that the cup product of the pairs $(\omega_{g,l-1},P_g(\Phi_\infty)F^{\flat}_{g})$ and $(\omega_{h',m-1},P_{h'}(\Phi_\infty)F^{\flat}_{h'})$ can be described by a pair $(\omega_{g,l-1}\otimes\omega_{h',m-1},\Xi'_{g,h'})$, where
\[
\Xi'_{g,h'}=a(\Phi_\infty,\Phi_\infty)\big(P_g(\Phi_\infty)(F^\flat_g)\otimes\omega_{h',m-1}\big)-b(\Phi_\infty,\Phi_\infty)\big(\omega_{g,l-1}\otimes P_{h'}(\Phi_\infty)(F_{h'}^\flat)\big).
\]
The polynomials $a(X,Y)$ and $b(X,Y)$ are chosen so that
\[
P_g*P_{h'}(XY)=a(X,Y)P_g(X)+b(X,Y)P_{h'}(Y).
\]
Using the fact that $F_g^\flat\otimes\omega_{h',m-1}+\omega_{g,l-1}\otimes F_{h'}^\flat=\nabla(F_g^\flat\otimes F_{h'}^\flat)$ is a $\nabla$-closed form, we see that $\Xi'_{g,h'}$ is equal to $P_{g}*P_{h'}(\Phi_\infty\otimes\Phi_\infty)(F_g^\flat\otimes\omega_{h',m-1})$ modulo $\nabla$-closed forms.

\medskip
It follows again from the formalism of cup-product in finite polynomial cohomology and from the above considerations that $\Xi_{g,h'}$ can be represented by the differential form 
\begin{equation}
\label{almostthere}
\Upsilon \big(\DET^{\infty}_{\mathbf{r}}\cup (P_{g}*P_{h'}(\Phi_\infty\otimes\Phi_\infty)(F_g^\flat\otimes\omega_{h',m-1}))\big)=P_{gh'}(\Phi_\infty)\big(\Upsilon(\DET^{\infty}_{\mathbf{r}}\cup (F_g^\flat\otimes\omega_{h',m-1}))\big),   
\end{equation}
where $\DET^\infty_{\mathbf{r}}\in H^0_{\rig}(Z_0\langle D_0\rangle, \Hs_{\mathbf{r}})$ is simply the restriction of the de Rham incarnation of $\DET^{\et}_{\mathbf{r}}$ to the rigid cohomology of $Z_0$.

\medskip
The equality in \eqref{almostthere} follows observing that $\Phi$ acts on $\DET^{\infty}_{\mathbf{r}}$ as multiplication by $q^{r+2-k}$.

\medskip
Recall the notation $\breve{f}:=\lambda_{M_1}(f)^{-1}\cdot w_{M_1}(f)$, where $\lambda_{M_1}(f)$ is the pseudo-eigenvalue for the action of $w_{M_1}$ on $f$. Since the linear functional defining $\eta^{\varphi=a_p}_{f'}$ factors through the $\breve{f}$-isotypic projection and, since $\Phi_\infty$ commutes with $U_p$, we obtain that
\begin{align*}
\mathscr{I}_p(f,g,h)&=\frac{1}{P_{fgh}(p^{-d})}\cdot \langle\eta^{\varphi=a_p}_{f',\infty}\,, e_{\ord}(\Xi_{g,h'})\rangle_{Z_0,\rig}=\\
&=\frac{1}{P_{fgh}(p^{-d})}\cdot \Big\langle\eta^{\varphi=a_p}_{f',\infty}\,, P_{gh'}(\Phi_\infty)e_{\ord}\big(\Upsilon(\DET^{\infty}_{\mathbf{r}}\cup F_g^\flat\otimes\omega_{h',m-1})\big)\Big\rangle_{Z_0,\rig}=\\
&=\Big\langle\eta^{\varphi=a_p}_{f',\infty}\,,e_{\ord}\big(\Upsilon(\DET^{\infty}_{\mathbf{r}}\cup F_g^\flat\otimes\omega_{h',m-1})\big)\Big\rangle_{Z_0,\rig}
\end{align*}
The last equality uses the (easily checked) fact that $\Phi_\infty(\omega_{\breve{f}}\otimes t_k)=(a_p(f)\cdot p)^{-d}\cdot\omega_{\breve{f}}\otimes t_k$, which implies that $P_{gh'}(\Phi_\infty)(\omega_{\breve{f}}\otimes t_k)=P_{fgh'}(p^{-d})$.
\medskip

Since we assume that $f$ is $p$-primitive and $p$-ordinary (in particular of small slope), Coleman's results in \cite{Col1997} show that the map obtained as the composition
\[
S_k(\Gamma_1(Mp^t),L)\to \frac{H^0(W,\omega^{\nu})_L}{d^{k-1}\big(H^0(W,\omega^{2-k})_L\big)}\cong H^1_{\rig}(Z_0\langle D_0\rangle,\Hs_{k-2})_L,    
\]
(where $W$ is any affinoid strict neighbourhood of $\A_\infty$ in $\W_\infty$ and the isomorphism is a reinterpretation of \eqref{Colemaniso}, since rigid cohomology can be computed as the de Rham cohomology of wide open subspaces) becomes an isomorphism after passing to the $\breve{f}$-isotypic quotients (which also entails applying the ordinary projector $e_\ord$). Hence, in order to compute the pairing 
\[
\Big\langle\eta^{\varphi=a_p}_{f',\infty}\,, e_{\ord}\big(\Upsilon(\DET^{\infty}_{\mathbf{r}}\otimes F_g^\flat\otimes\omega_{h',m-1})\big)\Big\rangle_{Z_0,\rig}
\]
we are just left to exhibit any (cuspidal) modular form which maps to the class
\[
e_{\ord}\big(\Upsilon(\DET^{\infty}_{\mathbf{r}}\cup( F_g^\flat\otimes\omega_{h',m-1}))\big)
\]
under the above composition.

\medskip
This can be done essentially in the same way as in \cite[pp. 1023-1024]{BSV2020a}, obtaining that, as classes in $H^1_{\rig}(Z_0\langle D_0\rangle,\Hs_{k-2})_L$, we have
\[
e_{\ord}\big(\Upsilon(\DET^{\infty}_{\mathbf{r}}\cup( F_g^\flat\otimes\omega_{h',m-1}))\big)=(-1)^{k-2}(r-k+2)!\cdot \omega_{\Xi^\ord(g,h')}\otimes t_{k},
\]
where $\Xi^\ord(g,h'):=e_{\ord}(d^{(k-l-m)/2}g\times h')$.
\medskip
Combining everything we can conclude that
\begin{equation}
\begin{aligned}
\mathscr{I}_p(f,g,h)&=(-1)^{k-2}(r-k+2)!\cdot\eta^{\varphi=a_p}_{f'}(e_{\ord}(d^{(k-l-m)/2}g\times h')=\\
&=(-1)^{k-2}(r-k+2)!\cdot a_1\big(e_{\breve{f}}(\Tr_{Mp^t/M_1p^t}(e_{\ord}(d^{(k-l-m)/2}g\times h'))\big).
\end{aligned}
\end{equation}

\medskip
This completes the proof of Theorem \ref{explicirreclawconj}.

\subsection{On the proof in the case \texorpdfstring{$(k,l,m)\neq(2,2,2)$}{(k,l,m)neq(2,2,2)}}
\label{thegeneralcasesubsection}
During the revision process for this article, the preprint \cite{ABSV2025} has appeared. The authors introduce a syntomic (and finite polynomial) formalism that allows automorphic coefficients and is engineered to obtain proofs of $p$-adic explicit reciprocity laws in cases where the relevant Galois cohomology classes are (potentially) crystalline, but come from the cohomology of a Shimura variety which only admits a strictly semistable model over the ring of integers of a large enough finite extension $F$ of $\Q_p$. We will refer to this cohomology theory as the \emph{ABSV syntomic (or finite polynomial) cohomology}.

\medskip
Here we briefly indicate how one could use the results of loc. cit. to obtain a full proof of Theorem \ref{explicirreclawconj} in the case of a balanced triple of weights $(k,l,m)\neq(2,2,2)$. Since the resulting proof would be formally the same as the one obtained above, we decided to sketch how the results of loc. cit. can be applied in our case, leaving the details to the interested reader. 

\medskip
The ABSV syntomic cohomology satisfies the analogues of properties (ii), (iii), (iv) of Remark \ref{NNsyntomicfeaturesrmk} (cf. Sections from 13 to 16 of loc. cit.), taking into account the two following caveats.

\medskip
The first issue is that Gysin morphisms (i.e., what concerns point (iv)) are only constructed in the case of divisors. In our case, the diagonal embedding $d_t: X_t\hookrightarrow X_t^3$ has codimension $2$, but we can clearly view it as a successive embedding $X_t\hookrightarrow X_t^2\hookrightarrow X_t^3$, so we can still obtain Gysin maps in the ABSV formalism simply applying their construction twice (taking a bit of care in obtaining suitable semistable models for $X_t,X_t^2,X_t^3$ over the ring of integers of a large enough finite extension of $\Q_p$).

\medskip
The second problem concerns trace maps, i.e., point (iii). Note that the formalism of loc. cit. employs the semilinear (non-linearized) version of Frobenius. In order to dispose of the necessary syntomic trace map for $X_t^3$ (resp. $X_t$) one is led to check that the pair $(P_{fgh}((T/p^4)^d),4)$ (resp. $(P_{fgh}((T/p^2)^d),2)$ is \emph{admissible} in the sense of \cite[Definition 13]{ABSV2025}, where $P_{fgh}$ is as in the statement of Proposition \ref{syntomictrace} and $d=[F:\Q_p]$ as usual. These two conditions are equivalent to each other and translate to the fact that $P'_{fgh}(T):=P_{fgh}(T^d)$ cannot have roots of the form $\zeta$ or $p^{-1}\zeta$ with $\zeta$ a root of unity, which is already implied by our assumption that the triple $(f,g,h)$ is $(F,1-T)$-convenient.

\begin{rmk}
In fact, in the first version of \cite{ABSV2025}, the authors had defined a trace map for their syntomic (and finite polynomial) formalism only after imposing slightly more restrictive assumptions on the possible roots of the polynomials employed in the machinery. In particular this would not have allowed roots of the form $p^{-1/2}\zeta$ for $P'_{fgh}$. Luckily this restriction has been removed in the current version of loc. cit., which now only impose standard assumptions for the existence of syntomic traces.
\end{rmk}

\begin{lemma}
\label{admissiblelemma}
The polynomial $P'_{fgh}$ does not have roots of the form $\zeta$ or $p^{-1}\zeta$ with $\zeta$ a root of unity in the following cases:
\begin{enumerate}[(i)]
\item in the setting of Assumption \ref{cuspasslocp}, if the weight $k$ of $f$ satisfies $k> 2$;
\item in the setting described by Remark \ref{princseriescase}.
\end{enumerate}
\begin{proof}
The same calculations in the proof of Proposition \ref{efgequal} show that, in the setting of Assumption \ref{cuspasslocp}, one can write $P(T):=P'_{fgh}(T)$ in the form $P(T)=\prod_{i\in I}(1-\gamma_i^{-1}T)$, with $\ord_p(\gamma_i)=-k/2$, whence the trace-admissibility if $k> 2$.

On the other hand, under the condition (PS) of Remark \ref{princseriescase}, we can choose $P(T)$ as above with
\[
\gamma_i^{-1}=\frac {p^{r+2}\zeta}{a_p(f)\alpha_{j_{1},g}\alpha_{j_{2},h}},
\]
where again $\zeta$ is a root of unity and, for $j_1$ ranging in $\{1,2\}$, we let $\alpha_{1,g},\alpha_{2,g}$ denote the two \emph{eigenvalues} of Frobenius acting on $\D_{\st,F}(V^*(g))$ (the situation is completely analogous to the one described in the proof of Lemma \ref{decompeta} in the potentially crystalline case). The same notation applies to $h$. In this case, we can rely on the Ramanujan--Petersson conjecture (now a theorem), which implies that $|\alpha_{j_{1},g}|_\infty=p^{(l-1)/2}$ (where $|\cdot|_\infty$ denotes the Euclidean absolute value), and similarly $|\alpha_{j_2,h}|_\infty=p^{(m-1)/2}$ and $|a_p(f)|_\infty=p^{(k-1)/2}$. Again, a rapid calculation shows that $|\gamma_i|_\infty=p^{-1/2}$ and the lemma follows.
\end{proof}
\end{lemma}

A surrogate of property (i) of Remark \ref{NNsyntomicfeaturesrmk} is spelled out in \cite[Section 12]{ABSV2025}. In particular, the diagram in the statement of \cite[Proposition 12.1]{ABSV2025} shows that one can lift classes in the Bloch--Kato subspace $H^1_{e}(F,-)$ (cf. Definition \ref{BKsubspacesdef}) to ABSV syntomic cohomology. Thus, one can formally obtain the analogue of Proposition \ref{syntomictrace} working with ABSV syntomic cohomology.

\medskip
Finally, the discussion in \cite[Section 16.4 and 16.5]{ABSV2025} describes the behaviour of the theory concerning the restriction to rigid syntomic (or finite polynomial cohomology) of the smooth locus of a connected component of the special fiber of the given semistable model, obtaining the analogue of Proposition \ref{extrestrprop}. Note that the rigid syntomic (and finite polynomial) cohomology with coefficients in the smooth case appearing in \cite[Definition 6]{ABSV2025} coincides with the one introduced in Section \ref{rigidsyntomicsmoothsection}. Hence, one can virtually proceed in the proof to obtain the formula in equation \eqref{rigidpairing} and then conclude exactly as in Section \ref{endofproof}.

\printbibliography

\end{document}